\DocumentMetadata{
  uncompress,
  lang        = en-US,
  pdfversion  = 2.0,
  pdfstandard = ua-2,
  pdfstandard = a-4f, %or a-4
  testphase   = latest
}

\documentclass[11pt, letterpaper%,reqno
]{amsart}

\usepackage{zref}

\usepackage{pdfcomment}

\usepackage{graphicx}
\usepackage{color}

\usepackage{amsmath}

\usepackage{hyperref}
\hypersetup{
    colorlinks=false,
    linkcolor=blue,
    filecolor=magenta,   
    citecolor=red,   
    urlcolor=cyan,
    pdftitle={Overleaf Example},
    pdfpagemode=FullScreen,
    }

\theoremstyle{plain}
\newtheorem{thm}{Theorem}
\newtheorem*{thm*}{Theorem}
\newtheorem{prop}{Proposition}[section]
\newtheorem*{prop*}{Proposition}
\newtheorem{lem}{Lemma}[section]
\newtheorem*{lem*}{Lemma}

\theoremstyle{definition}
\newtheorem{dfn}{Definition}[section]
\newtheorem{conj}{Conjecture}
\newtheorem*{conj*}{Conjecture}

\theoremstyle{remark}
\newtheorem{rmk}{Remark}

\newcommand{\Vol}{\mathrm{Vol}}
\newcommand{\tr}{\operatorname{trace}}
\newcommand{\Li}{\mathrm{Li}_2}

\newcommand{\ADO}{\operatorname{ADO}}

\newcommand{\re}{\operatorname{Re}}
\newcommand{\im}{\operatorname{Im}}

\newcounter{appendix}

\allowdisplaybreaks

\newcommand{\CS}{\operatorname{CS}}
\title[Complex. tetra., fund. groups, and volume conj.]
{Complexified tetrahedrons, fundamental groups, and volume conjecture for double twist knots}

\author{Jun Murakami}

\email{murakami@waseda.jp}

\address{Department of Mathematics, Faculty of  Science and Engineering, Waseda University, 3-4-1 Ohkubo, Shinjuku-ku, Tokyo, 169-9555, JAPAN}

\thanks{This work was supported by JSPS KAKENHI Grant Numbers JP20H01803, JP20K20881.}
\date{\today}



\begin{document}
\begin{abstract}
In this paper, the volume conjecture for  double twist knots are proved.  
The main tool is the complexified tetrahedron and the associated $\mathrm{SL}(2, \mathbb{C})$ representation of the fundamental group.  
A complexified tetrahedron is a version of a truncated or a doubly truncated tetrahedron whose edge lengths and the dihedral angles are complexified.  
The colored Jones polynomial is expressed in terms of the quantum $6j$ symbol, which corresponds to the complexified tetrahedron.  
\end{abstract}

\maketitle

\section*{Introduction}
Let $K$ be a  framed knot or link in $S^3$.  
In the following, knots include links unless otherwise described.   
Let $V_N(K)$ be the colored Jones polynomial of $K$ which corresponds to the $N+1$ dimensional irreducible representation of the quantum group  $\mathcal{U}_q(sl_2)$. 
Here $V_N(K)$  is normalized to satisfy  $V_N(\phi)=1$ and $V_N(\bigcirc) = -(q^{N+1}- q^{-N-1})/(q-q^{-1})$ for the trivial   knot.  
The parameter $q$ corresponds to  $A^2$ where $A$ is the parameter used for defining the Kauffman bracket polynomial. 
Let %$N$ be an odd integer greater than or equal to three and 
$J_{N-1}(K) = V_{N-1}(K)/V_{N-1}(\bigcirc)$ where $q = \exp(\pi i/N)$ for $i = \sqrt{-1}$, which is the $2N$-th root of unity.   
The volume conjecture predicts that certain limit of the colored Jones polynomial gives Gromov's simplicial volume $||S^3\setminus K||$ of the complement of $K$ as follows.  
%
%\end{document}
%
\begin{conj}[Volume conjecture \cite{MM}]
For a knot or link $K$, 
\[
2\, \pi \lim_{N\to \infty}\frac{\log\left|J_{N-1}(K)\right|}{N}
=
v_3 \, ||S^3\setminus K||
\]
where $v_3$ is the hyperbolic volume of the regular ideal tetrahedron.
\end{conj}
If $S^3\setminus K$ admits the hyperbolic structure, in other words, $K$ is a hyperbolic knot or link, then $v_3 \, ||S^3\setminus K|| = \Vol(K)$ where $\Vol(K)$ is the hyperbolic volume of $S^3\setminus K$.  
For hyperbolic knots and links, the following is also conjectured. 
\begin{conj}[Complexified volume conjecture \cite{MMOTY}]
For a hyperbolic knot or link $K$, 
\[
2\, \pi\, \lim_{N\to \infty}\frac{\log J_{N-1}(K)}{N}
=
\Vol(K) +\mathrm{CS}(K)\, \sqrt{-1}
\quad (\mathrm{mod}\  \pi^2 \sqrt{-1}\,\mathbb{Z})
\]
where $\mathrm{CS}(K)$ is $2 \pi^2$ times the Chern-Simons invariant  $cs(S^3 \setminus K)$, where $cs(S^3 \setminus K)$ is a real number between $0$ and $1/2$.  
\end{conj}
For prime hyperbolic knots, this conjecture is proved for  knots with less than or equal to seven crossings.  
Here, we prove Conjecture 1 for all hyperbolic double twist knots.
\begin{figure}[htb]
\begin{small}
\[
\begin{matrix}
\begin{matrix}
\includegraphics[scale=0.8, alt = {Borromean rings}]{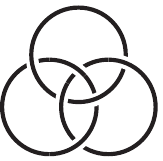}
\end{matrix}
&
\begin{matrix}
\includegraphics[scale=0.8, alt = {Borromean rings}]{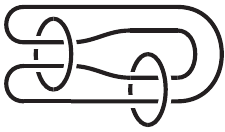}
\end{matrix}
& 
\begin{matrix}
\includegraphics[scale=0.8, alt = {Borromean rings with a half twist}]{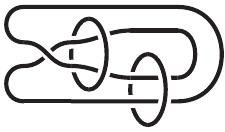}
\end{matrix}
\\[5pt]
\text{Borromean rings $B$} &
 \text{Another expression of $B$} 
& 
\text{$B_{1}$: first variation of $B$}
\\
\begin{matrix}
\includegraphics[scale=0.8, alt = {Borromean rings with a full twist}]{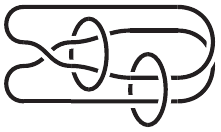}
\end{matrix}
& 
\begin{matrix}
\includegraphics[scale=0.75, alt={Whitehead link}]{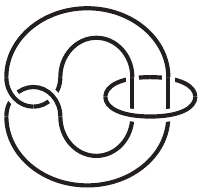}
\end{matrix}
&  
\begin{matrix}
\includegraphics[scale=0.75, alt={twisted Whitehead link}]{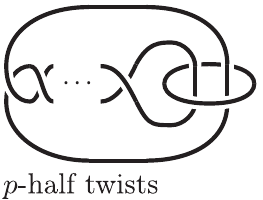}
\end{matrix}
\\[5pt]
\!\!\!\! \text{$B_{1,1}$: second variation of $B$} 
&  \text{Whitehead link $W$} 
&  \text{twisted Whitehead link $W_p$}
\end{matrix}
\]
\[
\begin{matrix}
\begin{matrix}
\includegraphics[scale=0.8, alt={twist knot}]{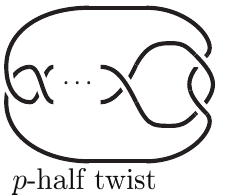}
\end{matrix}
& \qquad& 
\begin{matrix}
\includegraphics[scale=0.8, alt={double twist knot}]{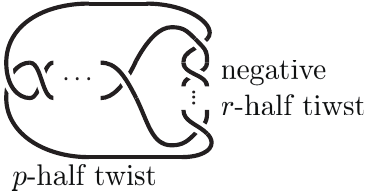}
\end{matrix}
\\
\text{twist knot $T_p$}
& &
\text{double twist knot $D_{p,r}$}
\end{matrix}
\]
\end{small}
\caption{Knots and links handled in this paper.  }
\label{fig:borromean}
\end{figure}
\begin{thm}
Let $K$ be a hyperbolic double twist knot. 
Then the following holds. 
\[
2\, \pi\, \lim_{N\to \infty}\frac{J_{N-1}(K)}{N}
=
\Vol(K) +\mathrm{CS}(K)\, \sqrt{-1}
\quad (\mathrm{mod}\  \pi^2 \sqrt{-1}\,\mathbb{Z}).
\]
\end{thm}
\begin{rmk}
The volume conjecture for hyperbolic knot with crossing number less than or equal to 7 are proved in \cite{O1}, \cite{OY} and \cite{O2}.  
That for the twist knot $T_p$ for $p\geq 6$ is proved in \cite{CZ}.  
\end{rmk}
\begin{rmk}
Combining the result in \cite{OT}, we get the following for any hyperbolic double twist simble component knot $K$.  
\[
J_{N-1}(K) \underset{N\to\infty}{\sim}e^{N \zeta(K)}\omega(K) \Big(1+ O(\frac{1}{N})\Big),
\]
where 
$\zeta(K) = \sqrt{-1}\left(\mathrm{Vol}(S^3\setminus K) +\sqrt{-1}\,\mathrm{CS}(S^3\setminus K)\right)$, 
$\omega(K) = \pm \frac{\tau(K)}{2 \sqrt{-1}}$ and $\tau(K)$ is the twisted Reidemeister torsion of $K$ associated with the geometric $\mathrm{SL}(2, \mathbb{C})$ representation of $\pi_1(S^3 \setminus K)$.  
\end{rmk}
The main tool is the complexified tetrahedron.  
Volume formulas of hyperbolic tetrahedrons are given in \cite{CK}, \cite{MY} in terms of dihedral angles at edges and in \cite{MU} in terms of edge lengths. 
The formulas in \cite{MY} and \cite{MU} are based on the volume conjecture for the quantum $6j$ symbol, and they are analytic functions on the parameters.  
These formulas are also work for truncated tetrahedra as shown in \cite{U} and for doubly truncated tetrahedra as in \cite{KM}.  
\begin{figure}[htb]
\[
\begin{matrix}
\begin{matrix}
\includegraphics[scale=0.8]{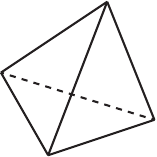}
\end{matrix}
& 
\begin{matrix}
\includegraphics[scale=0.8]{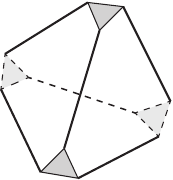}
\end{matrix}
& 
\begin{matrix}
\includegraphics[scale=0.9]{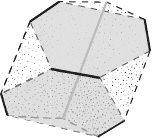}
\end{matrix}
\\
\text{usual tetrahedron}
& 
\text{truncated tetrahedron} 
& 
\text{doubly truncated tetrahedron}
\end{matrix}
\]
\caption{A usual tetrahedron, a truncated tetrahedron and a doubly truncated tetrahedron.
Any face which truncate a vertex is perpendicular to the original three faces of the tetrahedron which are adjacent to the vertex.}
\end{figure}
Here the length considered to be a real number and the angle considered to be a pure imaginary number.  
Now let us complexify these numbers of a truncated tetrahedron and a doubly truncated tetrahedron as in Figure \ref{fig:complexify}.  
\begin{figure}[htb]
\begin{align*}
\text{truncated edge:}&\quad
\begin{matrix}
\includegraphics[scale=0.9]{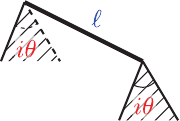}
\end{matrix}
\longrightarrow\ \ 
\begin{matrix}
\includegraphics[scale=0.9]{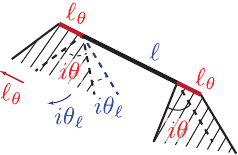}
\end{matrix}
\\
\text{doubly truncated edge:}&\quad
\begin{matrix}
\includegraphics[scale=0.9]{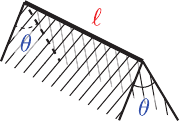}
\end{matrix}
\longrightarrow\ \ 
\begin{matrix}
\includegraphics[scale=0.9]{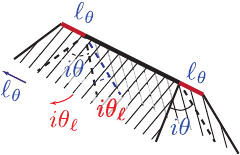}
\end{matrix}
\end{align*}
\caption{Complexify the angle and the length at an edge.  The parameter $i\theta$ is modified to $\ell_\theta + i\theta$ and the parameter $\ell$ is modified to $\ell + i \theta_\ell$.  
The shaded faces correspond to the truncated faces.}
\label{fig:complexify}
\end{figure}
The adjacent edges at an endpoint of the edge are rotated by $\theta_\ell$ and then faces (no more planner) glued at the edge are shifted by $\ell_\theta$.  
Then the angle parameter $i\theta$ is generalized to $\ell_\theta + i\theta$ and the length parameter $\ell$ is generalized to $\ell + i \theta_\ell$.  
%The volume formula in \cite{MY} still satisfies Schlafli's differential equation between the angles and the lengths.  
After such deformation, the faces of the truncated tetrahedron is no more planner.  
But, by assigning elements of $\mathrm{PSL}(2, \mathbb{C})$ to the edges of the truncated tetrahedron, we can define the volume of such generalized tetrahedron by considering the fundamental domain of the action by such group elements. 
For the complexified tetrahedron, the Schl\"afli differential formula is generalized to the differential equation satisfied by the Neumann-Zagier function.  
\par
The difficulty for proving the volume conjecture is to check the condition for applying the saddle point method to the potential function obtained from  $J_{N-1}(K)$, which is a sum of terms  consisting of a product of quantum factorials and some powers of $q$.  
For the large $N$ case, this sum can be reformulated into an integral of the potential function, where the integral range corresponds to the range of the sum of $J_{N-1}(K)$.    
To apply the saddle point method, this integral range must be wide enough to surround the saddle point, which is very hard to show for complicated knots.  
Here we express the colored Jones polynomial of each double twist knot by using the quantum $6j$ symbol, and is expressed by parameters  assigned to the edges of the tetrahedral graph.   
In this expression, the range for sum is rather simple and it is not hard to see that we can apply the saddle point method.  
The edge parameters correspond to the saddle point are complex numbers, and the corresponding geometric object is the complexified tetrahedron.
The complement of the double twist knot is decomposed into a union of two copies of such complexified tetrahedron, while the expression of the colored Jones polynomial obtained from the quantum $R$ matrix corresponds to an ideal tetrahedral decomposition of the  complement.  
\par
The new idea of this article is to introduce the complexified tetrahedron which is constructed from the geometric $\mathrm{SL}(2, \mathbb{C})$ representation  of  the fundamental group of the complement.  
We also use the ADO invariant \cite{ADO}, \cite{CM} to investigate $J_{N-1}(K)$.  
For the  techniques to apply the saddle point method and the Poisson sum formula, we just follow the arguments developed in papers \cite{O1,O2, OY} to prove the volume conjecture for hyperbolic knots with small crossing numbers.   
\par
The paper organized as follows.  
In Section 1, we explain the volume conjecture for Borromean rings. 
In this case, volume conjecture is already solved, and here we reconsider it by using the expression of the  colored Jones invariant in terms of the quantum $6j$ symbol.  
In Section 2, we tread twisted Whitehead links.  
The volume conjecture is also solved for this case, but here reprove it by using the complexified tetrahedron and the quantum $6j$ symbol.  
For the twisted Whitehead link case, we  use a complexified tetrahedron which appears as a deformation of the regular ideal octahedron.  
In Section 3, the double twist knots are investigated.  
The method to prove the volume conjecture is same as for the twisted Whitehead links explained in Section 2.  
\par
Some notions and detailed computations are given in appendices.  
Especially, in Appendix B, colored Jones invariants are reformulated by using the ADO invariants.  
This part is the most complicated part of this paper, but the reformulation of the colored Jones polynomial explained here simplifies the rest of the proof of the volume conjecture.  
%\par
%At the end, you may convince that 
%this method also  work for all the  hyperbolic knots obtained from the fully augmented links including the two-bridge knots.  
\medskip
\par\noindent
{\bf Acknowledgment.}
The author was strongly encouraged to pursuit this research when I attended  ``Winter School on
Low-dimensional
Topology
and Related Topics'' at IBS-CGP in Pohang, Korea in December 2023,     
and he would like to thank all the participants of the school, especially Jessica Purcell, Seonhwa Kim, Thiago de Paiva Souza, and the organizer Anderson Vera.  
He also would like to thank Anh Tran for giving me a lot of information about  $SL(2, \mathbb{C})$ representations of  the double twist knots and two-bridge knots.
\section{Borromean rings}
The volume conjecture for the Borromean rings is easily proved elementary, but here we recall the proof to see its corresponds to the $\mathrm{PSL}(2, \mathbb{C})$ representation of the fundamental group of the complement.  
Throughout this paper, $N$ is assumed to be an odd positive integer greater than or equal to $3$.  
\subsection{Representation matrix}
Let $B$ be the Borromean rings in Figure \ref{fig:borromean}.  
We first construct the parabolic $\mathrm{SL}(2, \mathbb{C})$ representation $\rho$ of $\pi_1(S^3\setminus B)$ which corresponds to the hyperbolic structure of $S^3 \setminus B$. 
In other words, let $\Gamma$ be the image of $\rho$, then $S^3 \setminus B$ is isomorphic to $\mathbb{H}^3/\Gamma$, where $\mathbb{H}^3$ is the hyperbolic three space.  
Here we use the upper half model, so $\mathbb{H}^3$ is identified with $\mathbb{C} \times \mathbb{R}_{>0}$ and $\partial \mathbb{H}^3$ is identified with $\mathbb{C}$.  
To assign elements of $\pi_1(S^3 \setminus B)$, we draw $B$ as in Figure \ref{fig:borromeangenerator} and assign the elements $g_1$, $\cdots$, $g_4$, $h_1$, $h_2$ as in the figure.   
\begin{figure}[htb]
\[
\includegraphics[scale=0.9]{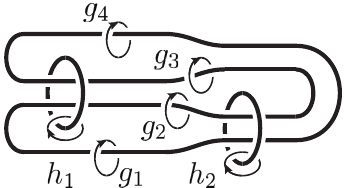}
\]
\caption{Elements of $\pi_1(S^3\setminus B)$.  
The base point is located above the plane.  
%Take $g_1$, $h_1$, $h_2$ as generators, then $g_3$, $g_4$  are expressed by these generators.
}
\label{fig:borromeangenerator}
\end{figure}
Then the relations of $\pi_1(S^3\setminus B)$ are given as follows.  
\begin{multline}
\pi_1(S^3\setminus B)
=
\\
\left<g_1, g_2, g_3, g_4, h_1, h_2
 \mid
g_2 = h_1\, g_1^{-1}\, h_1^{-1}, \ 
g_3= h_1 \, g_4^{-1}\, h_1^{-1}, \ 
g_3^{-1} = h_2 \, g_2 \, h_2^{-1}
\right>.
\label{eq:relationB}
\end{multline}
Now let us consider parabolic representation $\rho$.  
Let $g_{12} = g_1 \, g_2$ and % $u$, $u^{-1}$ be its eigenvalues, and let 
$g_{23} = g_2\, g_3$. %and $v$, $v^{-1}$ be its eigenvalues.  
Since $h_1$ parabolic, $g_{23}$ is also parabolic.  
The eigenvalues of $\rho(g_i)$ are all $1$ or all $-1$.  
Recall that any parabolic matrix of $\mathrm{SL}(2, \mathbb{C})$ with eigenvalue $\mp 1$ is represented as $\pm\begin{pmatrix}
-1+\alpha\beta & \beta^2 \\ -\alpha^2 & -1-\alpha\beta
\end{pmatrix}$
for some complex numbers $\alpha$ and $\beta$.  
So, up to the conjugation, we can assign 
\[
\rho(g_1) = 
\varepsilon\!\begin{pmatrix}-1 & x \\ 0 & -1\end{pmatrix}, 
\ \ 
\rho(g_2) = 
\varepsilon\!\begin{pmatrix}-1+y & y \\ -y & -1-y\end{pmatrix}, 
\ \ 
\rho(g_3) = 
\varepsilon\!\begin{pmatrix}
-1 & 0 \\ -z & -1
\end{pmatrix},
\]
where $\varepsilon=\pm1$.
Since $h_1$ and $g_{23}$ are commutative and $\rho(h_1)$ is parabolic, $\rho(g_{23})$ must be parabolic with eigenvalue $1$ or $-1$.  
Hence  $\tr \rho(g_{23})$ must be $2$ or $-2$.  
On the other hand, $\tr \rho(g_{23}) = 2 - yz$, so if $\tr \rho(g_{23}) = 2$, then $y$ or $z$ is zero, which contradict  the assumption that the representation $\rho$ is non-abelian.  
Therefore, $yz= 4$ and $\tr \rho(g_{23})=-2$, which means that the eigenvalue of $\rho(g_{23})$ is $-1$.  
By this reason, we assume that the eigenvalues of $\rho(g_i)$ and $\rho(h_j)$ are all $-1$.   
Similar argument for $g_{12}$ and $h_2$ implies that $xy=4$.  
We also have $g_4 = (g_1\,g_2\,g_3)^{-1}$ and $\tr(g_4)=-2$ since we assume that the eigenvalue of $\rho(g_4)$ is $-1$.  
This means that $x y + x z + y z + x y z =x^2 + 4x + 16= 0$ and we get the following two solutions.
\begin{align}
x &= -2 + 2i, \quad y = -1 - i, \quad z = -2 + 2 i,
\label{eq:bsol1}
\\
x &= -2 - 2i, \quad y = -1+ i, \quad z = -2 - 2 i.
\label{eq:bsol2}
\end{align}
Choose the solution \eqref{eq:bsol2} for $\rho$ and let $p_i$, $p_{ij}$ be the fixed points of $\rho(g_i)$, $\rho(g_{ij})$ in $\mathbb{C}$.  
Then
\[
p_1 = \infty, \ \ p_2 = -1, \ \ p_3 = 0, \ \ p_4 = 1, 
\ \ p_{12} = i, \ \ p_{23} = -i,
\]
and these points are the vertices of a regular ideal octahedron $O_1$ in $\mathbb{H}^3$.  
The action of $\rho(g_1)$ to $\mathbb{C}$ is the translation by $2 + 2i$.  
Let $O_2$ be another regular ideal octahedron with vertices
\[
q_1 = \infty, \ \ q_2 = i, \ \ q_3 = 1+i, \ \ q_4 = 2+i, 
\ \ q_{12} = 1+2i, \ \ q_{23} = 1,
\]
then $O_1 \cup O_2$ is the fundamental domain of the action of $\mathrm{Im}\,\rho$.  
\begin{figure}[htb]
\[
\includegraphics[scale=0.8, alt={Two regular octahedron}]{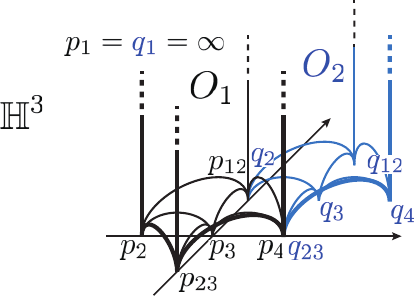}
\]
\caption{Regular ideal octahedra $O_1$, $O_2$ in the upper half space whose union is the fundamental domain of the action of $\rho(\pi_1(S^3\setminus B))$.  }
\label{fig:octahedron12}
\end{figure}
\subsection{Volume conjecture}
The colored Jones polynomial $J_{N-1}(B)$ is computed in Appendix A, and given by \eqref{eq:BJ},  
 that is the following.  
\begin{multline}
J_{N-1}(B) = 
\\
N^2\sum_{0 \leq k,l \leq N-1}
\sum_{\max(k,l) \leq s \leq \min(k+l, N-1)}\!\!
\dfrac{  \{s\}!^2}
{\{s-k\}!^2 \, \{s-l\}!^2 \, \{k+l-s\}!^2 },
\label{eq:borromeanJ}
\end{multline}
where 
\[
\{k\} = q^k - q^{-k}, \quad
\{k\}! = \{k\}\{k-1\}\cdots\{1\}\ \  \text{for $k \geq 1$} \ \ \text{and}\ \  \{0\}!=1.    
\]
\par
Now we prove the volume conjecture for $B$ by using  \eqref{eq:borromeanJ}.
The idea of proof is the same as that in \cite[Section 3.2]{MYo1}.    
The terms in the sum are all positive and the limit
$
2 \pi \lim_{N\to \infty} \frac{\log  J_{N-1}(B)}{N}
$
is given by the largest term in the sum.  
The maximal is attained at $k = l=\lfloor \frac{N-1}{2}\rfloor$ and $s = \lfloor \frac{3(N-1)}{4}\rfloor$ and the maximal value is
$2\left(-\Lambda(\frac{3\pi}{4})+7\Lambda(\frac{\pi}{4})\right) = 
16 \Lambda(\frac{\pi}{4}) = 7.3277...$, which is equal to the twice of the volume of the regular ideal octahedron and is equal to the volume of $S^3\setminus B$.  
Here $\Lambda(x)$ is the Lobachevsky function given by $\Lambda(x) = - \int_0^x \log | 2 \sin t |\, dt$.  
\subsection{Regular ideal octahedron}
The regular ideal octahedron 
\begin{figure}[htb]
\[
\begin{matrix}
\begin{matrix}
\includegraphics[scale=0.8]{tetrahedron2}
\end{matrix}
& & 
\begin{matrix}
\includegraphics[scale=0.98]{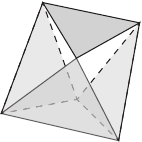}
\end{matrix}
\\
\text{truncated tetrahedron}
& & 
\text{regular ideal octahedron}
\end{matrix}
\]
\caption{Recular ideal octahedron is an extremal truncated tetrahedron.  
The faces have checkerboard coloring, and the white faces corresponds to the faces of the original tetrahedron, and the vertices corresponds to the edges of the original tetrahedron. }
\label{fig:octahedron}
\end{figure}
can be thought as an extremal case of the truncated tetrahedron whose dihedral angles at edges are all zero.  
In this case, the length of edges are also zero.  
\subsection{Variations of the Borromean rings}
Here we investigate the variations $B_{1}$ and $B_{1,1}$ of the Borromean rings $B$ in Figure \ref{fig:borromean}.  
%
%\subsection{Representations}
Let $g_1$, $\cdots$, $g_4$, $h_1$, $h_2$ be the elements of $\pi_1(S^3 \setminus B_1)$ given in Figure \ref{fig:borromeangenerator1}.  
\begin{figure}[htb]
\[
B_1 : \ 
\begin{matrix}
\includegraphics[scale=0.8, alt={generators}]{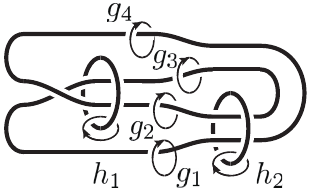}
\end{matrix}\ ,
\qquad
B_{1,1} : \ 
\begin{matrix}
\includegraphics[scale=0.8, alt={generators}]{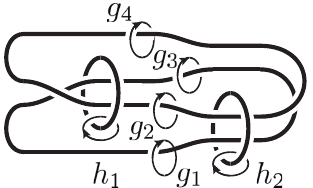}
\end{matrix}
\]
\caption{The elements $g_1$, $g_2$, $g_3$, $g_4$, $h_1$, $h_2$ in $\pi_1(S^3 \setminus B_1)$ and $\pi_1(S^3 \setminus B_{1,1})$.}
\label{fig:borromeangenerator1}
\end{figure}
The fundamental groups $\pi_1(S^3\setminus B_1)$ and $\pi_1(S^3\setminus B_{1,1})$ are presented by
\begin{align}
&\pi_1(S^3\setminus B_1)
=
\left<g_1, g_2, g_3, g_4, h_1, h_2\right.
 \mid
\label{eq:relationB1}
\\
& \qquad\qquad
\left.g_2 = h_1g_4^{-1}h_1^{-1}, \ 
g_3= h_1 g_4 g_1g_4^{-1}h_1^{-1}, \ 
g_3^{-1} = h_2 g_2 h_2^{-1}
\right>, 
\notag
\\
&\pi_1(S^3\setminus B_{1,1})
=
\left<g_1, g_2, g_3, g_4, h_1, h_2\right.
 \mid
\label{eq:relationB11}
\\
&\qquad
\left.g_2 = h_1g_4^{-1}h_1^{-1}, \ 
g_3= h_1 g_4 g_1g_4^{-1}h_1^{-1}, \ 
g_3^{-1} = h_2 g_2^{-1}g_1 g_2 h_2^{-1}
\right>.
\notag
\end{align}
Let $\rho'$, $\rho'_1$, $\rho'_{1,1}$ be the geometric $\mathrm{SL}(2, \mathbb{C})$ representations of $\pi_1(S^3\setminus B)$, $\pi_1(S^3\setminus B_1)$, $\pi_1(S^3\setminus B_{1, 1})$ respectively so that 
\begin{align*}
\rho'(g_{23}) &= \rho'_1(g_{23}) = \rho'_{1,1}(g_{23})
= 
\begin{pmatrix}
-1 & x \\ 0 & -1
\end{pmatrix}, 
\\
\rho'(g_{1})& = \rho'_1(g_{1}) = \rho'_{1,1}(g_{1})
= 
\begin{pmatrix}
-1+y & y \\ -y & -1-y
\end{pmatrix}, 
\\
\rho'(g_{2})& = \rho'_1(g_{2}) = \rho'_{1,1}(g_{2})
= 
\begin{pmatrix}
-1 & 0 \\ -z & -1
\end{pmatrix}. 
\end{align*}
Let $\tau$ be one of $\rho'$, $\rho'_1$, $\rho'_{1,1}$, then $\tau$ must satisufy
 $\tr \tau(g_2)=\tr \tau(g_3) =\tr\tau(g_4)= -2$, 
 and we get 
\[
x = 2i, \ \  y = -2i\ \  z = 2i, \quad \text{or}\quad
x = -2i, \ \  y = 2i, \ \  z = -2i
\]
for all $\rho'$, $\rho'_1$, $\rho'_{1,1}$.  
By choosing the first solution for $x$, $y$, $z$, 
the representation matrices for $h_1$ are 
given as follows from the relations \eqref{eq:relationB}, \eqref{eq:relationB1}, \eqref{eq:relationB11}.  
\begin{align*}
\rho'(h_1) = 
\begin{pmatrix}
-1 & -1 \\ 0 & -1
\end{pmatrix}, 
\quad
\rho'_1(h_1) = 
\rho'_{1,1}(h_1) = 
\begin{pmatrix}
-1 & -1+i \\ 0 & -1
\end{pmatrix}.  
\end{align*}
The fixed points $r_1$, $r_2$, $r_3$, $r_4$, $r_{23}$, $r_{12}$ of $g_1$, $g_2$, $g_3$, $g_4$, $g_{23}$, $g_{12}$ are given as follows.  
\[
r_1 = -1, \ r_2 = 0, \ r_3 = i, \ r_4 = -1+i,
\ r_{23} = \infty, \ r_{12} = \frac{-1+i}{2}.
\]
Let $O_1$ be the regular ideal octahedron with vertices $r_1$, $\cdots$, $r_4$, $r_{23}$, $r_{12}$, and let $O_2$ be that with vertices $s_1 = -1+i$, $s_2 = i$, $s_3= 2i$, $s_4 = -1+2i$, $s_{23} = \infty$, 
then $O_1 \cup O_2$ is the fundamental domain for the actions of $\pi_1(S^3\setminus B)$, $\pi_1(S^3\setminus B_1)$, $\pi_1(S^3\setminus B_{1, 1})$.  
By doing such computation for $h_2$ instead of $h_1$, we get the similar result.  
Here we get the same fundamental domain for the actons of the fundamental groups  $\pi_1(S^3\setminus B)$, $\pi_1(S^3 \setminus B_1)$, $\pi_1(S^3 \setminus B_{1,1})$.
However, the actions of $h_1$ and $h_2$ are different as in Figure \ref{fig:haction} while the actions of $g_1$, $\cdots$, $g_4$ coincide respectively for $B$, $B_1$ and $B_{1,1}$.    
\begin{figure}[htb]
\[
\begin{matrix}
\text{The action of $h_1$} : &\quad
\begin{matrix}
\includegraphics[scale=0.8]{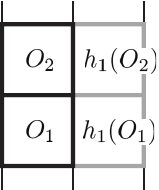}
\end{matrix},
 \ &\ 
\begin{matrix}
\includegraphics[scale=0.8]{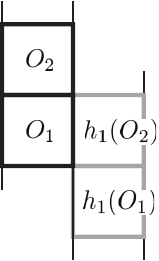}
\end{matrix},
 \ &\ 
\begin{matrix}
\includegraphics[scale=0.8]{cusph1B1}
\end{matrix},
\\
\text{The action of $h_2$} : &\quad
\begin{matrix}
\includegraphics[scale=0.8]{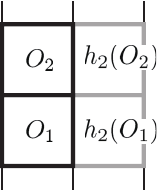}
\end{matrix},
 \ &\ 
\begin{matrix}
\includegraphics[scale=0.8]{cusph2B}
\end{matrix},
 \ &\ 
\begin{matrix}
\includegraphics[scale=0.8]{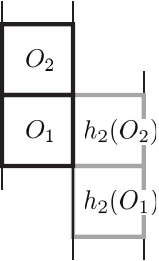}
\end{matrix},
\\
& B & B_1 & B_{1,1}
\end{matrix}
\]
\caption{The actions of $h_1$ and $h_2$ on the cusp diagrams of the  components corresponding to $h_1$ and $h_2$ respectively. }
\label{fig:haction}
\end{figure}
\par
The colored Jones polynomials of $B_1$ and $B_{1,1}$ are given by \eqref{eq:JonesB1}, \eqref{eq:JonesB11} as follows.  
\[
J_{N-1}(B_1) = 
N^2
q^{\frac{(N-1)^2}{4}}
\sum_{k, l = 0}^{N-1}
\sum_{m = \max(k,l)}^{\min(k+l, N-1)}\!\!
\dfrac{ q^{\left(k-\frac{N-1}{2}\right)^2} \{m\}!^2}
{\{m-k\}!^2 \, \{m-l\}!^2 \, \{k+l-m\}!^2 },
\]
\[
J_{N-1}(B_{1,1}) = 
N^2
q^{\frac{(N-1)^2}{2}}
\sum_{k, l =0}^{ N-1}
\sum_{m=\max(k,l)}^{ \min(k+l, N-1)}
\!\!\!\!\!
\dfrac{ q^{\left(k-\frac{N-1}{2}\right)^2} 
q^{\left(l-\frac{N-1}{2}\right)^2} \{m\}!^2}
{\{m-k\}!^2 \, \{m-l\}!^2 \, \{k+l-m\}!^2 }.
\]
These formulas have the phase factors $q^{\left(k-\frac{N-1}{2}\right)^2}$ and $q^{\left(k-\frac{N-1}{2}\right)^2}q^{\left(l-\frac{N-1}{2}\right)^2}$ added to $J_{N-1}(B)$,  and no more real numbers.  
For $J_{N-1}(B_1)$,  the term with $k =(N-1)/2$, $l = (N-1)/2$, $s = \lfloor3(N-1)/4\rfloor$ have the maximal modulus among the terms in the sums and the oscillation at $k =(N-1)/2$ is stopped, so we have
\[
\lim_{N\to\infty}\frac{2\pi}{N}\log J_{N-1}(B_1) 
=
\lim_{N\to\infty}\frac{2\pi}{N}\left(\log |J_{N-1}(B) | +\frac{ \pi N}{4} \sqrt{-1}\right).  
\]
Similarly, for $J_{N-1}(B_{1,1})$,  the term with $k =(N-1)/2$, $l = (N-1)/2$, $s = \lfloor3(N-1)/4\rfloor$ have the maximal modulus among the terms in the sums and the oscillation around $k =(N-1)/2$ and $l = (N-1)/2$ is very small, so we have
\[
\lim_{N\to\infty}\frac{2\pi}{N}\log J_{N-1}(B_{1,1}) 
=
\lim_{N\to\infty}\frac{2\pi}{N}\log J_{N-1}(B) .  
\]
The above rough argument can be replaced by a rigorous argument by using the Poisson sum formula and the saddle point method as in \cite{O1}.  
%In the next section, we use such methods for the twisted Whitehead links and the double twist knots.  
The hyperbolic volumes of the complements of $B_1$ and $B_{1,1}$ are equal to that of the complement of $B$ since these complements are both split into two regular ideal tetrahedrons.  
The Chern-Simons invariants are obtained from the imaginary of the complex volume by SnapPy,    and we get
$\CS(B) = \CS(B_{1,1})=0$, $\CS(B_1) = \pi^2/2$.  
Therefore, we have
\begin{thm}
Conjecture 2 holds for $B_1$ and $B_{1,1}$.  
\end{thm}
\section{Twisted Whitehead links}
In this section, we introduce the  complexified tetrahedron, which is a deformation of the regular hyperbolic octahedron, by using $\mathrm{SL}(2, \mathbb{C})$ representation of 
$\pi_1(S^3\setminus W_p)$ for the twisted Whitehead link $W_p$ with $|p| \geq 2$.  
Then we prove Conjecture 1 for $W_p$ with the help of the complexified tetrahedron, which is a deformation of the regular ideal octahedron used in the previous section.  
Conjecture 1 is already proved by \cite{Z}, and here we explain how the hyperbolic volume relates to the complexified tetrahedron, especially to its complexified length and angle, which corresponds to the eigenvalues of representation matrices of certain elements of $\pi_1(S^3\setminus W_p)$.   
Note that the Whitehead link $W$ is equal to $W_2$, and $W_{-2}$ is the mirror image of $W_2$,
We exclude $W_0$ and $W_{\pm1}$ since they are not hyperbolic.   
\subsection{Representation matrices}
\label{subsection:Wrep}
Assign the generators of $\pi_1(S^3\setminus W_p)$ as in Figure \ref{fig:whiteheadgenerator}.  
These generators satisfy the following relations.
\begin{equation}
\begin{aligned}
g_4 = h g_1 h^{-1}, \ \ 
g_3^{-1} = h g_2 h^{-1}, \ \ 
g_1 g_2 g_3 g_4 = 1, 
\qquad\qquad\qquad\qquad\qquad
\\
g_4^{-1} = 
(g_2 g_3)^{\frac{p}{2}}  g_3  (g_2 g_3)^{-\frac{p}{2}}, \ \ 
g_1^{-1} = 
(g_2 g_3)^{\frac{p}{2}}  
g_2 (g_2 g_3)^{-\frac{p}{2}},
\ \  \text{($p$ : even)}
\\
g_4 = h g_1 h^{-1}, \ \ 
g_3^{-1} = h g_2 h^{-1}, \ \ 
g_1 g_2 g_3 g_4 = 1, 
\qquad\qquad\qquad\qquad\qquad
\\
g_4^{-1} = 
(g_2 g_3)^{\frac{p-1}{2}}  g_2
(g_2 g_3)^{-\frac{p-1}{2}}, \ 
g_1^{-1} = 
(g_2 g_3)^{\frac{p+1}{2}} g_3 
(g_2 g_3)^{-\frac{p+1}{2}}.
\ \  \text{($p$ : odd)}
\end{aligned}
\label{eq:whiteheadrelation}
\end{equation}
\begin{figure}[htb]
\[
\begin{matrix}
\includegraphics[scale=0.8, alt={generators}]{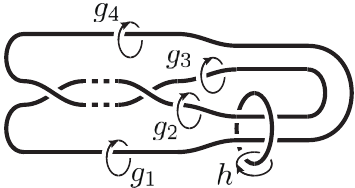}
\end{matrix}
\qquad\qquad
\begin{matrix}
\includegraphics[scale=0.8, alt={generators}]{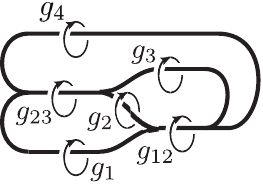}
\end{matrix}
\]
\caption{Generators of $\pi_1(S^3\setminus W_p)$ and related tetrahedral graph.  
The base point is located above the plane.  
}
\label{fig:whiteheadgenerator}
\end{figure}
Now we construct the geometric representation $\rho : \pi_1(S^3\setminus W) \to \mathrm{SL}(2, \mathbb{C})$.  
The matrices corresponding to the meridians are all parabolic.  
As in the case of the Borromean rings, we assume that the eigenvalues of $\rho(g_i)$ and $\rho(h)$ are $-1$.  
%Note that any parabolic matrix is of the form 
%$\begin{pmatrix}
%-1 + \alpha\beta & \beta^2 \\ -\alpha^2 & -1 - \alpha\beta
%\end{pmatrix}$
%for some $\alpha$, $\beta \in \mathbb{C}$.  
Let $g_{12}=g_1g_2$, $g_{23} = g_2g_3$.
For geometric representation, it is known that the matrix $\rho(g_{23})$ is diagonalizable.  
By applying conjugation, we may assume that $\rho(g_{23})$ is a diagonal matrix and an off-diagonal element of $\rho(g_1)$ is the minus of the other off-diagonal element of $\rho(g_1)$.  
Now we put
\begin{align*}
\rho(g_1) &= \begin{pmatrix}
-1 + x & x u \\ -x u^{-1} & -1-x
\end{pmatrix},\qquad
\rho(g_2) = \begin{pmatrix}
-1 + a_2b_2 & b_2^2 \\ -a_2^2 & -1-a_2b_2
\end{pmatrix},
\\
\rho(g_3) &= \begin{pmatrix}
-1 + a_3b_3 & b_3^2 \\ -a_3^2 & -1-a_3b_3
\end{pmatrix}, \quad
\rho(g_{23}) = \begin{pmatrix}
u & 0 \\ 0 & u^{-1}
\end{pmatrix}.
\end{align*}
Since $g_{12}$ commutes with the parabolic matrix $\rho(h)$ and $\rho$ is a non-abelian representation, we get  $\tr \rho(g_{12}) = -2$.  
From the relations 
\begin{align*}
&\tr \rho(g_{23}) =u+u^{-1}, \quad 
\text{$\rho(g_{23})$ is a diagonal matrix,}
\\&
\tr \rho(g_1g_2g_3) = \tr \rho(g_{12}) = -2,
\end{align*}
we get the following matrices.  
\begin{equation}
\begin{aligned}
\rho(g_1) &= 
\left(
\begin{array}{cc}
 -\frac{2 }{u+1} & \frac{u(u-1)}{u+1}
   \\
- \frac{u-1}{u(u+1)} & -\frac{2u}{u+1} \\
\end{array}
\right),
\ \ 
\rho(g_2) = 
\left(
\begin{array}{cc}
 -\frac{2u}{u+1} &\!\!\!
   \frac{u\left(\sqrt{u}-1\right)^3}
   {\left(\sqrt{u}+1\right) (u+1)} 
   \\
   -\frac{\left(\sqrt{u}+1\right)^3}{u\left(\sqrt{u}-1\right) (u+1)} & -\frac{2
   }{u+1} \\
\end{array}
\right),
\\
\rho(g_3) &= 
\left(
\begin{array}{cc}
 -\frac{2u}{u+1} &
   -\frac{\left(\sqrt{u}-1\right)^3}{\left(\sqrt{u}+1\right) (u+1)} \\
   \frac{\left(\sqrt{u}+1\right)^3} {\left(\sqrt{u}-1\right)  (u+1)} 
   & -\frac{2}{u+1} \\
\end{array}
\right)\!, \ \ 
\rho(g_4) = 
\left(
\begin{array}{cc}
 -\frac{2 }{u+1} & -\frac{u-1}{u+1} \\
 \frac{u-1}{u+1} &
   -\frac{2u}{u+1} \\
\end{array}
\right)\!.
\end{aligned}
\label{eq:whiteheadsolution1}
\end{equation}
Let $p_i$ be the fixed point of $\rho(g_i)$ for $i=1$, $2$, $3$, $4$ and $p_{12}$ be the fixed point of $g_{12}$.  
Moreover,  let $p_{23}^0$ and $p_{23}^1$ be the two fixed points of $\rho(g_{23})$. 
Since these fixed points are given by the ratios of the elements of the eigenvectors, 
we get
\begin{align*}
p_1 &= -u, \quad
p_2 = \frac{u\left(\sqrt{u}-1\right)^2
   }{\left(\sqrt{u}+1\right)^2}, \quad
p_3 = -\frac{\left(\sqrt{u}-1\right)^2}
   {\left(\sqrt{u}+1\right)^2}, \quad
p_4 = 1,\quad
\\
p_{12} &= \frac{\left(\sqrt{u}-1\right)
   \sqrt{u}}{\sqrt{u}+1}, \qquad
p_{23}^0 = 0, \qquad p_{23}^1 = \infty.
\end{align*}
By the relation \eqref{eq:whiteheadrelation}, we have
\[
{g_{23}}^{\frac{p}{2}} \cdot p_3 = p_4, \qquad
{g_{23}}^{\frac{p}{2}} \cdot p_2 = p_1,  \qquad \ \ 
\text{($p$ : even)}
\]
\[
{g_{23}}^{\frac{p-1}{2}} \cdot p_2 = p_4, \qquad
{g_{23}}^{\frac{p+1}{2}} \cdot p_3 = p_1.   \qquad 
\text{($p$ : odd)}
\]
Since $\rho(g_{23}) = \begin{pmatrix} u & 0 \\
0 & u^{-1} \end{pmatrix}$, 
$p_1 = -u p_4$, $p_2 = -u p_3$, 
we have
\[
(-u)^{p} p_3 = p_4, \qquad
(-u)^{p} p_2 = p_1.
\]
These two equations are equal to the following equation.  
\begin{equation}
-(-u)^{p} \frac{\left(\sqrt{u}-1\right)^2}
{\left(\sqrt{u}+1\right)^2} = 1.  
\label{eq:equationu}
\end{equation}
For the Whitehead link $W=W_2$, 
the above equation is
\[
(u+1) \left(u^2-2 u^{3/2}+2
   \sqrt{u}+1\right) = 0.
\]
The solutions are $u = 1.78615 \pm 2.27202 i$ and $u = -1$,   
where the first two solutions give the geometric representations.  
For generic $p$, there are many solutions for $u$ satisfying \eqref{eq:equationu}.
To find the geometric solution among these solutions, we  consider the  complexified tetrahedron and the developing map associated with this tetrahedron as in the following subsection.  
\subsection{Complexified terahedron}
Here we construct the complexified tetrahedron for a twisted Whitehead link with respect to $\rho(g_1)$, $\cdots$, $\rho(g_{23})$.  
At first, we assign the fixed points on the complex plane associated with $\rho(g_1)$, $\cdots$, $\rho(g_{23})$ as before.
\par
For the Whitehead link case with $u = 1.78615-2.27202i$, 
\begin{align*}
p_1 &= -1.786 + 2.272 i, \ \ 
p_2 = -0.2138 - 0.2720 i, \ \ 
p_3 = -0.0283 + 0.1163 i, \ \ 
\\
p_4 &= 1,\ \ \ 
p_{12} = -0.2571 + 0.5291 i, \ \ \ 
p_{23}^0 = 0, \ \ \  p_{23}^1 = \infty.
\end{align*}
Here we see that $p_3 = -u\, p_2$, $p_4 = -u\, p_1$, $p_2 = u^2\, p_1$ and $p_3 = u^2\, p_4$ as in Figure \ref{fig:uaction}.  
\begin{figure}[htb]
\[
\begin{matrix}
\includegraphics[scale=0.8, alt={action of $-u$}]{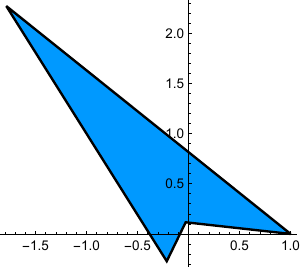}
\end{matrix}
\qquad\qquad
\begin{matrix}
\includegraphics[scale=1, alt={action of $(-u)^2$}]{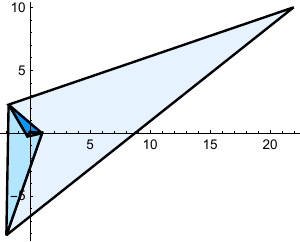}
\end{matrix}
\]
\caption{The action of $-u$ and $u^2$ to the quadrilateral $p_1p_2p_3p_4$.}
\label{fig:uaction}
\end{figure}
Let $p_i'$ be the point on the line $p_{23}^0p_{23}^1$ such that the geodesic line $p_ip_i'$ is perpendicular to $p_{23}^0p_{23}^1$.  
Then construct four geodesic triangles $F_j$ whose vertices are $p_{12}$, $p_j$, $p_{j+1}$ for $j=1$, $2$, $3$, $4$.  
Here $j+1$ means $j+1 \mod 4$.  
Now we choose two surfaces $F_5$, $F_8$ where the boundary of $F_1$ is  $p_1p_1'\cup p_1'p_2'\cup p_{2}'p_{2} \cup p_{1}p_1$ and the boundary of $F_2$ is $p_4p_4'\cup p_4'p_1' \cup p_1'p_1 \cup p_1p_4$.  
Let $\rho(g_{23})^{1/2} = 
\begin{pmatrix}
i u^{1/2} & 0 \\ 0 & -i u^{-1/2}
\end{pmatrix}$.  
%Then the action of $\left(\rho(g_{23})^{1/2}\right)^2$ to $\mathbb{H}^3$ is equal to that of  $\rho(g_{23})$.  
Let $F_6 = \rho(g_{23})\, F_8$ and $F_7 = \rho(g_{23})^{1/2}\, F_5$.  
Now we introduce the complexified tetrahedra $T$, which is the hyperbolic solid surrounded by $F_1$, $\cdots$, $F_8$.  
The surfaces $F_2$, $F_4$, $F_5$, $F_7$ correspond to the faces and the surfaces $F_1$, $F_3$, $F_6$, $F_8$ shaded in Figure \ref{fig:complexifiedtetrahedron} correspond to the vertices of the tetrahedral graph in Figure \ref{fig:whiteheadgenerator}.  
The solid $T$  is considered to be a deformation of the regular ideal octahedron.
There are many ways to take $F_5$ and $F_8$, and here we choose them so that $T \cup \rho(g_{23})^{1/2} T$ is a fundamental domain of the action of $\rho(\pi_1(S^3\setminus W))$.  
\begin{figure}[htb]
\[
\hspace{-5mm}
\begin{matrix}
\includegraphics[scale=0.8, alt={complexified tetrahedron}]{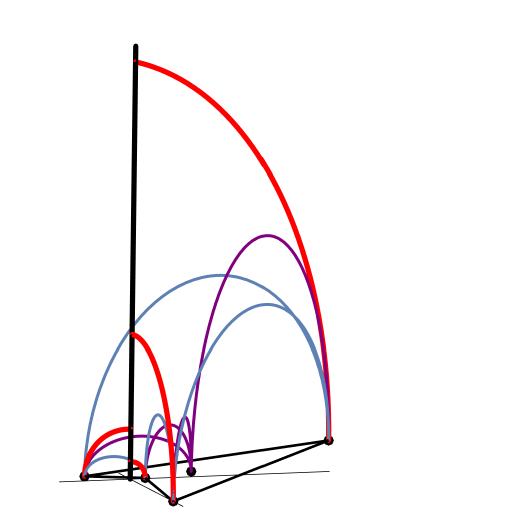}
\end{matrix}
\hspace{-65mm}
\mathbb{H}^3
\hspace{40mm}
\begin{matrix}
\includegraphics[scale=0.8]{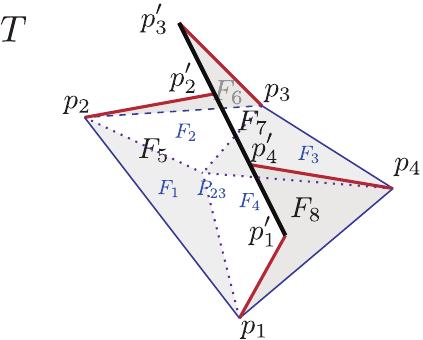}
\\
\text{The view from the edge corresponding to $g_{12}$.}
\end{matrix}
\]
\caption{Complexified tetrahedron $T$.}
\label{fig:complexifiedtetrahedron}
\end{figure}
For general $p$, there are two solutions of \eqref{eq:equationu} corresponding to the geometric representation. 
They are solutions satisfying
\[
p \arg (-u)+\arg(p_3) =2\pi i, \quad p \arg (-u)+\arg(p_3) =-2\pi i.
\]
For these solutions, $\pi_1(S^3 \setminus W_p)(T \cup \rho(g_{23})^{1/2}T)$ covers the hyperbolic space $\mathbb{H}^3$ evenly.  
\subsection{Poisson sum formula}
From now on, we prove the volume conjecture for $W_p$.  
The colored Jones polynomial $J_N(W_p)$ is given in \eqref{eq:whitehead1} as follows.  
\begin{multline*}
J_{N-1}(W_p)
=
\\
N \frac{q^{p\frac{(N-1)^2}{4}}
}{4\pi i}
 \sum_{k=0}^{N-1}
 \frac{d}{dx}\!\!
 \left(
%q^{p(k+\varepsilon-\frac{N-1}{2})^2-p\varepsilon^2}
q^{p(x-\frac{N-1}{2})^2}
\{2x+1\}
\hfill
\left.
\sum_{l=0}^{N-1}
\sum_{s - \frac{x-k}{2} = \max(k,l)}^{\min(k+l, N-1)}
\!\!\!\!
\xi_N(x, l, s)\!\right)\right|_{x=k},
\end{multline*}
where $\xi_N(k, l, s) = \dfrac{  \{s\}!^2}
{\{s-k\}!^2 \, \{s-l\}!^2 \, \{k+l-s\}!^2 }$.
The function $\xi_N(x, l, s)$ is real valued and 
it takes the maximal at $s_0$ given in \eqref{eq:saddles} and $l = \frac{N-1}{2}$.  
Hence
\[
J_{N-1}(W_p)
=
N \frac{q^{p\frac{(N-1)^2}{4}}
}{4\pi i}
 \sum_{k=0}^{N-1}
 \left.
 \frac{d}{dx}
 \Big(
%q^{p(k+\varepsilon-\frac{N-1}{2})^2-p\varepsilon^2}
q^{p(x-\frac{N-1}{2})^2}
\{2x+1\}
D_N
\xi_N(x, \tfrac{N-1}{2}, s_0)\!\Big)\!\right|_{x=k}
\]
where $D_N$ is a constant with polynomial growth and 
\[
s_0 = \frac{N}{2 \pi i} \log w_0, \ \  
w_0 = \frac{(u +1)(v+1) - \sqrt{(u +1)^2(v+1)^2 - 16 u v}}{4}= -\sqrt{u},\]
\[u = q^{2x+1}, \qquad v = q^{2l+1}=-1
\]
 as shown in Appendix C.  
Let $N\alpha = x+ \frac{1}{2}$, $N\gamma_0=s_0+\frac{1}{2}$ and
\begin{multline*}
\Psi_W(\alpha) = 
- 4\pi^2 \big(
 \gamma_0^2 - 2( \alpha +\tfrac{1}{2})\gamma_0 +  \alpha^2 + \tfrac{1}{2} \alpha + \tfrac{1}{4}
\big)
\\
-2\Li(e^{2\pi i \gamma_0})
+2\Li(e^{2 \pi i (\gamma_0-\alpha)})+
2\Li(-e^{2 \pi i\gamma_0})
+2\Li(-e^{2 \pi i (\alpha-\gamma_0)})
-\frac{2\pi^2}{3}
.
\end{multline*}
Then
\begin{multline*}
J_{N-1}(W_p)
=
\\E_N  q^{p\frac{(N-1)^2}{4}}
%\,\times
%\\
 \sum_{k=0}^{N-1}
 \left.\!\!
 \frac{d}{d\alpha}
%q^{p(k+\varepsilon-\frac{N-1}{2})^2-p\varepsilon^2}
%q^{p(x-\frac{N-1}{2})^2}
\{2N\alpha\}
\exp\Big(\!\tfrac{N}{2\pi i}\big(
-2\pi^2 p(\alpha-\tfrac{1}{2})^2\! +\! \Psi_W(\alpha)\big)\!\Big)\right|_{\alpha=\frac{2k+1}{2N}=\alpha}
,
\end{multline*}
where $E_N$ is a constant  which grows at most  polynomially with respect to $N$.  
\par
To see the asymptotics of $J_{N-1}(W_p)$, 
we use the Poisson sum formula.  
Let $f$ be a rapidly decreasing function, 
then
\[
\sum_{k\in \mathbb{Z}} f(k) = \sum_{k \in \mathbb{Z}} \hat{f}(k)
\]
where $\hat{f}$ is the Fourier transform of $f$ given by
\[
\hat{f}(x) = \int_{\mathbb{R}} e^{-2\pi i k t} f(t) \, dt.
\]
To apply this to the parameter $l$, we extend the function $\Psi_W$ by $0$ for $\alpha\leq0$ and $\alpha \geq 1$.  
Then
\begin{small}
\begin{multline*}
J_{N-1}(W_p)
=E_N \, q^{p\frac{(N-1)^2}{4}}
\,\times
\\
 \left.\sum_{k\in \mathbb{Z}}
 \int_0^N
e^{-2\pi i k t}
  \frac{d}{d\alpha}
%q^{p(k+\varepsilon-\frac{N-1}{2})^2-p\varepsilon^2}
%q^{p(x-\frac{N-1}{2})^2}
\{2N\alpha\}
\exp\Big(\tfrac{N}{2\pi i}\big(
-2\pi^2 p(\alpha-\tfrac{1}{2})^2 + \Psi_W(\alpha))\Big)\right|_{\alpha=\frac{2t+1}{2N}}
\!\!dt 
=
\\
N E_N \, q^{p\frac{(N-1)^2}{4}}
\,\times\hfill
\\
 \sum_{k\in \mathbb{Z}}
 \int_0^1
 e^{-2 \pi i k N\alpha+\pi i k}
 \frac{d}{d\alpha}
\{2N\alpha\}
\exp\Big(\tfrac{N}{2\pi i}\big(
-2\pi^2 p(\alpha-\tfrac{1}{2})^2 + \Psi_W(\alpha))\Big)
d\alpha 
. 
    \end{multline*}
\end{small}
Now we apply integral by part and we get
\begin{multline*}
J_{N-1}(W_p)
=N E_N \, q^{p\frac{(N-1)^2}{4}}
\,\times
\\
 \left.\sum_{k\in \mathbb{Z}}
 (-1)^k
e^{-2\pi i N k \alpha}
%q^{p(k+\varepsilon-\frac{N-1}{2})^2-p\varepsilon^2}
%q^{p(x-\frac{N-1}{2})^2}
\{2N\alpha\}
\exp\Big(\tfrac{N}{2\pi i}\big(
-2\pi^2 p(\alpha-\tfrac{1}{2})^2 + \Psi_W(\alpha))\Big)\right|_{-\infty}^\infty
-
\\
 E_N q^{p\frac{(N-1)^2}{4}}
\,\times\hfill
\\
  2 \pi i \sum_{k\neq0}
 (-1)^{k+1}
 k \!  \int_0^1\!
e^{-2\pi i N k \alpha}
\{2N\alpha\}
\exp\Big(\tfrac{N}{2\pi i}\big(
    -2\pi^2 p(\alpha-\tfrac{1}{2})^2 + \Psi_W(\alpha))\!\Big)
d\alpha
\\
=
2   \pi i E_N \, q^{p\frac{(N-1)^2}{4}}
\,\times\hfill
\\
 \sum_{k\in \mathbb{Z}}
 (-1)^k k   \int_0^1
e^{-2\pi i N k \alpha}
\{2N\alpha\}
\exp\Big(\tfrac{N}{2\pi i}\big(
-2\pi^2 p(\alpha-\tfrac{1}{2})^2 + \Psi_W(\alpha))\Big)
d\alpha. 
\end{multline*}
Let
\begin{equation*}
\Phi_{W_p}(\alpha) = 
-2\pi^2 p(\alpha-\tfrac{1}{2})^2 + \Psi_W(\alpha). 
\end{equation*}
Then
\[
J_{N-1}(W_p) = 
E_N
 q^{p\frac{(N-1)^2}{4}}
\sum_{k \neq 0}
(-1)^kk
\int_0^1
e^{-2\pi i Nk \alpha}
\{2N\alpha\}
e^{\frac{N}{2 \pi i}\Phi_{W_p}(\alpha)}
d\alpha.
\]
In the rest,  we follow the method in \cite{O1}.  
Let 
\[
\Phi_{W_p}^+(\alpha) = \Phi_{W_p}(\alpha) - \frac{4\pi^2}{N} \alpha, \qquad
\Phi_{W_p}^-(\alpha) = \Phi_{W_p}(\alpha) + \frac{4\pi^2}{N} \alpha.
\]
Then we have
\begin{multline*}
J_{N-1}(W_p) = 
\\
E_N\,
 q^{p\frac{(N-1)^2}{4}}
\sum_{k \neq 0}
(-1)^kk
\int_0^1
e^{-2\pi i Nk \alpha}
\left(e^{\frac{N}{2 \pi i}\Phi_{W_p}^+(\alpha)}
-
e^{\frac{N}{2 \pi i}\Phi_{W_p}^-(\alpha)}\right)
d\alpha.
\end{multline*}
\subsection{Saddle point method}
Here we investigate 
\[
\lim_{N \to \infty} \frac{2 \pi}{N} \log
\int_{0} ^1 
e^{-2  \pi i Nk \alpha}
E'_N 
e^{N\Phi_{W_p}^\pm(\alpha)} \, d\alpha
\]
with the help of  the saddle point method.  
We first compute for $k=1$.  
Let $v_W$ be the hyperbolic volume of the complement of $W$.  
Choose a small positive $\delta$ so that 
$|\im\Phi_{W_p}^\pm(\alpha)| < {v_W}$ for $\alpha \in [0, \delta]$ and $[1-\delta, 1]$ and   
we devide the integral in the above formula  into three parts.
\begin{multline*}
\int_{0} ^1 
e^{-2  \pi i  N \alpha}
E'_N 
e^{\frac{N}{2\pi i}\Phi_{W_p}^\pm(\alpha)} \, d\alpha
=
\int_{0} ^{\delta}
e^{-2  \pi i  N \alpha}
E'_N 
e^{\frac{N}{2\pi i}\Phi_{W_p}^\pm(\alpha)} \, d\alpha
+
\\
\int_{\delta} ^{1-\delta} 
e^{-2  \pi i  N \alpha}
E'_N 
e^{\frac{N}{2\pi i}\Phi_{W_p}^\pm(\alpha)} \, d\alpha
+
\int_{1-\delta} ^1 
e^{-2  \pi i  N \alpha}
E'_N 
e^{\frac{N}{2\pi i}\Phi_{W_p}^\pm(\alpha)} \, d\alpha.  
\end{multline*}
Then, 
\[
\left|
\int_{0} ^{\delta}
e^{-2  \pi i  N \alpha}
E'_N 
e^{\frac{N}{2\pi i}\Phi_{W_p}^\pm(\alpha)} \, d\alpha
\right|
< E^{\prime\prime}_N e^{N \frac{v_W}{2\pi}}
\]
and 
\[
\left|
\int_{1-\delta} ^{1}
e^{-2  \pi i  N \alpha}
E'_N 
e^{\frac{N}{2\pi i}\Phi_{W_p}^\pm(\alpha)} \, d\alpha
\right|
< E^{\prime\prime}_N e^{N \frac{v_W}{2\pi}}
\]
for some factors $E_N^{\prime\prime}$ with polynomial growth.  
%\par
The remaining integral is estimated by the value at the saddle point, where the saddle point $\alpha_0$ is the point that the differential of $\Phi_{W_p}^{\pm}(\alpha)$ vanishes.  
\par
Now let us consider the Whitehead link case, i.e. $p=2$.  
Let $\alpha_0$ be the solution of
\[
\frac{1}{2\pi i}\frac{d}{d \alpha} \left(4\pi^2 \alpha + \Phi_{W_2}(\alpha) \right)=0. 
\]
By taking the exponential of this equation, we get
\begin{equation}
- \frac{(1+e^{2 \pi i \frac{\alpha+1}{2}})^2}
{(1 - e^{2 \pi i \frac{\alpha+1}{2}})^2}
e^{4 \pi i \frac{\alpha+1}{2}}
=
- \frac{(1-e^{ \pi i \alpha})^2}
{(1 + e^{ \pi i \alpha})^2}
e^{4 \pi i \alpha}
=
1.
\label{eq:equationalpha}
\end{equation}
Note that this equation is equal to \eqref{eq:equationu} by putting $u = e^{2 \pi i \alpha}$, and is an algebraic equation.  
So it has several solutions and they satisfy 
\[
\frac{1}{2\pi i}\frac{d}{d \alpha} \left( \Phi_{W_2}(\alpha) \right)=2 \pi i k.   \qquad (k \in \mathbb{Z})
\]
Then $\alpha_0$ is one of the solutions of \eqref{eq:equationalpha} satisfying
\[
\frac{1}{2\pi i}\frac{d}{d \alpha} \left( \Phi_{W_2}(\alpha_0) \right)=2 \pi i.  
\]
We actually have such solution $\alpha_0 = 0.856035... - 0.168907... i = \frac{1}{2 \pi i}\log(1-i + \sqrt{-1-2i})$.  
We can see this solution as the saddle point in
the contour graph of $\operatorname{Re} \Phi_{W_2}(\alpha)$  given in Figure \ref{fig:contourW}. 
\begin{figure}[htb]
\[
\raisebox{2cm}{$\begin{matrix}
\operatorname{Im} \alpha \\
\uparrow
\end{matrix} $}
\,
\begin{matrix}
\includegraphics[scale=0.35]{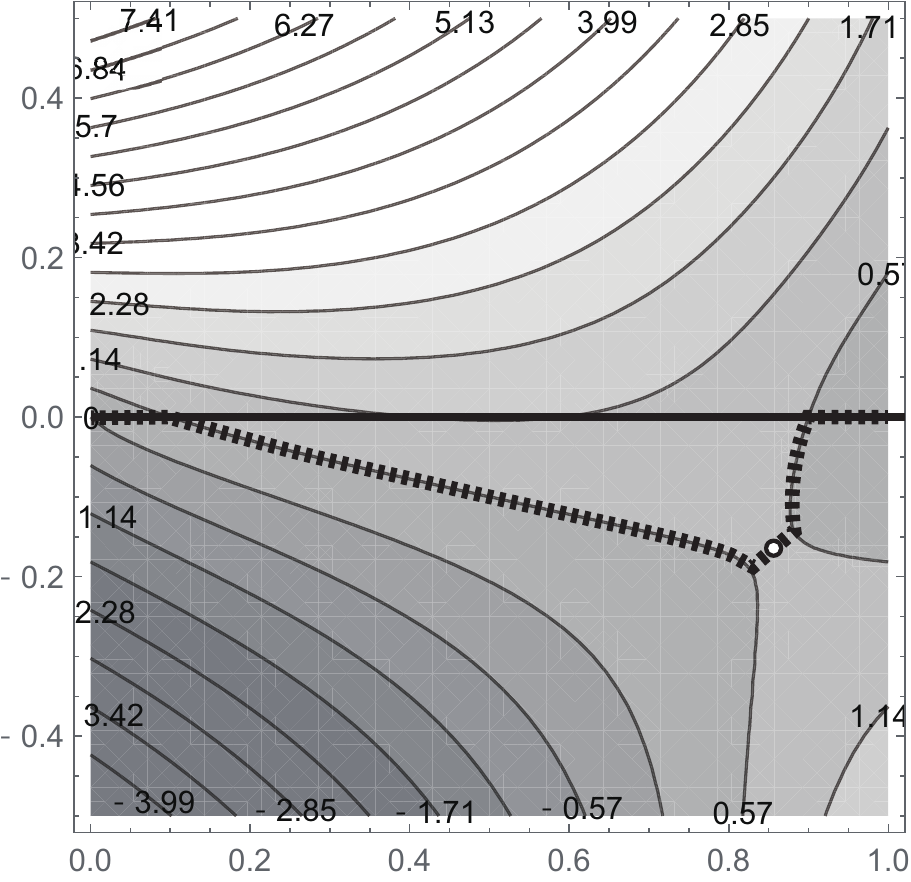}
\\
\qquad\qquad\qquad
\to \operatorname{Re} \alpha
\end{matrix}
\ 
\begin{tabular}{ll}
\small
Thick line : the original integral path
\\[5pt]\small
Dashed line : the deformed  path
\\[5pt]\small
$\circ$ : The saddle point
\end{tabular}
\]
\caption{Contour graph of $\mathrm{Re} \frac{1}{2\pi i}(4 \pi^2 \alpha+\Phi_{W_2}(\alpha))$.}
\label{fig:contourW}
\end{figure}
In this case, the end points $\alpha=0$ and $\alpha=1$ of the integral path
are located on the different regions of $\operatorname{Re}\frac{1}{2\pi i}\Phi_{W_2}(\alpha) \leq 0.57$ and we can apply the saddle point method by deforming the integral path to the dashed line in Figure \ref{fig:contourW}.  
Therefore, 
\[
\int_{\delta} ^{1-\delta} 
e^{-2  \pi i  N \alpha}
E'_N 
e^{N\Phi_{W_2}(\alpha)} \, d\alpha
\underset{N\to \infty}{\sim}
\frac{1}{2\pi i}
e^{\frac{N}{2\pi i}\left(4 \pi^2 \alpha_0 + \Phi_{W_2}(\alpha_0)\right)}. 
\]
Let $\alpha_0^\pm$ be the solution of
\[
\frac{d}{d \alpha}
\left(4\pi^2 \alpha +  \Phi_{W_2}^\pm(\alpha)\right) =0. 
\]
Let 
\[
c = \dfrac{\frac{1}{N}}{\frac{d^2}{d \alpha^2}
\left(4\pi^2 \alpha_0 +  \Phi_{W_2}^\pm(\alpha_0)\right)}.
\]
Then %, %since $\operatorname{Re}\alpha_0 > 0$ \textcolor{red}{??????}, 
%there is a constant $c$ such that 
$\alpha_0^\pm = \alpha_0 \pm \frac{c}{N} + O(\frac{1}{N^2})$ and
\begin{multline*}
\int_{\delta} ^{1-\delta} 
e^{-2  \pi i  N \alpha}
D_N (\alpha)
e^{\frac{N}{2 \pi i}\Phi_{W_2}^\pm(\alpha)} \, d\alpha
\underset{N\to \infty}{\sim}
D_N(\alpha_0)
e^{\frac{N}{2 \pi i}\left(4 \pi^2 \alpha_0^\pm + \Phi_{W_2}^\pm(\alpha_0^\pm )\right)}
\\
=
D_N(\alpha_0)
e^{\frac{N}{2 \pi i}\left(4 \pi^2 (\alpha_0\pm \frac{c}{N}) + 
\Phi_{W_2}^\pm(\alpha_0\pm \frac{c}{N} 
+ O(\frac{1}{N^2}) )\right)}
\\
=
D_N(\alpha_0)
e^{\frac{N}{2 \pi i}\left(4 \pi^2 \alpha_0 + 
\Phi_{W_2}(\alpha_0 )\mp \frac{\alpha_0}{N} + O(\frac{1}{N^2}) \right)}.
\end{multline*} 
Therefore,
\begin{multline*}
\lim_{N \to \infty} \frac{2 \pi}{N}\log |J_{N-1}(W_2)|
=
%\lim_{N \to \infty} \frac{2 \pi}{N}
%\, \times
%\\
%\log 
%\left|
%D_N(\alpha_0)
%\left(
%e^{\frac{N}{2 \pi i}\left(4 \pi^2 \alpha_0 + 
%\Phi_{W_2}(\alpha_0 ) - \frac{\alpha_0}{N} + O(\frac{1}{N^2}) \right)}
%-
%e^{\frac{N}{2 \pi i}\left(4 \pi^2 \alpha_0 + 
%\Phi_{W_2}(\alpha_0 ) + \frac{\alpha_0}{N} + O(\frac{1}{N^2})
%\right)}
%\right)
%\right|
%\hfill
\\
=\lim_{N \to \infty} \frac{2 \pi}{N}
\log 
\left|
(e^{-\frac{\alpha_0}{2\pi i}} - e^{\frac{\alpha_0}{2\pi i}})
D_N(\alpha_0)
e^{\frac{N}{2 \pi i}\left(4\pi^2 \alpha_0 + 
\Phi_{W_2}(\alpha_0 ) + O(\frac{1}{N^2}) \right)}
\right|
\\
=
 \operatorname{Im}\big(4\pi^2 \alpha_0 +\Phi_{W_2}(\alpha_0)\big).  
\end{multline*}
\par
For  $p > 2$, the contour graph is similar to the case $p=2$ and we can apply the similar  argument to get
\[
\lim_{N \to \infty} \frac{2 \pi}{N}\log |J_{N-1}(W_p)|
=
 \operatorname{Im}\big(4\pi^2 \alpha_0 + \Phi_{W_p}(\alpha_0^{(p)})\big),
\]
where $\alpha_0^{(p)}$ is the solution of 
\begin{equation}
\frac{d}{d \alpha}\left(
4\pi^2 \alpha+ \Phi_{W_p}(\alpha)\right) =0. 
\label{eq:equationalphaplog}
\end{equation}
For positive $p$, $1/2 < \re \alpha_0^{(p)} < 1$ and so
\[
e^{\gamma_0(\alpha_0^{(p)}, 1/2)} = e^{2 \pi i(\alpha_0^{(p)}+1)/2} = -e^{\pi i \alpha_0^{(p)}} = \sqrt{e^{2 \pi i \alpha_0^{(p)}}}.
\]  
By taking the exponential of the equation \eqref{eq:equationalphaplog}, we see that $\alpha_0^{(p)}$ is a solution of 
\begin{equation}
- \frac{(1-e^{ \pi i \alpha})^2}
{(1 + e^{ \pi i \alpha})^2}
(-e^{2 \pi i \alpha})^p
=
1
\label{eq:equationalphap}
\end{equation}
 satisfying \eqref{eq:equationalphaplog}.  
For such $\alpha_0^{(p)}$, the value $\mathrm{Im}\Phi_{W_p}(\alpha_0^{(p)})$ satisfies 
\[
\mathrm{Im}(4 \pi^2 \alpha_0 + \Phi_{W_2}(\alpha_0) )< \mathrm{Im}(4 \pi^2 \alpha_0^{(p)} +\Phi_{W_p}(\alpha_0^{(p)})) < v_B
\]
where $v_B$ is the hyperbolic volume of the complement of the borromean rings $B$, 
 and the condition to apply the saddle point method is also fulfilled.  
Actually, the contour graph for $p=5$, $20$ is given in Figure \ref{fig:contourW520}.  
If $p$ becomes large, then the term $\re (2 \pi i p (\alpha-\frac{1}{2})^2)$ becomes dominant.  
\begin{figure}[htb]
\[
\begin{matrix}
\begin{matrix}
\includegraphics[scale=1]{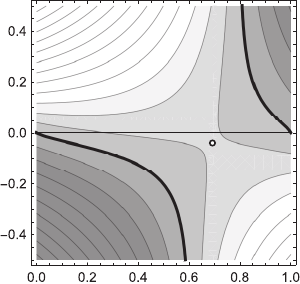}
\end{matrix}
& \qquad & 
\begin{matrix}
\includegraphics[scale=1]{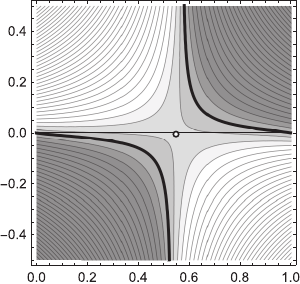}
\end{matrix}
\\
p = 5 & & p=20
\end{matrix}
\]
\caption{The contour graph of $\operatorname{Re}\frac{1}{2 \pi i}\big(4 \pi^2 \alpha +  \Phi_{W_p}(\alpha)\big)$ for $p=5$ and $20$. The thick contour indicates level $0$ and other contours represent integer levels and the small circles represent the saddle points.  }
\label{fig:contourW520}
\end{figure}
%
%\par
%For  $p\leq -2$, 
%$0 < \re \alpha_0^{(p)} < 1/2$ and so
%\[
%e^{\gamma_0(\alpha_0^{(p)}, 1/2)} = e^{2 \pi i(\alpha_0^{(p)}+1)/2} = -e^{\pi i \alpha_0^{(p)}} = -\sqrt{e^{2 \pi i \alpha_0^{(p)}}}.
%\]  
%By taking the exponential of this equation, we see that $\alpha_0^{(p)}$ is a solution of 
%\begin{equation}
%- \frac{(1+e^{ \pi i \alpha})^2}
%{(1 - e^{ \pi i \alpha})^2}
%e^{2p \pi i \alpha}
%=
%1
%\label{eq:equationalphap}
%\end{equation}
% satisfying \eqref{eq:equationalphaplog}.  
%
The saddle points $\alpha_0^{(p)}$ for $2 \leq |p|\leq 100$ are given in Figure \ref{fig:saddleW}.  
\begin{figure}[htb]
\[
\raisebox{1.5cm}
{
$\begin{matrix} \operatorname{Im}\alpha\\
\uparrow
\end{matrix}$
}
\includegraphics[scale=1.25]{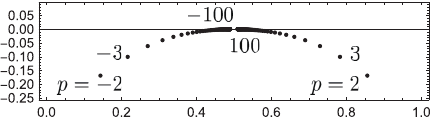}
\ \ \to \operatorname{Re}\alpha
\]
\caption{Saddle points $\alpha_0^{(p)}$ for $|p|\geq 2$.}
\label{fig:saddleW}
\end{figure}
\par
The contribution of the term $k=-1$ is the same as $k=1$ term.  
\par
We have to check the contribution  of the term  $k$  with $|k|\geq 2$ is negrigible.  
In such case, the saddle point moves and the imaginary  part of the value at the saddle point is smaller than $v_{W_p}$.  
If $|k|$ is sufficiently large, then there is no saddle points and the integral path can be moved to the path on which the imaginary part of the value is 0 as for Figure \ref{fig:klarge}.  
\begin{figure}[htb]
\[
\begin{matrix}
\includegraphics[scale=0.7]{contourW5_1}
\\
p=5, \ k=1
\end{matrix}
\qquad
\begin{matrix}
\includegraphics[scale=0.7]{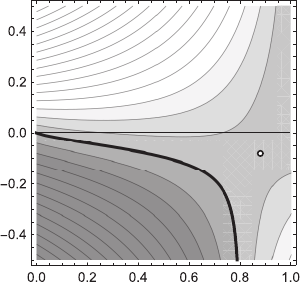}
\\
p=5, \ k=2
\end{matrix}
\qquad
\begin{matrix}
\includegraphics[scale=0.7]{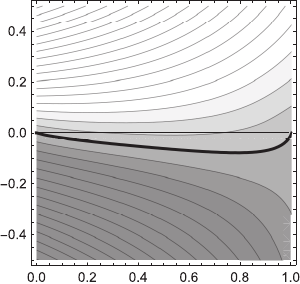}
\\
p=5, \ k=3
\end{matrix}
\]
\caption{The contour graph of $\operatorname{Re}\frac{1}{2 \pi i}\big(4 k\pi^2 \alpha +  \Phi_{W_p}(\alpha)\big)$ for $p=5$ and $k=1$, $2$, $3$.  
The thick lines are contours of level $0$. }
\label{fig:klarge}
\end{figure}
%
%
%
%
%
%On the other hand, 
%\[
%\int\!\!\!\!\int\!\!\!\!\int_{{D\setminus D_2}} 
%e^{-2 \pi i N(m_\alpha \alpha + m_\beta \beta + m_\gamma \gamma)} 
%\frac{\partial}{\partial \alpha}
%e^{\pm2\pi i \alpha+\frac{N}{2\pi i} \, \Phi_W(\alpha, \beta, \gamma) } \, d\alpha\, d\beta\,d\gamma
%\]
%
%
%
%
%
%
\subsection{Volume of the complement}
Here we show the following.
\begin{thm}
The value
 $\frac{1}{i} \big(4 \pi^2 \alpha_0^{(p)} +
 \Phi_{W_p}(\alpha_0^{(p)})\big)$ 
 is equal to the complex volume of the complement of the twisted Whitehead link  $W_p$,  
\end{thm}
 \begin{proof}
The key is the coincidence of  $e^{2 \pi i \alpha_0^{(p)}}$  with the eigenvalue $u$ of $\rho(g_{23})$ in \S\ref{subsection:Wrep}.  
To prove the theorem,  we compare $\zeta(\alpha, \frac{1}{2}, \frac{\alpha+1}{2})$ with the Neumann-Zagier potential function, which relates to the hyperbolic volume of the deformation of the  complement of the Borromean rings $B$ and its variation $B_1$.
For even $p$,   
 $W_p$ is obtained from the Borromean rings $B$ by the $2/p$ surgery along the component $C$ which corresponds to $h_1$ in Figure \ref{fig:borromeangenerator}, and for odd $p$, $W_p$ is obtained from $B_1$ by $2/(p-1)$ surgery.  
We deform the complement of $B$ by changing the cusp shape of $C$.  
\par
First we prove for positive even $p$ case.  
Let $\rho_{\mu,\lambda} : \pi_1(S^3\setminus B) \to \mathrm{SL}(2, \mathbb{C})$ be the non-parabolic representation of $\pi_1(B\setminus B)$ where
$\mu$ and $\lambda$ are eigenvalues of $h_1$ and $g_{23}$ respectively.   
Let $m$ and $l$ be the the dilatations with respect to the meridian and the longitude of the cusp along $C$ respectively, then
it is known that  $e^m = \mu^2$ and $e^l = \lambda^2$, 
and $\rho_{\mu, \lambda}$ gives a deformed hyperbolic structure to the complement of $B$ such that the cusp shape along $C$ matches $\mu$ and $\lambda$. 
For such deformation, the volume of the complement with respect to this deformed hyperbolic structure is studied by Neumann and Zagier \cite{NZ}.  
Let  
 $f(m)$ be the Neumann-Zagier function for the complement of $B$ given in \cite{NZ}.  
The function $f(m)$ is determined by the following differential equation.  
\[
\frac{d}{d m} f(m) = -\frac{1}{2}l, 
\qquad
f(0) = 0.  
\]
\par
Such deformation  is actually realized as a deformation of  a union of two ideal regular octahedrons which form the complement of $B$. 
Let $p_i$ be the fixed point of $\rho(g_i)$ given by \eqref{eq:whiteheadsolution1} for $i=1$, $2$, $3$, $4$.  
Since $h_1$ commute with $g_{23}$ and $\rho(g_{23})$ is a diagonal matrix, $\rho(h_1)$ is also a diagonal matrix, and 
the action of  $\rho(h_1)$ sends $p_1$ to $p_2$ and $p_4$ to $p_3$.  
These points satisfy
\[
p_2 = -\frac{\left(\sqrt{u}-1\right)^2}{\left(\sqrt{u}+1\right)^2}\, p_1, \qquad
p_3 = -\frac{\left(\sqrt{u}-1\right)^2}{\left(\sqrt{u}+1\right)^2}\, p_4,
\]
so 
$e^{m} = \mu^{2} = -\frac{\left(\sqrt{u}-1\right)^2}{\left(\sqrt{u}+1\right)^2}$, 
and we choose
\[
m =\log \left(-\frac{\left(\sqrt{u}-1\right)^2}
{\left(\sqrt{u}+1\right)^2}\right), \qquad
\mu = -e^{\frac{m}{2}}.
\qquad(u = e^{2 \pi i \alpha})
\]  
On the other hand, the eigenvalues of 
$\rho(g_{23})$ are 
$\lambda=u^{\pm1} = e^{\pm 2 \pi i  \alpha}$, 
$e^l = \lambda^2=u^{\pm2} = e^{\pm 4 \pi i  \alpha}$ and  we put 
\[
l = 4 \pi i \alpha - 2\pi i.
\]  
For positive $p$, $\re \alpha_0^{(p)} > 1/2$, so we adjust $l$ so that $0 \leq \im l < 2\pi$ by subtracting $2 \pi i$.
%Note that $l = 2 \xi = 4 \pi i \alpha - 2 \pi i$ and $m = 2\eta$ for $\xi$, $\eta$ in \cite{Yok}.  
The  function $\Phi_W(\alpha)$ satisfies
\begin{equation}
\frac{d}{d l} \Phi_W(\alpha)
=
\frac{1}{4 \pi i}\frac{d}{d\alpha} \Phi_W(\alpha)
=
\frac{1}{2}\log\big(-\frac{\left(\sqrt{u}-1\right)^2}{\left(\sqrt{u}+1\right)^2}
\big)
= 
\frac{1}{2} m .  
\label{eq:alphamu}
\end{equation}
Let 
$H(m) =
 \Phi_W(\alpha) 
-
 \frac{1}{2} m l$ 
 where $\lambda$ and $\mu$ satisfies \eqref{eq:alphamu},
 then we have
\begin{equation}
\frac{\partial}{\partial m}
H(m)
=
\frac{\partial}{\partial \alpha} 
\Phi_W(\alpha) 
\frac{\partial \alpha}{\partial m} 
-
\frac{1}{2} l  
-
\frac{1}{2} m \frac{\partial \alpha}{\partial m}
= 
-\frac{1}{2} l.  
\label{eq:diffmu}
\end{equation}
The differential equation \eqref{eq:diffmu} for $H(m)$ is the same differential equation for the Neumann-Zagier function $\Phi(m)$ in \cite{MYo}, which is explained in Appendix D.  
Note that $u$, $v$ in \cite{MYo} are equal to $m/2$, $l/2$.  
Let 
\[
h(m) = H(m)+ \frac{1}{4} m l.
\]
If $m=0$, then  $\mu = -1$, $u = -1$, $\alpha=\frac{1}{2}$ and $h(0)$ coincides with $\mathrm{Vol}(S^3\setminus B)$. 
Therefore,   $h(\mu)-\mathrm{Vol}(S^3\setminus B)$ equals to the function $f(m)$  in \cite{NZ}.  
Moreover, the length and the torsion of the core geodesic of the surgery component is given by the real part and the imaginary part of $l$.  
%\par
Hence, by \eqref{eq:NZ} in Appendix D,  we have 
\[
\mathrm{Vol}(S^3\setminus W_p) +
i \, \mathrm{CS}(S^3\setminus W_p)
=
\frac{1}{i}
\big(
h(m) - \frac{\pi i}{2}  \log l
\big).
\]
Since $m + \frac{p}{2} l = 2 \pi i$,  we have
\begin{multline*}
i \left(\mathrm{Vol}(S^3\setminus W_p) +
i\,\mathrm{CS}(S^3\setminus W_p)\right)
=
h(m)  - \frac{\pi i}{2}  l
=
\\
\Phi_W(\alpha_0^{(p)}) 
-\frac{1}{4} m l 
- \frac{\pi i}{2} l
%\\
=
\Phi_W(\alpha_0^{(p)}) 
-\frac{1}{4} \big(2 \pi i - 
\frac{p}{2} l\big) l - \frac{\pi i}{2} l
=
\\
\Phi_W(\alpha_0^{(p)}) 
- 4\pi^2\frac{p}{2}\big(\alpha_0^{(p)}-\frac{1}{2}\big)^2 
+ 4 \pi^2 \big(\alpha_0^{(p)} - \frac{1}{2}\big).
\end{multline*}
The last formula coincides with  $\Phi_{W_p}(\alpha_0^{(p)}) - 2\pi i \left(2 \pi i (\alpha_0^{(p)} - \frac{1}{2})\right)$ at the saddle point $\alpha_0^{(p)}$ and so we get
\[
\frac{1}{i}\left(\Phi_{W_p}(\alpha_0^{(p)})- 2\pi i \left(2 \pi i (\alpha_0^{(p)} - \frac{1}{2})\right)\right) = \mathrm{Vol}(S^3\setminus W_p) + i \, \mathrm{CS}(S^3\setminus W_p).
\]  
\par
For positive odd  $p$, $W_p$ is obtained by applying $2/(p-1)$ surgery to the middle complent of  $B_1$ in Figure \ref{fig:borromean}.   
We assign $m$ and $l$ along the component getting the surgery, then we get similar function $h(m)$ which corresponds to the Neumann-Zagier function.  
The only difference is that $ h(0) = \mathrm{Vol}(S^3 \setminus B_1)+ i\, \mathrm{CS}(S^3 \setminus B_1)$, which implies that $\frac{1}{i}  \Phi_{W_p}(\alpha_0^{(p)}) = \mathrm{Vol}(S^3\setminus W_p) + i \, \mathrm{CS}(S^3\setminus W_p)$.  
\par
The proof for negative $p$ case is similar.  
\end{proof}
\section{Double twist knots}
We explain the complexified tetrahedron coming from $\mathrm{SL}(2, \mathbb{C})$ representation of 
$\pi_1(S^3\setminus D_{p,r})$ for the hyperbolic double twist knot $D_{p,r}$., and 
we prove Conjecture 1 for $D_{p, r}$ with the help of the complexified tetrahedron as in the previous section for the twisted Whitehead link.     
Note that the twist knot $T_p$ is equal to $D_{p,2}$, and $D_{-p,-r}$ is the mirror image of $D_{p,r}$,
\subsection{Representation matrices}
We first construct $\mathrm{SL}(2, \mathbb{C})$ representation.   
Let $g_1$, $g_2$, $g_3$, $g_4$, $g_{12}$, $g_{23}$ be elements of $\pi_1(S^3\setminus D_{p,r})$
as  in Figure \ref{fig:doublegenerator}.  
\begin{figure}[htb]
\[
%\text{$(s, t)$ double twist knot $K$}: \ 
%\begin{matrix}
%\includegraphics[scale=0.9]{doubletwistknot}
%\end{matrix}
%\ \ 
%\text{Elements of $\pi_1(S^3\setminus K)$}: \ 
\begin{matrix}
\includegraphics[scale=0.9, alt={generators}]{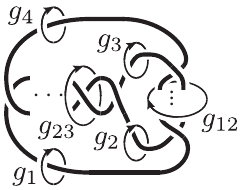}
\end{matrix}
\]
\caption{Elements $g_1$, $g_2$, $g_3$, $g_4$, $g_{12}$, $g_{23}$ in $\pi_1(S^3\setminus D_{p,r})$.}
\label{fig:doublegenerator}
\end{figure}
Then $g_1$, $\cdots$, $g_4$, $g_{12}$, $g_{23}$ satisfy the following relation.  
\begin{align}
&g_{12} = g_1g_2, \ \ 
g_{23} = g_2g_3, \ \ 
g_1g_2g_3g_4 = 1, 
\label{eq:relation1}
\\[3pt]
&\begin{cases}
g_1^{-1} = {g_{23}}^{\frac{p}{2}} g_2 {g_{23}}^{-\frac{p}{2}} \\
g_4^{-1} = {g_{23}}^{\frac{p}{2}} g_3 {g_{23}}^{-\frac{p}{2}} 
\end{cases} \!\!\!\text{if $p$ is even},\ 
\begin{cases}
g_1^{-1} = {g_{23}}^{\frac{p+1}{2}} g_3 {g_{23}}^{-\frac{p+1}{2}} \\
g_4^{-1} = {g_{23}}^{\frac{p-1}{2}} g_2 {g_{23}}^{-\frac{p-1}{2}} 
\end{cases} \!\!\!\text{if $p$ is odd},
\label{eq:relation2}
\\[3pt]
&\begin{cases}
g_4^{-1} = {g_{12}}^{\frac{r}{2}} g_1 {g_{12}}^{-\frac{r}{2}} \\
g_3^{-1} = {g_{12}}^{\frac{r}{2}} g_2 {g_{12}}^{-\frac{r}{2}} 
\end{cases} \!\!\!\text{if $r$ is even},\ 
\begin{cases}
g_4^{-1} = {g_{12}}^{\frac{r+1}{2}} g_2 {g_{12}}^{-\frac{r+1}{2}} \\
g_3^{-1} = {g_{12}}^{\frac{r-1}{2}} g_1 {g_{12}}^{-\frac{r-1}{2}} 
\end{cases}\!\!\! \text{if $r$ is odd}.
\label{eq:relation3}
\end{align}
Let $\rho$ be the geometric $\mathrm{SL}(2, \mathbb{C})$ of $\pi_1(S^3 \setminus D_{p, r})$, then $\rho(g_1)$, $\cdots$, $\rho(g_4)$ are parabolic matrices.  
Here we assume that the eigenvalue of $\rho(g_i)$ is $-1$.  
We also assume that 
\[
\rho(g_{23}) = \begin{pmatrix} u &0 \\ 0 &u^{-1}\end{pmatrix}
\] 
and the eigenvalues of 
$\rho(g_{12})$ are $v$ and $v^{-1}$.  
Then, up to the conjugation, $\rho$ is given as follows.  
\begin{align*}
\rho(g_1) &= 
\left(\begin{matrix}
 -\frac{2 }{{u}+1} &  \frac{u({u}-1)}{{u}+1} \\[3pt]
 -\frac{{u}-1}{u({u}+1)} &  -\frac{2u}{{u}+1} 
\end{matrix}\right),
\\[5pt]
\rho(g_2) &= \!
\begin{pmatrix}
 -\frac{2u}{{u}+1} &  \!\!\!\!\!
-u \frac{(u+1)^2(v^2+1) - 8uv-(u+1)(v-1) \sqrt{D}}
{2v   ({u}-1) ({u}+1)
   }    \\
 \frac{ (u+1)^2(v^2+1) - 8uv+(u+1)(v-1) \sqrt{D}}
  {2uv ({u}-1) ({u}+1)}
& -\frac{2}{{u}+1} 
\end{pmatrix}, 
\\[5pt]
\rho(g_3) &= \!
\left(
\begin{matrix}
 -\frac{2u}{{u}+1} & \!\!\!\!
 \frac{ (u+1)^2(v^2+1) - 8uv-(u+1)(v-1) \sqrt{D}}
 {2 v ({u}-1) ({u}+1)}
    \\
- \frac{(u+1)^2(v^2+1) - 8uv+(u+1)(v-1) \sqrt{D}}
{2v   ({u}-1) ({u}+1)} & 
   -\frac{2}{{u}+1} 
\end{matrix}
\right), 
\\[5pt]
\rho(g_4) &= 
\begin{pmatrix}
-\frac{2}{u+1} & -\frac{u-1}{u+1} 
\\[3pt]
 \frac{u-1}{u+1} & -\frac{2u}{u+1}
\end{pmatrix},
\end{align*}
where $D = (u+1)^2(v+1)^2 - 16uv$.
Let $p_1$, $p_2$, $p_3$, $p_4$ be the fixed points of $\rho(g_1)$, $\rho(g_2)$, $\rho(g_3)$, $\rho(g_4)$ on $\partial \mathbb{H}^2$.  
Then they are given as follows. 
\begin{equation}
\begin{split}
p_1 &=-u, \qquad
p_2 =  -u\frac{(u+1)^2(v^2+1) - 8uv-(u+1)(v-1) \sqrt{D}}
{2v ({u}-1)^2  },
\\
p_3 &= 
\frac{(u+1)^2(v^2+1) - 8uv-(u+1)(v-1) \sqrt{D}}
{2v ({u}-1)^2  },
\qquad
p_4 = 1.  
\end{split}
\label{eq:p}
\end{equation}
Let $p_{23}^0$, $p_{23}^1$ be the fixed points of $\rho(g_{23})$, then $p_{23}^0 = 0$ and $p_{23}^1 = \infty$, and let $p_{12}^0$, $p_{12}^1$ be the fixed points of $\rho(g_{23})$, then they are
\begin{align*}
p_{12}^0 &= -\frac{(u+1)^2(v+1)-8u - (u+1)\sqrt{D}}{4(u-1)}, \\
p_{12}^1 &= 
-\frac{(u+1)^2(v+1)-8uv+ (u+1)\sqrt{D}}{4(u-1)v}.  
\end{align*}
Let $\rho'$ be the representation similar to $\rho$ where $g_{12}$ is mapped to the diagonal matrix  
\[\rho'(g_{12}) = \begin{pmatrix}
v & 0 \\ 0 & v^{-1}
\end{pmatrix}.
\]
Such $\rho'$ is obtained by the transformation matrix 
\[
Q = 
\begin{pmatrix}
 - \frac{(u+1)(v+1)(uv+1)-8uv - (uv-1)\sqrt{D}}{2v(u-1)(v-1)}
 &
-\frac{(u+1)^2(v+1)-8u - (u+1)\sqrt{D}}{4(u-1)} 
\\[3pt]
-\frac{(u+1)(v+1)^2-8uv - (v+1)\sqrt{D}}{4u v(v-1)}
& 1
 \end{pmatrix}.   
\]
For $g \in \pi_1(S^3\setminus K)$, let $\rho'(g) = Q^{-1}\rho(g) \, Q$, 
then we have
\begin{align*}
\rho'(g_1) &= 
\left(
\begin{matrix}
 -\frac{2v}{v+1} &
- \frac{v(v-1)}{v+1} \\[3pt]
 \frac{v-1}{v(v+1)} &
   -\frac{2}{v+1} \\
\end{matrix}
\right),
\qquad
\rho'(g_2) = 
\begin{pmatrix}
-\frac{2v}{v+1} & \frac{v-1}{v+1} 
\\[3pt]
 -\frac{v-1}{v+1} & -\frac{2}{v+1}
\end{pmatrix},
\\[5pt]
\rho'(g_3) &= \!\!
\left(
\begin{matrix}
 -\frac{2}{{v}+1} & \!\!
- \frac{
(u^2+1)(v+1)^2 - 8uv-(u-1)(v+1) \sqrt{D}}
{2  u ({v}-1) ({v}+1)}
    \\
 \frac{(u^2+1)(v+1)^2 - 8uv+(u-1)(v+1) \sqrt{D}}{2u
   ({v}-1) ({v}+1)
   } & 
   -\frac{2v}{{v}+1} 
\end{matrix}
\right)\!, 
\\[5pt]
\rho'(g_4) &=\! \!
\begin{pmatrix}
 -\frac{2}{{v}+1} & 
\!\!\!\!\!\!
v\frac{
  (u^2+1)(v+1)^2 - 8uv-(u-1)(v+1) \sqrt{D}}{2
  u ({v}-1) ({v}+1)}
    \\
- \frac{(u^2+1)(v+1)^2 - 8uv+(u-1)(v+1) \sqrt{D}}
{2uv   ({v}-1) ({v}+1)
   } & 
   -\frac{2v}{{v}+1} 
\end{pmatrix}\!.
\end{align*}
The fixed points $p'_1$, $p'_2$, $p'_3$, $p'_4$ of $\rho'(g_1)$, $\rho'(g_2)$, $\rho'(g_3)$, $\rho'(g_4)$ on $\partial \mathbb{H}^3$  are
\begin{equation}
\begin{split}
p'_1 &=-v, \qquad
p'_2 = 1, \qquad
p'_3 = \frac{(u^2+1)(v+1)^2 - 8uv-(u-1)(v+1) \sqrt{D}}
{2u ({v}-1)^2  },
\\
p'_4 &= 
-v\frac{(u^2+1)(v+1)^2 - 8uv-(u-1)(v+1) \sqrt{D}}
{2u ({v}-1)^2  }.  
\end{split}
\label{eq:pprime}
\end{equation}
The fixed points ${p_{12}^0}'$ and  ${p_{12}^1}'$  of $\rho'(g_{12})$ are ${p_{12}^0}'=0$ and ${p_{12}^1}'=\infty$, and the fixed points  ${p_{23}^0}'$ and  ${p_{23}^1}'$ of $\rho'(g_{23})$ are
\begin{align*}
{p_{23}^0}'&=-\frac{(u+1)(v+1)^2-8v - (v+1)\sqrt{D}}{4(v-1)}, 
\\
{p_{23}^1}'&=-\frac{(u+1)(v+1)^2-8uv + (v+1)\sqrt{D}}{4u(v-1)}.  
\end{align*}
The eigenvalues $u$ and $v$ are determined by the relations \eqref{eq:relation2} and \eqref{eq:relation3}.  
They satisfy
\begin{equation}
(-u)^{-p} p_4 = (-u)^{-p}= p_3, \qquad (-v)^{r} p'_2 =  (-v)^{r} = p'_3.
\label{eq:uv}
\end{equation}
Moreover, the geometric representation is given by a solution among the solutions of \eqref{eq:uv} satisfying
\begin{equation}
p \log (-u) + \log p_3 = \pm 2 \pi \sqrt{-1}, \qquad
-r \log (-v) + \log p'_3 = \pm 2 \pi \sqrt{-1}.  
\label{eq:complete}
\end{equation}
\subsection{Complexified tetrahedron}
Here we explain the complexified tetrahedron $T$ determined by the fixed points $p_1$, $\cdots$, $p_{12}^1$, which is congruent to the complexified tetrahedron $T'$ determined by $p'_1$, $\cdots$, ${p_{12}^1}'$.   
\par 
For $D_{6,2}$,  the solution of equations \eqref{eq:uv} and \eqref{eq:complete}  where the sums are both $+2\pi i$ is
given by
\[
u = -0.619307 - 0.884567 i, \qquad
v = 1.72565 + 2.06055 i, 
\]
The fixed points are given  as follows. 
\[
\begin{tabular}{l}
$p_1 = 0.6193 + 0.8846 i$, \  \ $p_2 =0.0596 + 0.6786 i$, \ \ 
 $p_3 = 0.5464 + 0.3152 i$, 
 \\
$p_4 =  1$, 
\ \ 
$p_{12}^0 = 0.2495 + 0.7240 i$,  \ \ $p_{12}^1 = 0.8631 + 0.2152 i$, 
\\[3pt]
$p'_1 =  -1.7257 - 2.0606 i$,  \ \ $p'_2 = 1$,\ \ 
$p'_3 = -1.2680 + 7.1116 i$,\\
$p'_4 =  16.842 - 9.659 i$,
${p_{23}^0}' = 3.974 + 0.959 i$,  \ \ ${p_{23}^1}' = 3.450 - 3.264 i$.
\end{tabular}
\]
Then, $T$ and $T'$ in $\mathbb{H}^3$ corresponding to $D_{6,2}$ are given as in Figure \ref{fig:doubletet}.  
\begin{figure}[htb]
\[
\begin{matrix}
\begin{matrix}
\includegraphics[scale=0.6, alt={complexified tetrahedron}]{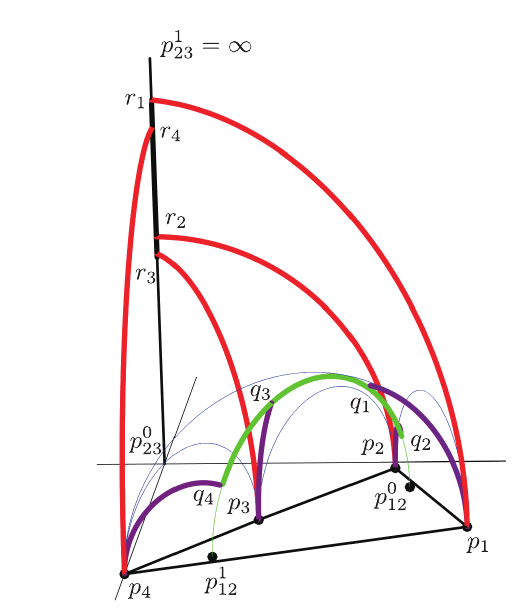}
\end{matrix}
&\underset{Q^{-1}}\longrightarrow & 
\begin{matrix}
\includegraphics[scale=0.6, alt={complexified tetrahedron}]{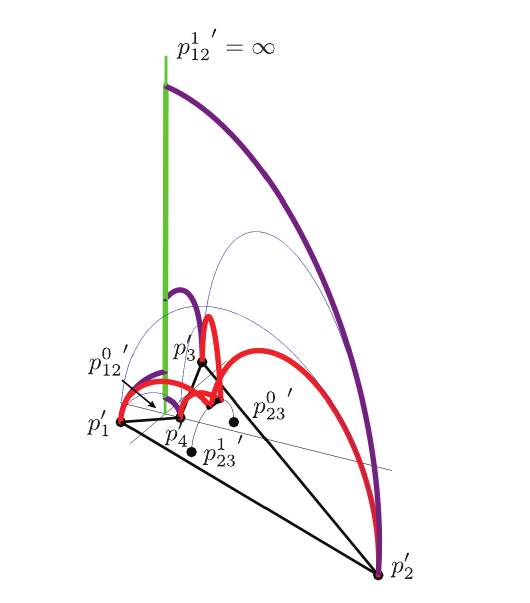}
\end{matrix}
\\
T & & T'
\end{matrix}
\]
\caption{The complexified tetrahedrons $T$ and $T'$ corresponding to $D_{6,2}$.}
\label{fig:doubletet}
\end{figure}
The elements
$\rho(g_{23})$, $\rho(g_{12})$ have axes $l_{23}$, $l_{12}$, so we assign complex parameters to these axes $u$, $v$, which is the eigenvalues of $g_{23}$, $g_{12}$.  
Let $r_1$, $r_2$, $r_3$, $r_4$ be the foots of perpendicular  on $l_{23}$ from $p_1$, $p_2$, $p_3$, $p_4$.  
Similarly, Let $q_1$, $q_2$, $q_3$, $q_4$ be the foots of perpendicular  on $l_{12}$ from $p_1$, $p_2$, $p_3$, $p_4$.  
Let us define eight faces $p_1p_2r_2r_1$, $p_2p_3r_3r_2$, $p_3p_4r_4r_3$, $p_4p_1r_1r_4$, $p_1p_2q_2q_1$, $p_2p_3q_3q_2$, $p_3p_4q_4q_3$, $p_4p_1q_1q_4$.  
These faces are not flat and are not defined uniquely, but the edges of the faces are straight lines and we define these faces topologically.  
Let $T$ be the subset of $\mathbb{H}^3$ surrounded by these eight faces, and this is the complexified tetrahedron corresponding to the representation $\rho$.   
Let $T_1$ be similar complexified tetrahedron constructed from $(-u) p_1$, $(-u)p_2$, $(-u)p_3$, $(-u)p_4$,  $(-u)l_{12}$ and $(-u)l_{23} = l_{23}$.  
Then $T$ and $T_1$ are adjacent at the face $p_3p_4r_4r_3$ and $T\cup T_1$ is a fundamental domain of the action of $\pi_1(S^3\setminus D_{6,2})$ to $\mathbb{H}^3$ given by $\rho$.  
\par
The action of $\rho(g_{23})$ on $\partial \mathbb{H}^3$ corresponds to the multiplication of $u^2$, so we get the picture in the upper row of Figure \ref{fig:action}.  
\begin{figure}[htb]
\[
\begin{matrix}
\begin{matrix}
\includegraphics[scale=0.6]{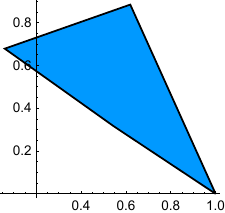}
\end{matrix}
&
\begin{matrix}
\includegraphics[scale=0.8]{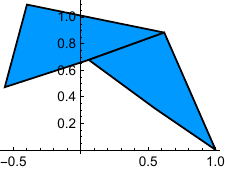}
\end{matrix}
&
\begin{matrix}
\includegraphics[scale=0.9]{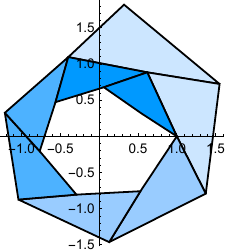}
\end{matrix}
\\
\text{ $p_1p_2p_3p_4$}
&
p_1p_2p_3p_4 \cup (-u)p_1p_2p_3p_4
&
p_1p_2p_3p_4 \cup\cdots \cup (-u)^7p_1p_2p_3p_4
%\\
%& \text{the action of $g_{23}$} 
\\[5pt]
\begin{matrix}
\includegraphics[scale=0.6]{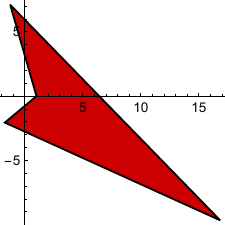}
\end{matrix}
&
\begin{matrix}
\includegraphics[scale=0.85]{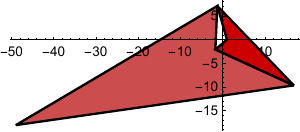}
\end{matrix}
&
\begin{matrix}
\includegraphics[scale=0.45]{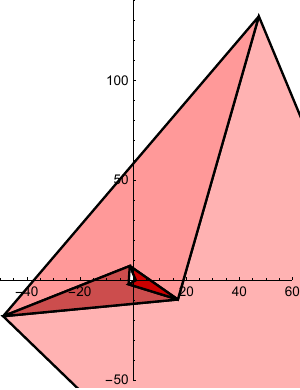}
\end{matrix}
\\
\text{ $p'_1p'_2p'_3p'_4$}
&
p'_1p'_2p'_3p'_4 \cup (-v)p'_1p'_2p'_3p'_4
&
p'_1p'_2p'_3p'_4 \cup\cdots \cup (-v)^{3}p'_1p'_2p'_3p'_4
%\\
%& 
%\text{the action of $g_{12}$}
\end{matrix}
\]
\caption{The actions of $\rho(g_{23})$ and $\rho'(g_{12})$ on $\partial \mathbb{H}^3$.  
The upper row explains the action of $\rho(g_{23})$ and the lower row explains the action of $\rho'(g_{12})$.  
They act $\partial \mathbb{H}^3$ by rotations and enlargements around the origin. 
%The middle pictures correspond to the fundamental domain of the action of $\pi_1(S^3\setminus D_{6,2})$.
}
\label{fig:action}
\end{figure}
Similarly, the action of $\rho'(g_{12})$ corresponds to the multiplication of $v^{2}$ and is also explained in the lower row of the figure.  
These pictures show that $p_3$ is  the square of the eigenvalue of the element in $\pi_1(S^3\setminus D_{p,r})$ representing the meridian, and $-u$ is the eigenvalue of the element representing the longitude of  the first surgery component for constructing $D_{p,r}$ from the Borromean rings $B$ (or $B_1$, $B_{1,1}$).  
Similarly, $p'_3$ corresponds to the square of the eigenvalue of the element in $\pi_1(S^3\setminus D_{p,r})$ representing the meridian, and $-v$ is the eigenvalue of the element representing the longitude of  the second surgery component. 
These diagrams represents the cusp shapes around the surgery components which are edges of the two complexified tetrahedrons giving the decomposition of the complement.  
\subsection{Poisson sum formula}
We reformulate the colored Jones polynomial $J_{N-1}(D_{p, r})$ into integral form by using the Poisson sum formula.  
The colored Jones polynomial $J_{N-1}(D_{p, r})$ is given by \eqref{eq:double1} in Appendix C as follows.  
\begin{multline*}
J_{N-1}(D_{p, r}) 
=
-\frac{N^2q^{(p-r)\frac{(N-1)^4}{4}}}{16 \pi^2} \, \times
\\
\sum_{k,l=0}^{N-1}\!\!
\frac{\partial^2}{\partial x \partial y}
q^{p(x - \frac{N-1}{2})^2-r(y-\frac{N-1}{2})^2}
\{2x+1\}\{2y+1\}\!\!\!\!\!\!\!\!\!\!
\left.
 \sum_{s-\frac{x-k+y-l}{2} = \max(k,l)}^{\min(k+l, N-1)}
 \!\!\!\!\!\!\!\!\!\!\!\!
 \xi_N(x, y, s)
\right|_{\text{$\scriptstyle\begin{matrix}x=k\\[-5pt]
y=l\end{matrix}$}}.  
\end{multline*}
Since $\xi_N(x, y, s)$ is a real positive number, 
it takes the maximal at $s_0$ given in \eqref{eq:saddles}.  
Hence
\begin{multline*}
J_{N-1}(D_{p, r}) 
=
-\frac{N^2q^{(p-r)\frac{(N-1)^4}{4}}}{16 \pi^2} \, \times
\\
\sum_{k,l=0}^{N-1}
\left.\frac{\partial^2}{\partial x \partial y}
q^{p(x - \frac{N-1}{2})^2-r(y-\frac{N-1}{2})^2}
\{2x+1\}\{2y+1\}
F_N \xi_N(x, y, s_0)
\right|_{\text{$\scriptstyle\begin{matrix}x=k\\[-5pt]
y=l\end{matrix}$}}.   
\end{multline*}
where $F_N$ is a constant with polynomial growth and%
\[
s_0 = \frac{N}{2 \pi i} \log w_0,  \quad
w_0 = \frac{(u +1)(v+1) - \sqrt{(u +1)^2(v+1)^2 - 16 u v}}{4}, 
\]
\[
u = q^{2x+1}, \qquad
v = q^{2y+1}
\]
 as shown in Appendix C.  
Let $N\alpha = x+ \frac{1}{2}$, $N\beta = y+ \frac{1}{2}$, $N\gamma_0=s_0+\frac{1}{2}$ and
\begin{multline*}
\Psi_B(\alpha, \beta) = 
- 4\pi^2 \big(
 \gamma_0^2 - 2( \alpha +\beta)\gamma_0 + \alpha^2 + \alpha\beta + \beta^2
\big)
-2\Li(e^{2\pi i \gamma_0})
\\
+2\Li(e^{2 \pi i (\gamma_0-\alpha)})+
2\Li(e^{2 \pi i(\gamma_0-\beta)})
+2\Li(-e^{2 \pi i (\alpha+\beta-\gamma_0)})
-\frac{2\pi^2}{3}.
\end{multline*}
Then
\begin{multline*}
J_{N-1}(D_{p, r}) 
=
G_N \frac{N^2q^{(p-r)\frac{(N-1)^4}{4}}}{16 \pi^2} \, \times
\sum_{k,l=0}^{N-1}
\frac{\partial^2}{\partial x \partial y}
%q^{p(k - \frac{N-1}{2})^2-r(l-\frac{N-1}{2})^2}
\{2N\alpha\}\{2N\beta\}
\\
\left.
\exp\Big(\tfrac{N}{2\pi i} \big(-2 \pi^2 p(\alpha-\tfrac{1}{2})^2 + 2\pi^2 r(\beta-\tfrac{1}{2})^2 + \Psi_B(\alpha, \beta)
\big)\Big)
\right|_{\text{$\scriptstyle\begin{matrix}\alpha=\frac{2k+1}{2N}\\[0pt]
\beta=\frac{2l+1}{2N}\end{matrix}$}},   
\end{multline*}
where $G_N$  is a constant  with polynomial growth.  
\par
Now we apply the Poisson sum formula for $k$ and $l$.  
Let 
\begin{align*}
&\Phi_{D_{p, r}}(\alpha, \beta) 
=
\frac{1}{2\pi i}\left(
-2\pi^2p(\alpha-\frac{1}{2})^2 + 4 \pi^2 r(\beta-\frac{1}{2})^2 +  \Psi_D(\alpha, \beta)
\right),
\\
&\Phi_{D_{p, r}}^{\varepsilon_1, \varepsilon_2}(\alpha, \beta) 
=
\\
&\qquad
\frac{1}{2\pi i}\left(
-2\pi^2p((\alpha-\frac{1}{2})^2 +\varepsilon_1 \frac{\alpha}{N}) + 4 \pi^2 (r(\beta-\frac{1}{2})^2 - \varepsilon_2\frac{\beta}{N}) +  \Psi_D(\alpha, \beta)
\right), 
\end{align*}
where $\varepsilon_1, \varepsilon_2 = \pm1$.  
Then
\begin{multline*}
J_{N-1}(D_{p, r}) 
=
\\
-\frac{N^2q^{(p-r)\frac{(N-1)^4}{4}}}{16 \pi^2} 
\sum_{\varepsilon_1, \varepsilon_2 \in \{-, +\}}
%\\
\sum_{k,l=0}^{N-1}
\frac{\partial}{\partial\alpha}
\left.
C_N e^{\frac{N}{2\pi i}\left(\Phi_{D_{p, r}}^{\varepsilon_1, \varepsilon_2}(\alpha, \beta) + O(\frac{1}{N})\right)}
\right|_{\text{$\scriptstyle\begin{matrix}\alpha=\frac{2k+1}{2N}\\[0pt]
\beta=\frac{2l+1}{2N}\end{matrix}$}}.
\end{multline*}
As in the case of twisted Whitehead links, the Poisson sum formula yields
\begin{multline*}
J_{N-1}(D_{p, r}) 
=
q^{(p-r)\frac{(N-1)^4}{4}} 
\, \times
\\
\sum_{\varepsilon_1, \varepsilon_2  \in \{-, +\}}
\sum_{m, n \in \mathbb{Z}}\!\!
(-1)^{m+n} \!\!
\int\!\!\!\!\int_D
\!\!
C'_N 
e^{-2\pi i (k\alpha + l\beta)}\!
\frac{\partial}{\partial\alpha}
e^{\frac{N}{2\pi i}\left(\Phi_{D_{p, r}}^{\varepsilon_1, \varepsilon_2}(\alpha, \beta) + O(\frac{1}{N})\right)}
d\alpha d\beta.
\end{multline*}
Hence, by reformulate as before, we get
\[
J_{N-1}(D_{p,r})
\underset{N\to \infty}{\sim}
\iint_D
C_N^{\prime\prime} e^{\frac{N}{2\pi i} \big(\pm 4\pi^2 \alpha \pm 4 \pi^2 \beta +\Phi_{D_{p, r}}(\alpha, \beta)\big)} d\alpha d \beta. 
\]
Every choice of the signature  gives the same asymptotics.  
\subsection{Saddle point method}
Here we investigate the integral
\begin{equation}
\iint_D e^{\frac{N}{2\pi i} \big(-4\pi^2 \alpha - 4 \pi^2 \beta +\Phi_{D_{p, r}}(\alpha, \beta)\big)}d\alpha d\beta
\label{eq:integral}
\end{equation}
where $D = [0, 1]^2$.  
%
%Let $(\alpha_0, \beta_0)$ be a solution of \eqref{eq:saddleD}.  
%
\begin{prop}
Let $p$, $r$ be  integers satisfying $p, r\geq 2$ and $p+r\geq 8$, or  
 $p, -r \geq 3$ and $p-r \geq 9$.  
    The asymptotics of the following integral is given by its value at the saddle point as follows.  
\[
\iint_D e^{\frac{N}{2\pi i}  \big(-4\pi^2 \alpha - 4 \pi^2 \beta +\Phi_{D_{p, r}}(\alpha, \beta)\big)}
d\alpha d\beta
\underset{N\to \infty}{\longrightarrow}
e^{\frac{N}{2\pi i} \big(- 4\pi^2 \alpha_0 - 4 \pi^2 \beta_0 +\Phi_{D_{p, r}}(\alpha_0, \beta_0)\big)},
\]  
where $(\alpha_0$, $\beta_0)$ is the solution of 
\begin{equation}
\begin{aligned}
\frac{\partial}{\partial \alpha}
\Big(
- 4\pi^2 \alpha - 4 \pi^2 \beta +\Phi_{D_{p, r}}(\alpha, \beta)\Big)&=0,
\\
\frac{\partial}{\partial \beta}\Big(
- 4\pi^2 \alpha - 4 \pi^2 \beta +\Phi_{D_{p, r}}(\alpha, \beta)\Big)&=0.
\end{aligned}
\label{eq:saddleD}
\end{equation}
This system of equations is called the saddle point equation.  
\label{prop:saddle}
\end{prop}
\begin{proof}
Let $v_{D_{p,r}}$ be the hyperbolic volume of the complement of $D_{p,r}$.  
Then we can push the integral region inside the contour of 
$\mathrm{Im}\big(- 4\pi^2 \alpha_0 - 4 \pi^2 \beta_0 +\Phi_{D_{p, r}}(\alpha_0, \beta_0)\big) = v_{D_{p,r}}$ 
to the saddle point as in Figure \ref{fig:saddle} for $D_{6,2}$ and Figure \ref{fig:deform} for $D_{5,3}$, $D{4, 4}$, $D_{6, -3}$, and $D_{5, -4}$.
\begin{figure}[htb]
{\normalsize
\[
\hspace{-5mm}
\begin{matrix}
\includegraphics[scale=0.7]{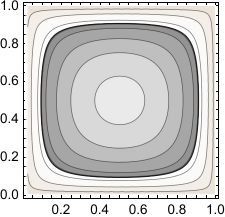}
\\
\alpha = x - i 0 \\
\beta = y - i 0 
\end{matrix}
\qquad
\begin{matrix}
\includegraphics[scale=0.7]{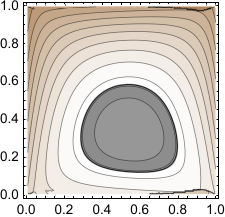}
\\
\alpha = x - 0.08i \\
\beta = y - 0.006 i 
\end{matrix}
\qquad
\begin{matrix}
\includegraphics[scale=0.7]{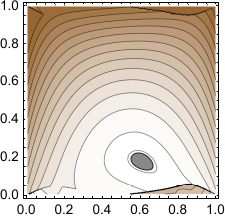}
\\
\alpha = x - 0.15i \\
\beta = y - 0.01  i 
\end{matrix}
\qquad
\begin{matrix}
\includegraphics[scale=0.7]{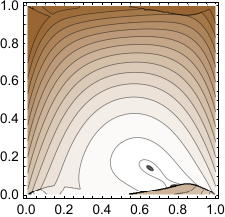}
\\
\alpha = x - 0.012i  \\    
\beta = y - 0.157i  
\end{matrix}
\]
}
\normalsize
\qquad
$D_{6,2}\qquad
\mathbb{C}^2\ \ 
\raisebox{0mm}
{\ \ $\begin{matrix}
\uparrow 
\\\mathbb{R}^2&\\
\ \ \   \searrow
\end{matrix}$}
\begin{matrix}
\includegraphics[scale=0.8]{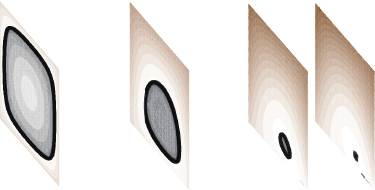}
\end{matrix}
$
\\[-2mm] \qquad\qquad\qquad
 \raisebox{0mm}{
 $
\longleftarrow (i \, \mathbb{R})^2$}
\caption{Push the integral region for $D_{6,2}$ to the imaginary direction.}
\label{fig:saddle}
\end{figure}
The contours os the boundary of the gray regions show the level indicating the hyperbolic volume of $S^3\setminus D_{p, r}$.  
Therefore, we can apply the saddle point method.  
In the figures, we see the contours of the function at  planes parallel to the real plane including the original integral region.  
In the function $- 4\pi^2 \alpha - 4 \pi^2 \beta +\Phi_{D_{p, r}}(\alpha, \beta)$, we can deform the parameters $p$ and $r$ continuously.  
For detail, see Appendix E.   
Therefore, we can also deform the integral region continuously from small $p$, $|r|$ to large $p$, $|r|$, where the saddle point converges to $\alpha = \beta = 1/2$ as in Figure \ref{fig:saddlepoints}.   
%In this deformation, the saddle point is getting closer to $(1/2, 1/2)$  and the value of $- 4\pi^2 \alpha - 4 \pi^2 \beta +\Phi_{D_{p, r}}(\alpha, \beta)$ at the saddle point increases.  
\end{proof}
\begin{figure}[htb]
\[
i\, \mathbb{R}\uparrow  \ 
\begin{matrix}
\includegraphics[scale=1.4, alt={saddle points in the complex plane}]{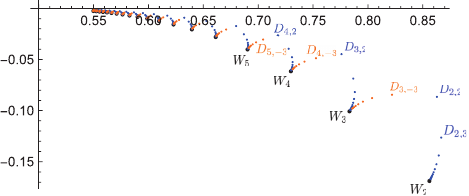}
\end{matrix}
\raisebox{2.5cm}{$\mathbb{R}$}
\]
\caption{Saddle points $\alpha$ for $D_{p, r}$ and $W_p$ with positive $p$.
Blue points are for $D_{p, r}$ with positive $r$, orange points are for negative $r$ and black points are for $W_p$ up to $p=20$.  }
\label{fig:saddlepoints}
\end{figure}
\begin{small}
\begin{figure}[htb]
\[
\begin{matrix}
D_{5,3} & 
\begin{matrix}
\includegraphics[scale=0.65]{contour0}
\\
\alpha = x - 0i \\
\beta = y - 0  i 
\end{matrix}
&
\begin{matrix}
\includegraphics[scale=0.65]{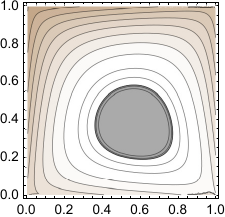}
\\
\alpha = x - 0.0125i  \\    
\beta = y - 0.04i  
\end{matrix}
&
\begin{matrix}
\includegraphics[scale=0.65]{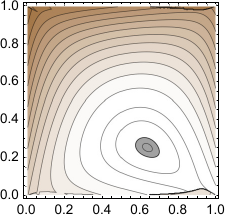}
\\
\alpha = x - 0.02i  \\    
\beta = y - 0.08i  
\end{matrix}
&
\begin{matrix}
\includegraphics[scale=0.65]{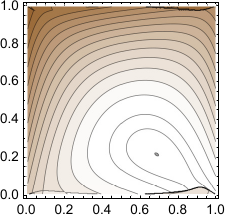}
\\
\alpha = x - 0.025i  \\    
\beta = y - 0.088i  
\end{matrix}
\\{}\\
D_{4,4}
& 
\begin{matrix}
\includegraphics[scale=0.65]{contour0}
\\
\alpha = x - 0i \\
\beta = y - 0  i 
\end{matrix}
&
\begin{matrix}
\includegraphics[scale=0.65]{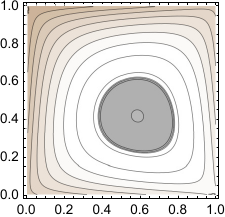}
\\
\alpha = x - 0.02i  \\    
\beta = y - 0.02i  
\end{matrix}
&
\begin{matrix}
\includegraphics[scale=0.65]{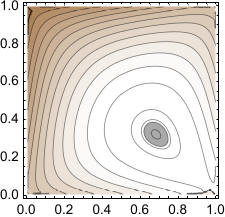}
\\
\alpha = x - 0.04i  \\    
\beta = y - 0.04i  
\end{matrix}
&
\begin{matrix}
\includegraphics[scale=0.65]{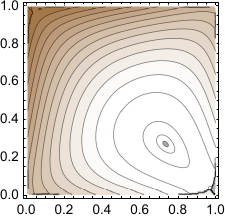}
\\
\alpha = x - 0.0477i  \\    
\beta = y - 0.0477i  
\end{matrix}
\\{}\\
D_{6, -3}
& 
\begin{matrix}
\includegraphics[scale=0.65]{contour0}
\\
\alpha = x - 0i \\
\beta = y - 0  i 
\end{matrix}
&
\begin{matrix}
\includegraphics[scale=0.65]{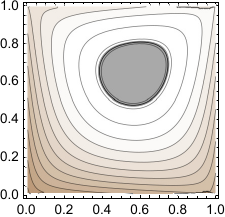}
\\
\alpha = x - 0.01i  \\    
\beta = y - 0.05i  
\end{matrix}
&
\begin{matrix}
\includegraphics[scale=0.65]{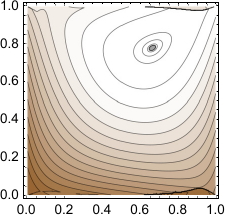}
\\
\alpha = x - 0.02i  \\    
\beta = y - 0.09i  
\end{matrix}
&
\begin{matrix}
\includegraphics[scale=0.65]{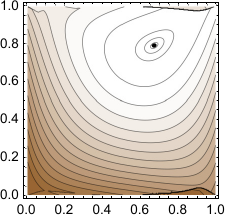}
\\
\alpha = x - 0.0206i  \\    
\beta = y - 0.0935i  
\end{matrix}
\\{}\\
D_{5, -4}
& 
\begin{matrix}
\includegraphics[scale=0.65]{contour0}
\\
\alpha = x - 0i \\
\beta = y - 0  i 
\end{matrix}
&
\begin{matrix}
\includegraphics[scale=0.65]{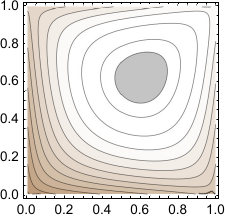}
\\
\alpha = x - 0.02i  \\    
\beta = y - 0.03i  
\end{matrix}
&
\begin{matrix}
\includegraphics[scale=0.65]{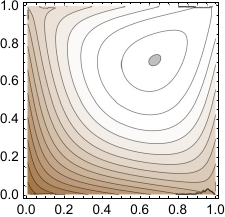}
\\
\alpha = x - 0.03i  \\    
\beta = y - 0.05i  
\end{matrix}
&
\begin{matrix}
\includegraphics[scale=0.65]{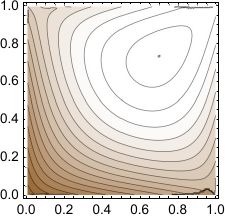}
\\
\alpha = x - 0.033i  \\    
\beta = y - 0.055i  
\end{matrix}
\end{matrix}
\]
\caption{Push the integral region for $D_{p,r}$ to the imaginary direction.}
\label{fig:deform}
\end{figure}
\end{small}
\subsection{Volume of the complement}
The potential function $\Phi_{D_{p, r}}(\frac{x}{2\pi i}, \frac{y}{2\pi i})$ satisfies 
\[
\exp\left(\frac{\partial}{\partial x}\Phi_{D_{p, r}}\Big(\frac{x}{2\pi i}, \frac{y}{2\pi i}\Big)\right) = u^p\,p_3
\]
and
\[
\exp\left(\frac{\partial}{\partial y}\Phi_{D_{p, r}}\Big(\frac{x}{2\pi i}, \frac{y}{2\pi i}\Big)\right) = v^{-r}\,{p'_3}
\]
for $p_2$, $p'_4$ in \eqref{eq:p},  \eqref{eq:pprime} and
$u = e^x$, $v=e^y$, since the actual computation shows that
\[
\exp\left(\frac{\partial}{\partial x}\Psi_B\Big(\frac{x}{2\pi i}, \frac{y}{2\pi i}\Big)\right) = p_3, 
\qquad
\exp\left(\frac{\partial}{\partial y}\Psi_B\Big(\frac{x}{2\pi i}, \frac{y}{2\pi i}\Big)\right) = p'_3.
\]
By comparing $\Phi_{D_{p, r}}\big(\frac{x}{2\pi i}, \frac{y}{2\pi i}\big)$ with the Neumann-Zagier function as in the case of the twisted Whitehead link, we get %the following.  
\begin{multline*}
\frac{1}{i}\left(\Phi_{D_{p, r}}(\alpha_0, \beta_0)- 2\pi i \Big(2 \pi i (\alpha_0 - \frac{1}{2})+2 \pi i (\beta_{0} - \frac{1}{2})\Big)\right) =
\\
 \mathrm{Vol}(S^3\setminus D_{p, r}) + i \, \mathrm{CS}(S^3\setminus D_{p, r}).
\end{multline*}
Therefore, the volume conjecture holds for $D_{p, r}$.  
The volume conjecture for the double twist knots $D_{p, r}$ with the integers $p$, $r$ excluded in Proposition \ref{prop:saddle} is already proved in  \cite{O1}, \cite{OY}, \cite{O2}.  
%
%
%
%\vfill
%\pagebreak
%{\ }
%\vfill
%\pagebreak
\bigskip
\par
\begin{center}
\large\bf{Appendices}
\end{center}
%
%\pagebreak
%
\setcounter{section}{0}
\renewcommand{\thesection}{Appendix \Alph{section}}
\section{ADO invariants for colored knotted graphs} 
\renewcommand{\thesection}{\Alph{section}}
%
%\section{Quantum invariants for knotted graphs}
 Here we recall two quantum invariants defined for colored knotted graphs,  which is also known as the quantum spin network.  
 The first one is the Kirillov-Reshetikhin invariant introduced in \cite{KR}, which is a generalization of the colored Jones polynomial, and the second one is the ADO invariant, which is also related to quantum $sl_2$ as the colored Jones polynomial, but this invariant is defined for the case that the quantum parameter $q$ is a root of unity.  
 The ADO invariant was introduced in \cite{ADO} for knots and links, and generalized to colored knotted graphs in \cite{CM}.  
 The colored Jones invariant $J_{N-1}(K)$ is equal to $(-1)^{N-1}\ADO_N(K)$, and is equal to $\ADO_N(K)$ for odd $N$, where all the components of $K$ are colored by $(N-1)/2$.
Here we compute $\ADO_{N}(K)$ instead of $J_{N-1}(K)$ to get the desired form of the invariant which fits to the investigation of the asymptitics of the invariant.  
\subsection{ADO invariant for colored knots and links}
We use the following notations.  
\begin{align*}
q^a &= \exp\big(\frac{\pi i a}{N}\big)  \ \ (a \in \mathbb{C}),
\qquad
\{a\} = q^a - q^{-a}, \ \ \{a, k\} = \prod_{j=0}^{k-1} \{a-j\},
\\
\left[\,\begin{matrix} a \\ b\end{matrix}\,\right]
&= \prod_{j=0}^{a-b-1}\frac{\{a - j\}}{\{a-b-j\}} \ \ (a-b \in \{0, 1, \ldots, N-1\}),
\\
t_a &= a(a+1-N) = (a-\frac{N-1}{2})^2 - \frac{(N-1)^2}{4}.  
\end{align*}
Let $\mathcal{U}_q(sl_2)$ be the quantum $sl_2$ at the $2N$-th root of unity $q$ and let $V_a$ be the highest weight irreducible module with the highest weight $q^a$.  
For $a \in (\mathbb{C}\setminus \mathbb{Z}/2)\cup (N\mathbb{Z}-1)/2$, $ \dim V_a = N$.   
\par
Let $K=K_1\cup K_2 \cup \cdots \cup K_\ell$ be a $\ell$ component oriented link diagram whose components are labeled by $c_1$, $\cdots$, $c_\ell$ where $c_i \in (\mathbb{C}\setminus \mathbb{Z}/2) \cup (N\mathbb{Z}-1)/2$.
The label $c_i$ is called the {\it color} of the $i$-th component $K_i$.  
Let $T_K$ be a $(1,1)$ tangle obtained by cutting the $j$-th component of $K$.  
Then, by assigning the quantum $R$ matrix to the crossings, evaluation map to the maximal points and coevaluation map to the minimal points given in \cite{CM}, we get a scalar matrix of size $N$.  
This scalar depends on the color $c_j$ for the $j$-th component, and by multiplying 
$\left[\begin{matrix} 2c_j + N \\ 2c_j + 1\end{matrix}\right]^{-1}$, 
we get the ADO invariant $\ADO_N(K^{c_1, \cdots, c_\ell})$ corresponding to the blackboard framing of $K$.  
Especially, the framings of a link diagram $K$ are all zero, them $\ADO_N(K^{c_1, \cdots, c_\ell})$ is a link invariant of $K$.  
\subsection{ADO invariant for colored knotted graphs}
\label{sec:ADO}
By introducing  operators corresponding to trivalent vertices, the ADO invariant is generalized to colored knotted graphs as in \cite{CM}. 
The ADO invariant is defined for a root of unity $q = e^{2\pi i/N}$ and the colors assigned to edges must contained in $\left(\mathbb{C} \setminus \mathbb{Z}/2\right) \cup N \mathbb{Z}/2$.  
In the following, we sometimes consider colors in $\mathbb{Z}/2$, and in such case, the corresponding invariant is considered to be a limit of the invariants with non-half-integer colors.  
Usually, such limit diverges, but sometimes it converges.  
\begin{dfn}
A coloring of a knotted graph is {\it admissible} if
the three colors $a$, $b$, $c$ of three edges around a vertex must satisfy the following condition.  
\begin{align*}
\begin{matrix}
\includegraphics[scale=0.8]{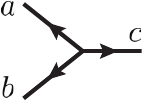}
\end{matrix}
&  \qquad
a + b+c = -2N+2, -2N+3,  \cdots, -N+1, 
\\
\begin{matrix}
\includegraphics[scale=0.8]{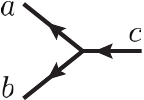}
\end{matrix}
 & \qquad
a+b-c = -N+1, -N+2, \cdots, 0, 
\\
\begin{matrix}
\includegraphics[scale=0.8]{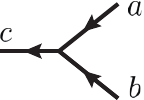}
\end{matrix}
& \qquad
a + b-c = 0, 1,  \cdots, N-1,
\\
\begin{matrix}
\includegraphics[scale=0.8]{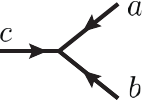}
\end{matrix}
&\qquad
a+b+c = N-1, N, \cdots, 2N-2. 
\end{align*}
\end{dfn}
In the rest, we only consider admissible colorings.  
\par
The ADO invariant for knotted graphs satisfies the following relations.  
\begin{align} 
&\mathrm{ADO}_N(\bigcirc^a) = 
\left[\begin{matrix} 2a+N\\2a+1\end{matrix}\right]^{-1} , 
\label{eq:ADOtrivial}
\\
&\mathrm{ADO}_N\left(\begin{matrix}
 \includegraphics[scale=0.8]{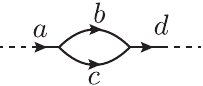}
\end{matrix}\right)
= 
 \delta_{ad}
 \left[\begin{matrix} 2a+N\\2a+1\end{matrix}\right] \ 
 \mathrm{ADO}_N\left(\begin{matrix} \includegraphics[scale=0.8]{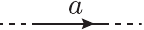}
\end{matrix}\right),
\label{eq:ADOtheta}
\\
&\mathrm{ADO}_N\left(\begin{matrix}
 \includegraphics[scale=0.8]{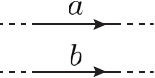}
\end{matrix}\right)
= 
\label{eq:ADOparallel}
\\
&\qquad\qquad
\sum_{a+b-c = 0, 1, \cdots, N-1}
\left[\begin{matrix} 2c+N\\2c+1\end{matrix}\right]^{-1} \,
\mathrm{ADO}_N\left(\begin{matrix}
 \includegraphics[scale=0.8]{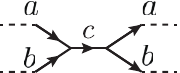}
\end{matrix}\right), 
\notag
\\
&\mathrm{ADO}_N\left(\begin{matrix}
 \includegraphics[scale=0.8]{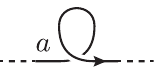}
\end{matrix}\right)
= q^{2t_a}\,
\mathrm{ADO}_N\left(\begin{matrix}
 \includegraphics[scale=0.8]{bubble0}
\end{matrix}\right)
, 
\label{eq:ADOtwistp}
\\
&\mathrm{ADO}_N\left(\begin{matrix}
 \includegraphics[scale=0.8]{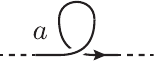}
\end{matrix}\right)
 = q^{-2t_a}
 \mathrm{ADO}_N\left(\begin{matrix}
 \includegraphics[scale=0.8]{bubble0}
\end{matrix}\right),
\label{eq:ADOtwistn}
\\
&\mathrm{ADO}_N\left(\begin{matrix}
 \includegraphics[scale=0.8]{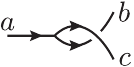}
\end{matrix}\right)
= q^{t_a - t_b -t_c}
\mathrm{ADO}_N\left(\begin{matrix}
 \includegraphics[scale=0.8]{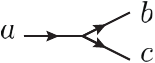}
\end{matrix}\right)
, 
\label{eq:ADOtwist3p}
\\
&\mathrm{ADO}_N\left(\begin{matrix}
 \includegraphics[scale=0.8]{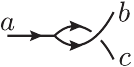}
\end{matrix}\right)
= q^{-(t_a-t_b-t_c)}
\mathrm{ADO}_N\left(\begin{matrix}
 \includegraphics[scale=0.8]{twist0}
\end{matrix}\right)
,
\label{eq:ADOtwist3n}
\\
&\mathrm{ADO}_N\left(\begin{matrix}
 \includegraphics[scale=0.8]{bubble0}
\end{matrix}\right)
= 
\mathrm{ADO}_N\left(\begin{matrix}
 \includegraphics[scale=0.8]{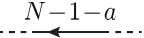}
\end{matrix}\right)  \quad
\text{(dual representation)}.
\label{eq:ADOinverse}
\end{align}
By using the above relations, we get the following relation.    
\begin{lem}
We can remove a circle around an edge as follows.
\begin{multline}
\mathrm{ADO}_N\left(\begin{matrix}
 \includegraphics[scale=0.8]{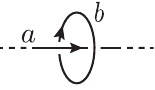}
\end{matrix}\right)
= 
\\
i^{N-1} q^{(2a+1-N)(2b+1-N)} \{ 2a+N, N-1\}\,
\mathrm{ADO}_N\left(\begin{matrix}
 \includegraphics[scale=0.8]{bubble0}
\end{matrix}\right).
\label{eq:hopf}
\end{multline}
\end{lem}
\begin{proof}
The lefthand side of the formula is computed as follows.  
\begin{align*}
&\mathrm{ADO}_N\left(\begin{matrix}
 \includegraphics[scale=0.8]{hopfab}
\end{matrix}\right)
\underset{\eqref{eq:ADOparallel}}{=} 
\\
&\sum_{a+b-c = 0, 1, \cdots, N-1}
\left[\begin{matrix} 2c+N\\2c+1\end{matrix}\right]^{-1} \,
\mathrm{ADO}_N\left(\begin{matrix}
 \includegraphics[scale=0.8]{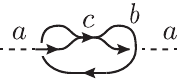}
 \end{matrix}\right)
\\
\underset{\eqref{eq:ADOtwist3p}}{=} &
\sum_{k = 0}^{N-1}
q^{2(t_{a+b-k} - t_a - t_b)}
\left[\begin{matrix} 2(a+b-k)+N\\2(a+b-k)+1\end{matrix}\right]^{-1} \,
\mathrm{ADO}_N\left(\begin{matrix}
 \includegraphics[scale=0.8]{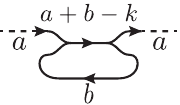}
\end{matrix}\right)\hfill
\\
=&
\frac{\{N-1\}! \, q^{-t_a-t_b}}{\{2(a+b)+N\}\cdots\{2(a+b)+1\}}
\times\hfill
\\& \qquad\qquad
\sum_{k = 0}^{N-1}
\{2(a+b-k)+1\} \, q^{t_{a+b-k}}
\mathrm{ADO}_N\left(\begin{matrix}
 \includegraphics[scale=0.8]{hopfab2}
\end{matrix}\right)
\\
\underset{\eqref{eq:ADOtheta}}{=} &
\frac{\{N-1\}! \, q^{-t_a-t_b}}{\{2(a+b)+N\}\cdots\{2(a+b)+1\}}
\left[\begin{matrix} 2a+N\\2a+1\end{matrix}\right] 
\times\hfill
\\&\qquad\qquad
\sum_{k = 0}^{N-1}
\{2(a+b-k)+1\} \, q^{t_{a+b-k}}
\mathrm{ADO}_N\left(\begin{matrix}
 \includegraphics[scale=0.8]{bubble0}
\end{matrix}\right) 
\\
{=} &
\frac{q^{-t_a-t_b}\, \{2a+N, N-1\}}{\{2(a+b)+N, N\}}\,
\times\hfill
\\ &\qquad
\sum_{k = 0}^{N-1}
\{2(a+b-k)+1\} \, q^{t_{a+b-k}}
\mathrm{ADO}_N\left(\begin{matrix}
 \includegraphics[scale=0.8]{bubble0}
\end{matrix}\right).  
\end{align*}
Now we compute 
\begin{align*}
&\sum_{k = 0}^{N-1}
\{2(a+b-k)+1\} \, q^{t_{a+b-k}}
\\
&=
\sum_{k = 0}^{N-1}
(q^{2(a+b-k)+1} - q^{-2(a+b-k)-1}) \, q^{2\left((a+b-k-\frac{N-1}{2})^2-\frac{(N-1)^2}{4}\right)}
\\
&=
-q^{-\frac{(N-1)^2}{2}}
\sum_{k = 0}^{N-1}
(q^{2(a+b-k)+1-N} - q^{-2(a+b-k)-1+N})  \, q^{\frac{1}{2}(2(a+b-k)+1-N)^2} \hfill
\\
&=
-q^{-\frac{(N-1)^2}{2}}
\sum_{k = 0}^{N-1}
 \left(q^{\frac{1}{2}(2(a+b-k)+2-N)^2-\frac{1}{2}}-
q^{\frac{1}{2}(2(a+b-k)-N)^2-\frac{1}{2}}\right)\hfill
\\
&=
-q^{-\frac{(N-1)^2+1}{2}}
\left(q^{2(a+b+1-N+\frac{N}{2})^2}-
q^{2(a+b+1-N-\frac{N}{2})^2}\right)\hfill
\\
&=
-q^{-\frac{(N-1)^2+1}{2}}
\,\times
\\
&\qquad
\left(q^{2\left((a+b+1-N)^2 + N(a+b+1-N)+\frac{N^2}{4}\right)}-
q^{2\left((a+b+1-N)^2 - N(a+b+1-N)+\frac{N^2}{4}\right)}
\right)
\\
&=
q^{2a^2+2b^2+4ab+4a+4b+1-4Na-4Nb}
\{2N(a+b)\}
.  
\end{align*}
Since 
\begin{multline*}
2a^2+2b^2+4ab+4a+4b+1-4Na-4Nb
- 2t_a -2t_b
=
\\
(2a +1-N)(2b+1-N) -N^2
\end{multline*}
and
\[
\{2(a+b)+N, N\} 
=
-i^{N-1}\{2N(a+b)\}, 
\]
we have
\begin{multline*}
\frac{q^{-t_a-t_b}\, \{2a+N, N-1\}}{\{2(a+b)+N, N\}}
%\times\hfill
%\\ \hfill
\sum_{k = 0}^{N-1}
\{2(a+b-k)+1\} \, q^{t_{a+b-k}}
\mathrm{ADO}_N\left(\begin{matrix}
 \includegraphics[scale=0.8]{bubble0}
\end{matrix}\right)
\\
=
q^{(2a+1-N)(2b+1-N) - N^2} 
\{2a+N, N-1\} 
\frac{\{2N(a+b)\}}{-i^{N-1}\{2N(a+b)\}}
\\
=
\frac{(-1)^{N-1}}{i^{N-1}}
q^{(2a+1-N)(2b+1-N)} 
\{2a+N, N-1\} 
\\
=
i^{N-1}
q^{(2a+1-N)(2b+1-N)} 
\{2a+N, N-1\}, 
\end{multline*}
and we get \eqref{eq:hopf}.  
\end{proof}
\subsection{Quantum $6j$ symbol}
The quantum $6j$ symbol of the ADO invariant is 
the ADO invariant for the tetrahedral graph labeled as in Figure \ref{fig:ADOtet}.  
\begin{figure}[htb]
\includegraphics[scale=0.8, alt={tetrahedral graph}]{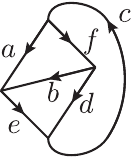}
\caption{The oriented tetrahedral graph labeled by $a$, $b$, $c$, $d$, $e$, $f$.}
\label{fig:ADOtet}
\end{figure}
The quantum $6j$ symbol $\left\{
\begin{matrix} a & b & e \\ d & c & f
\end{matrix}\right\}_q$ is given in \cite{CM} as follows.  
Let
\[
A_{xyz} = x+ y + z, \qquad B_{xyz} = x + y - z.  
\]
Then
\begin{multline}
\left\{
\begin{matrix} a & b & e \\ d & c & f
\end{matrix}\right\}_q
=
(-1)^{N-1}
%\left[\,\begin{matrix} 2f+N \\ 2f+1\end{matrix}
%\,\right]^{-1}
\frac{\{B_{dec}\}!\{B_{abe}\}!}{\{B_{bdf}\}!\{B_{afc}\}!}
\left[\,\begin{matrix} 2e \\ A_{abe}+1-N\end{matrix}
\right]
\left[\,\begin{matrix} 2e \\ B_{ced}\end{matrix}
\,\right]^{-1}\,\times
\\
\sum_{s = \max(0, -B_{bdf} + B_{dec})}^{\min(B_{dec}, B_{afc})}
\left[\,\begin{matrix} A_{acf}+1-N \\ 2c+s+1-N\end{matrix}
\right]
\left[\,\begin{matrix} B_{acf}+s \\ B_{acf}\end{matrix}
\,\right]
\,\times
\\
\hfill
\left[\,\begin{matrix} B_{bfd}+B_{dec}-s \\ B_{bfd}\end{matrix}
\,\right]
\left[\,\begin{matrix} B_{cde}+s \\ B_{dfb}\end{matrix}
\,\right].  
\label{eq:ADO6j}
\end{multline}
\begin{lem}
By using the quantum $6j$ symbol, we can remove a triangle in the colored knotted graph as follows.  
\begin{equation}
\ADO_N\left(
\begin{matrix}
\includegraphics[scale=0.8]{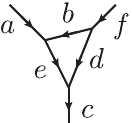}
\end{matrix}
\right)
=
\left\{
\begin{matrix} a & b & e \\ d & c & f
\end{matrix}\right\}_q
\ADO_N\left(
\begin{matrix}
\includegraphics[scale=0.8]{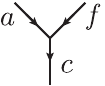}
\end{matrix}
\right).
\label{eq:removetriangle}
\end{equation}
\begin{equation}
\ADO_N\left(
\begin{matrix}
\includegraphics[scale=0.8]{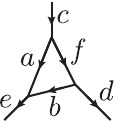}
\end{matrix}
\right)
=
\left\{
\begin{matrix} a & b & e \\ d & c & f
\end{matrix}\right\}_q
\ADO_N\left(
\begin{matrix}
\includegraphics[scale=0.8]{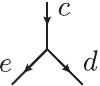}
\end{matrix}
\right).
\label{eq:removetriangle2}
\end{equation}
\end{lem}
\begin{proof}
The above two relations comes from the following  formulas.  
\[
\ADO_N\left(
\begin{matrix}
\includegraphics[scale=0.8]{alexandertet}
\end{matrix}
\right)
=
\left\{
\begin{matrix} a & b & e \\ d & c & f
\end{matrix}\right\}_q,
\]
\[
\ADO_N\left(
\begin{matrix}
\includegraphics[scale=0.8]{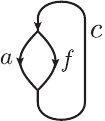}
\end{matrix}
\right)
=
\ADO_N\left(
\begin{matrix}
\includegraphics[scale=0.8]{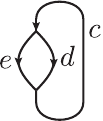}
\end{matrix}
\right)
=1.  
\]
The second formula comes from \eqref{eq:ADOtrivial} and \eqref{eq:ADOtheta}.  
\end{proof}
\begin{lem}
The ADO invariant of  the colored tetrahedral graph given in Figure \ref{fig:ADOtet} with colors
\[
a = \frac{N-1}{2}, \ 
b = \frac{N-1}{2} , \ 
c = \frac{N-1}{2}+\varepsilon ,\ 
d = \frac{N-1}{2} +\varepsilon, \ 
e = l + \varepsilon , \ 
f = k +\varepsilon
\] 
is the following.  
\begin{multline}
\left\{\begin{matrix} 
\frac{N-1}{2} &  \frac{N-1}{2}  & l
\\  
\frac{N-1}{2}+\varepsilon  &  \frac{N-1}{2}+\varepsilon & k+\varepsilon
\end{matrix}\right\}_q
=
\frac{\{N-1 + 2\varepsilon, N-1\}}{\{N-1\}!}
\, \times
\\
\sum_{s = \max(k,l)}^{\min(k+l, N-1)}
\!\!\!\!
\text{\footnotesize$\frac{ \{s\}!^2}
{\{s-k\}!\{s-l\}!^2\{k+l-s\}!
\{s-k-2\varepsilon, s-k\}
\{k\!+\!l\!-\!s\!+\!2\varepsilon, k\!+\!l\!-\!s\}}$}
.  
\label{eq:ADO6j1}
\end{multline}
Especially, if $\varepsilon= 0$, then we have
\begin{equation}
\left\{\begin{matrix} 
\frac{N-1}{2} &  \frac{N-1}{2}  & l
\\  
\frac{N-1}{2}  &  \frac{N-1}{2} & k
\end{matrix}\right\}_q
=
\sum_{s = \max(0, l+k-N+1)}^{\min(k, l)}
\frac{\{s\}!^2}
{\{s-k\}!^2 \{s-l\}!^2 \{k+l-s\}!^2}.  
\label{eq:ADO6j2}
\end{equation}
Moreover, we have the following.  
\begin{multline}
\left\{\begin{matrix} 
\frac{N-1}{2}+\delta &  \frac{N-1}{2}+ \varepsilon  & l+ \varepsilon+\delta
\\  
\frac{N-1}{2}- \delta  &  \frac{N-1}{2}+\varepsilon & k+\varepsilon - \delta
\end{matrix}\right\}_q
=
\\\hfill
\frac{ \{N-1-2\delta, N-1\}} {\{N-1\}!}
 \sum_{s = \max(k,l)}^{\min(k+l, N-1)}
\frac{ \{s\}!}
{\{s-k\}!\{s-l\}!\{k+l-s\}!}
\, \times \hfill
\\
\frac{\{s+2\varepsilon, s\}}{\{s-k+2\delta, s-k\}\{s-l-2\delta, s-l\}
\{k+l-s+2\varepsilon, k+l-s\}}
,  
\label{eq:ADO6j3}
\end{multline}
\begin{align}
\left\{\begin{matrix} 
l  &  \frac{N-1}{2}  & \frac{N-1}{2}
\\[4pt]
\varepsilon  &   \frac{N-1}{2}+\varepsilon &  \frac{N-1}{2}+\varepsilon 
\end{matrix}\right\}_q
&\!\!=
\frac{\{l+2\varepsilon, l\}}{\{l\}!}, %\hfill
\label{eq:ADO6j4}
%\end{equation}
\\
%
%\begin{equation}
\left\{\begin{matrix} 
l  &  \frac{N-1}{2}-\frac{\varepsilon}{2}   & \frac{N-1}{2}-\frac{\varepsilon}{2} 
\\[4pt]
\varepsilon  &   \frac{N-1}{2}+\frac{\varepsilon}{2} &  \frac{N-1}{2}+\frac{\varepsilon}{2} 
\end{matrix}\right\}_q
&\!\!=
\frac{\{N-1-\varepsilon, N-1\}}{\{N-1\}!},
\label{eq:ADO6j6}
%\end{equation}
\\
%
%\begin{equation}
\left\{\begin{matrix} 
-\varepsilon  &  \frac{N-1}{2}  & \frac{N-1}{2}\!-\! \varepsilon
\\[4pt]
N\!-\!1\!-\!l\! +\! \varepsilon  &   \frac{N-1}{2} &  \frac{N-1}{2}\!+\!\varepsilon 
\end{matrix}\right\}_q
&\!\!=
\frac{\{l-2\varepsilon, l\}}{\{l\}!},\qquad\quad
\label{eq:ADO6j5}
%\end{equation}
\\
%
%\begin{equation}
\left\{\begin{matrix} 
\frac{N-1}{2}\!-\!\frac{\varepsilon+\delta}{2}  &  \frac{N-1}{2} \!+\!\frac{\varepsilon-\delta}{2}  & -\delta 
\\[4pt]
\frac{N-1}{2}\!+\! \frac{\varepsilon+\delta}{2}  &   \frac{N-1}{2}\!+\! \frac{\varepsilon-\delta}{2} &  k\!+\!\varepsilon 
\end{matrix}\right\}_q
&\!\!=
\text{\footnotesize$\frac{\{k+\varepsilon-\delta, k\}\{N-1+\varepsilon+\delta, N-1\}}{\{k+\varepsilon+\delta, k\}!\{N-1\}!}$},  
\label{eq:ADO6j7}
%\end{equation}
\\
%
%\begin{equation}
\left\{\begin{matrix} 
l\!+\!\delta  &  \frac{N-1}{2}\! -\!\frac{\varepsilon+\delta}{2}  & \frac{N-1}{2}\! -\!\frac{\varepsilon-\delta}{2} 
\\[4pt]
\varepsilon  &   \frac{N-1}{2}\!+\! \frac{\varepsilon+\delta}{2} &   \frac{N-1}{2}\!+\! \frac{\varepsilon-\delta}{2}
\end{matrix}\right\}_q
&\!\!=
\frac{\{l+\varepsilon+\delta, l\}}{\{l-\varepsilon+\delta, l\}}.
\qquad\quad
\label{eq:ADO6j8}
\end{align}
\end{lem}
\begin{proof}
First we prove \eqref{eq:ADO6j1}. 
We have
$B_{dec} = l$, $B_{abe} =N-1-l$, $ B_{bdf}  = N-1-k$,
 $B_{afc} = k$ and
\begin{multline*}
\left\{\begin{matrix} 
\frac{N-1}{2} &  \frac{N-1}{2}  & l
\\  
\frac{N-1}{2}+\varepsilon  &  \frac{N-1}{2}+\varepsilon & k+\varepsilon
\end{matrix}\right\}_q
=
\\
(-1)^{N-1}
\frac{\{{l}\}!\{{N-1-l}\}!}{\{ {N-1-k}\}!\{{k}\}!}
\left[\,\begin{matrix} 2l  \\ l\end{matrix}
\right]
\left[\,\begin{matrix} 2l  \\ l\end{matrix}
\,\right]^{-1}\,\times
\\
\sum_{s = \max(0, l+k-N+1)}^{\min(k, l)}
\left[\,\begin{matrix} k+2\varepsilon \\ 
s+2\varepsilon\end{matrix}
\,\right]
\left[\,\begin{matrix} N-1-k+s  \\ 
N-1-k
\end{matrix}
\,\right]
\left[\,\begin{matrix} k+l-s  \\ k\end{matrix}
\,\right]
\, \times
\\
\hfill
\left[\,\begin{matrix} N-1-l+s+ 2 \varepsilon \\ k+ 2 \varepsilon
\end{matrix}
\,\right]
\\
=
\sum_{s = \max(0, l+k-N+1)}^{\min(k, l)}
\left[\,\begin{matrix} k+2\varepsilon \\ 
s+2\varepsilon\end{matrix}
\,\right]
\left[\,\begin{matrix} N-1-k+s  \\ 
N-1-k
\end{matrix}
\,\right]
\left[\,\begin{matrix} k+l-s  \\ k\end{matrix}
\,\right]
\, \times\hfill
\\
\hfill
\left[\,\begin{matrix} N-1-l+s+ 2 \varepsilon \\ k+ 2 \varepsilon
\end{matrix}
\,\right]
.  
\end{multline*}
By replacing 
$s$ to $k+l-s$, we get
\begin{multline*}
\left\{\begin{matrix} 
\frac{N-1}{2} &  \frac{N-1}{2}  & l
\\  
\frac{N-1}{2}+\varepsilon  &  \frac{N-1}{2}+\varepsilon & k+\varepsilon
\end{matrix}\right\}_q
=
\\
%(-1)^{N-1}\!\!\!\!
\sum_{s = \max(k,l)}^{\min(k+l, N-1)}
\left[\,\begin{matrix} k+2\varepsilon  \\ 
k+l -s+2\varepsilon\end{matrix}
\,\right]
\left[\,\begin{matrix} N-1+l  -s\\ 
N-1-k
\end{matrix}
\,\right]
\left[\,\begin{matrix} s \\ k\end{matrix}
\,\right]
\left[\,\begin{matrix} N-1+k -s+2\varepsilon \\ 
k+2\varepsilon
\end{matrix}
\,\right]
\\[6pt]
=
\frac{\{k+2\varepsilon, k\}}
{\{N-1-k\}! \{k\}! \{k+2\varepsilon, k\}} \, \times\hfill
\\
\sum_{s = \max(k,l)}^{\min(k+l, N-1)}
\!\!\!\!
\text{\small$\frac{\{N-1+l-s\}! \{s\}!
\{N-1+k-s+ 2 \varepsilon, N-1+k-s\}}
{\{s-l\}!\{k+l-s+2\varepsilon, k+l-s\} \{k+l-s\}!\{s-k\}!\{N-1-s\}!}$}
\\[6pt]
=
%(-1)^{N-1}
\frac{\{N-1 + 2\varepsilon, N-1\}}{\{N-1\}!}
 \sum_{s = \max(k,l)}^{\min(k+l, N-1)}
\frac{ \{s\}!^2}
{ \{s-k\}!\{s-l\}!^2\{k+l-s\}!}
\, \times \hfill
\\
\frac{1}
{\{s-k-2\varepsilon, s-k\}
\{k+l-s+2\varepsilon, k+l-s\}}
.  
%\label{eq:ADO6j2}
\end{multline*}
%\end{proof}
%
%
%\begin{lem}
%The ADO invariant of  the colored tetrahedral graph given in Figure \ref{fig:ADOtet} with colors
%\[
%a = \frac{N-1}{2}+\delta , \ 
%b = \frac{N-1}{2} + \varepsilon, \ 
%c = \frac{N-1}{2}+\varepsilon ,\ 
%d = \frac{N-1}{2}  -  \delta, \ 
%e = l + \varepsilon+ \delta , \ 
%f = k + \varepsilon - \delta
%\] 
%is the following.  
%%
%%
%%
%\end{lem}
%%
%\begin{proof}
\par
Next, we prove \eqref{eq:ADO6j3}. 
We have
$B_{dec} = l$, $B_{abe} =N-1-l$, $ B_{bdf}  = N-1-k$,
 $B_{afc} = k$ and
\begin{multline*}
\left\{\begin{matrix} 
\frac{N-1}{2}+\delta &  \frac{N-1}{2}+ \varepsilon  & l + \varepsilon+\delta
\\  
\frac{N-1}{2}-\delta  &  \frac{N-1}{2}+\varepsilon & k+\varepsilon- \delta
\end{matrix}\right\}_q
=
\\
(-1)^{N-1}
\frac{\{{l}\}!\{{N-1-l}\}!}{\{ {N-1-k}\}!\{{k}\}!}
\left[\,\begin{matrix} 2l + 2\varepsilon+2\delta \\ l+2\varepsilon + 2\delta\end{matrix}
\right]
\left[\,\begin{matrix} 2l + 2\varepsilon+2\delta \\ l+2\varepsilon+ 2\delta\end{matrix}
\,\right]^{-1}\,\times
\\
\sum_{s = \max(0, l+k-N+1)}^{\min(k, l)}
\left[\,\begin{matrix} k+2\varepsilon \\ 
s+2\varepsilon\end{matrix}
\,\right]
\left[\,\begin{matrix} N-1-k+s + 2 \delta \\ 
N-1-k+ 2 \delta
\end{matrix}
\,\right]
\left[\,\begin{matrix} k+l-s+2\varepsilon  \\ k+2\varepsilon\end{matrix}
\,\right]
\, \times\hfill
\\
\hfill
\left[\,\begin{matrix} N-1-l+s- 2 \delta \\ k- 2 \delta
\end{matrix}
\,\right]
\\
=
%(-1)^{N-1}\!\!\!\!\!\!\!\!\!\!\!\!
\sum_{s = \max(0, l+k-N+1)}^{\min(k, l)}
\left[\,\begin{matrix} k+2\varepsilon \\ 
s+2\varepsilon\end{matrix}
\,\right]
\left[\,\begin{matrix} N-1-k+s+ 2 \delta \\ 
N-1-k+ 2 \delta
\end{matrix}
\,\right]
\left[\,\begin{matrix} k+l-s +2\varepsilon \\ k +2\varepsilon \end{matrix}
\,\right]
\, \times\hfill
\\
\hfill
\left[\,\begin{matrix} N-1-l+s- 2 \delta \\ 
k- 2 \delta
\end{matrix}
\,\right]
.  
\end{multline*}
By replacing 
$s$ to $k+l-s$, we get
\begin{multline*}
\left\{\begin{matrix} 
\frac{N-1}{2}+\delta &  \frac{N-1}{2}+ \varepsilon  & l + \varepsilon+\delta
\\  
\frac{N-1}{2}-\delta  &  \frac{N-1}{2}+\varepsilon & k+\varepsilon- \delta
\end{matrix}\right\}_q
=
\sum_{s = \max(k,l)}^{\min(k+l, N-1)}
\left[\,\begin{matrix} k+2\varepsilon  \\ 
k+l -s+2\varepsilon\end{matrix}
\,\right]
\, \times
\\
\hfill
%(-1)^{N-1}\!\!\!\!
\left[\,\begin{matrix} N-1+l  -s+2\delta\\ 
N-1-k +2\delta
\end{matrix}
\,\right]
\left[\,\begin{matrix} s+ 2\varepsilon \\ k+ 2\varepsilon\end{matrix}
\,\right]
\left[\,\begin{matrix} N-1+k -s-2\delta \\ 
k-2\delta
\end{matrix}
\,\right]
\\[6pt]
=
\frac{\{k+2\varepsilon, k\}}
{\{N-1-k+2\delta, N-1-k\}\{k + 2 \varepsilon, k\}\{k-2\delta, k\}} \, \times\hfill
\\
\sum_{s = \max(k,l)}^{\min(k+l, N-1)}
{\text {$\frac{\{N-1+l-s+2\delta, N-1+l-s\}\{s+2\varepsilon, s\}}
{\{s-l\}!\{k+l-s+2\varepsilon, k+l-s\} \{k+l-s\}!}$
}}\,\times \hfill
\\
\hfill
{\text {$\frac{
\{N-1+k-s- 2 \delta, N-1+k-s\}}
{\{s-k\}!\{N-1-s\}!}$
}}
\\[6pt]
=
\frac{ \{N-1-2\delta, N-1\}} {\{N-1\}!}
 \sum_{s = \max(k,l)}^{\min(k+l, N-1)}
\frac{ \{s\}!}
{\{N-1\}! \{s-k\}!\{s-l\}!\{k+l-s\}!}
\, \times \hfill
\\
\frac{\{s + 2\varepsilon, s\}}
{\{s-k+2\delta, s-k\}\{s-l-2\delta, s-l\}
\{k+l-s+2\varepsilon, k+l-s\}}
.  
%\label{eq:ADO6j2}
\end{multline*}
The relations \eqref{eq:ADO6j4}, \eqref{eq:ADO6j6}, \eqref{eq:ADO6j5}, \eqref{eq:ADO6j7} and \eqref{eq:ADO6j8} are proved as follows.  
\begin{multline*}
\left\{\begin{matrix} 
l  &  \frac{N-1}{2}  & \frac{N-1}{2}
\\[4pt]
\varepsilon  &   \frac{N-1}{2}+\varepsilon &  \frac{N-1}{2}+\varepsilon 
\end{matrix}\right\}_q
=
\\
\frac{\{0\}!\{l\}!}{\{0\}!\{l\}!}
\left[\begin{matrix}
N-1 \\ l
\end{matrix}\right]
\left[\begin{matrix}
N-1 \\ N-1
\end{matrix}\right]^{-1}
\left[\begin{matrix}
l + 2\varepsilon \\ 2\varepsilon
\end{matrix}\right]
\left[\begin{matrix}
N-1 \\ N-1
\end{matrix}\right]
\left[\begin{matrix}
2\varepsilon \\ 2\varepsilon
\end{matrix}\right]
%\\
=
\frac{\{l+2\varepsilon,l\}}{\{l\}!},
\end{multline*}
\begin{multline*}
\left\{\begin{matrix} 
l  &  \frac{N-1}{2}-\frac{\varepsilon}{2}  & \frac{N-1}{2}-\frac{\varepsilon}{2}
\\[4pt]
\varepsilon  &   \frac{N-1}{2}+\frac{\varepsilon}{2} &  \frac{N-1}{2}+\frac{\varepsilon}{2} 
\end{matrix}\right\}_q
=
\\
\frac{\{0\}!\{l\}!}{\{0\}!\{l\}!}
\left[\begin{matrix}
N-1-\varepsilon \\ l-\varepsilon
\end{matrix}\right]
\left[\begin{matrix}
N-1-\varepsilon \\ N-1-\varepsilon
\end{matrix}\right]^{-1}
\left[\begin{matrix}
l + \varepsilon \\ \varepsilon
\end{matrix}\right]
\left[\begin{matrix}
N-1-\varepsilon \\ N-1-\varepsilon
\end{matrix}\right]
\left[\begin{matrix}
2\varepsilon \\ 2\varepsilon
\end{matrix}\right]
\\
=
\frac{\{N-1-\varepsilon, N-1\}}{\{N-1\}!},
\end{multline*}
\begin{multline*}
\left\{\begin{matrix} 
-\varepsilon  &  \frac{N-1}{2}  & \frac{N-1}{2}-\varepsilon
\\[4pt]
N-1-l+\varepsilon  &   \frac{N-1}{2} &  \frac{N-1}{2}+\varepsilon 
\end{matrix}\right\}_q
=
\frac{\{N-1-l\}!\{0\}!}{\{N-1-l\}!\{0\}!}
\, \times
\\
\left[\begin{matrix}
N-1-2\varepsilon \\ -2\varepsilon
\end{matrix}\right]
\left[\begin{matrix}
N-1-2\varepsilon \\ l-2\varepsilon
\end{matrix}\right]^{-1}
\left[\begin{matrix}
0 \\ 0
\end{matrix}\right]
\left[\begin{matrix}
N-1 \\ l
\end{matrix}\right]
\left[\begin{matrix}
N-1-l+2\varepsilon \\ N-1-l+2\varepsilon
\end{matrix}\right]
\\
=
\frac{\{l-2\varepsilon, l\}}{\{l\}!},
\end{multline*}
\begin{multline*}
\left\{\begin{matrix} 
\frac{N-1}{2}-\frac{\varepsilon+\delta}{2}  &  \frac{N-1}{2} +\frac{\varepsilon-\delta}{2}  & -\delta 
\\[4pt]
\frac{N-1}{2}+ \frac{\varepsilon+\delta}{2}  &   \frac{N-1}{2}+ \frac{\varepsilon-\delta}{2} &  k+\varepsilon 
\end{matrix}\right\}_q
=
\\
\frac{\{0\}!\{N-1\}!}{\{N-1-k\}!\{k\}!}
\left[\begin{matrix} -2\delta \\ -2\delta\end{matrix}
\right]
\left[\begin{matrix} -2\delta \\ -2\delta\end{matrix}
\right]^{-1}
\left[\begin{matrix} k+\varepsilon-\delta \\ 
\varepsilon-\delta \end{matrix} \right]
\left[\begin{matrix} N-1+\varepsilon+\delta \\ 
k+\varepsilon+\delta\end{matrix}
\right]
\\
=
\frac{\{k+\varepsilon-\delta, k\}\{N-1+\varepsilon+\delta, N-1\}}{\{k+\varepsilon+\delta, k\}\{N-1\}!}.
\end{multline*}
\begin{multline*}
\left\{\begin{matrix} 
l+\delta  &  \frac{N-1}{2} -\frac{\varepsilon+\delta}{2}  & \frac{N-1}{2} -\frac{\varepsilon-\delta}{2} 
\\[4pt]
\varepsilon  &   \frac{N-1}{2}+ \frac{\varepsilon+\delta}{2} &   \frac{N-1}{2}+ \frac{\varepsilon-\delta}{2}
\end{matrix}\right\}_q
=
\\
\frac{\{0\}!\{l\}!}{\{0\}!\{l\}!}
\left[\begin{matrix} N-1-\varepsilon+\delta \\ 
l-\varepsilon+\delta \end{matrix}
\right]
\left[\begin{matrix} N-1-\varepsilon+\delta \\ 
-\varepsilon+\delta\end{matrix}
\right]^{-1}
\left[\begin{matrix} l+\varepsilon+\delta \\ 
\varepsilon+\delta \end{matrix} \right]
\left[\begin{matrix} 2\varepsilon \\ 
2\varepsilon\end{matrix}
\right]
\\
=
\frac{\{l+\varepsilon+\delta, l\}}{\{l-\varepsilon+\delta, l\}}.
\end{multline*}
\end{proof}
\subsection{Symmetry}
Here we introduce the notion of symmetry for a function defined on the set $\{0, 1, 2, \cdots, N-1\}$.  
\begin{dfn}
A function $f$  defined on $\{0, 1, 2, \cdots, N-1\}$ is 
called {\it symmetric} if $f(k) = f(N-1-k)$, and is called {\it anti-symmetric} if $f(k) = - f(N-1-k)$.  
\end{dfn}
\begin{lem}
Let 
\begin{equation}
\xi_N(k, l, s) = 
\dfrac{  \{s\}!^2}
{\{s-k\}!^2 \, \{s-l\}!^2 \, \{k+l-s\}!^2 }.
\label{eq:xi}
\end{equation}
Then it satisfies
\begin{multline}
\xi_N(k, l, s) 
= 
\xi_N(N-1-k, l, N-1-s+l) 
= 
\xi_N(k, N-1-l, N-1-s+k)
\\
=
\xi_N(N-1-k, N-1-l, N-1-k-l+s).
\label{eq:klsymmetry}
\end{multline}
\label{lem:klsymmetry}
\end{lem}
\begin{proof}
We have
\begin{multline*}
\xi_N(N-k, l, N-1-s+l)
=
\dfrac{\{N-1-s+l\}!^2}
{  \{k+l-s\}!^2 \{N-1-s\}!^2 \{s-k\}!^2 }
\\
=
\dfrac{\{s\}!^2}
{ \{s-l\}!^2 \, \{k+l-s\}!^2 \, \{s-k\}!^2 }
=
\xi_N(k, l, s).
\end{multline*}
Similarly, we have
\begin{multline*}
\xi_N(k, N-l, N-1-s+k)
=
\dfrac{\{N-1-s+k\}!^2}
{\{N-1-s\}!^2  \{k+l-s\}!^2  \{s-l\}!^2 }
\\
=
\dfrac{\{s\}!^2}
{\{s-k\}!^2 \, \{k+l-s\}!^2 \, \{s-l\}!^2 }
=
\xi_N(k, l, s).
\end{multline*}
Combining these two, we get the last equality.  
\end{proof}
These relations imply the following symmetry of the quantum $6j$ symbols. 
\begin{prop}
The quantum $6j$ symbol defined by the ADO invariant satisfies the following symmetry.  
\begin{multline}
\left\{\begin{matrix} 
\frac{N-1}{2} &  \frac{N-1}{2}  & l
\\  
\frac{N-1}{2}  &  \frac{N-1}{2} & k
\end{matrix}\right\}_q
=
\left\{\begin{matrix} 
\frac{N-1}{2} &  \frac{N-1}{2}  & l
\\  
\frac{N-1}{2}  &  \frac{N-1}{2} & N-1-k
\end{matrix}\right\}_q
=
\\
\left\{\begin{matrix} 
\frac{N-1}{2} &  \frac{N-1}{2}  & N-1-l
\\  
\frac{N-1}{2}  &  \frac{N-1}{2} & k
\end{matrix}\right\}_q
=
\left\{\begin{matrix} 
\frac{N-1}{2} &  \frac{N-1}{2}  & N-1-l
\\  
\frac{N-1}{2}  &  \frac{N-1}{2} & N-1-k
\end{matrix}\right\}_q.
\label{eq:6jsymmetry}
\end{multline}
In other words, 
$\left\{\begin{matrix} 
\frac{N-1}{2} &  \frac{N-1}{2}  & l
\\  
\frac{N-1}{2}  &  \frac{N-1}{2} & k
\end{matrix}\right\}_q$
is symmetric with respect to $k$ and $l$.  
\end{prop}
\begin{proof}
We prove  the first equality.  
\begin{multline*}
\left\{\begin{matrix} 
\frac{N-1}{2} &  \frac{N-1}{2}  & l
\\  
\frac{N-1}{2}  &  \frac{N-1}{2} & k
\end{matrix}\right\}_q
=
\sum_{s=\max(k,l)}^{\min(N-1, k+l)}
\xi_N(k, l, s)
\underset{\eqref{eq:klsymmetry}}{=}
\\
\sum_{s=\max(k,l)}^{\min(N-1, k+l)}
\xi_N(N-1-k, l, N-1+l-s)
=
\sum_{s=\max(N-1-k,l)}^{\min(N-1, N-1-k+l)}
\xi_N(N-1-k, l, s)
\\
=
\left\{\begin{matrix} 
\frac{N-1}{2} &  \frac{N-1}{2}  & l
\\  
\frac{N-1}{2}  &  \frac{N-1}{2} & N-1-k
\end{matrix}\right\}_q.
\end{multline*}
The other equalities are proved similarly.  
\end{proof}
\renewcommand{\thesection}{Appendix \Alph{section}}
\section{Colored Jones invariants of some links}
\renewcommand{\thesection}{ \Alph{section}}
Here we compute the colored Jones invariant $J_{N-1}(K)$ for $K = B$, $B_1$, $B_{1,1}$, $W$, $W_P$, $T_p$ and $D_{p,r}$ given in Figure \ref{fig:borromean}.
\subsection{Colored Jones invariants and ADO invariants}
We compute $J_{N-1}(K)$ by using the ADO invariant.  
\begin{prop}
For a framed link $K$, the following holds.
\[
J_{N-1}(K)= (-1)^{N-1}\ADO_N(K^{\frac{N-1}{2}, \cdots, \frac{N-1}{2}}).
\]
\end{prop}
\begin{proof}
The invariants $J_{N-1}(K)$ and  $\ADO_N(K^{\frac{N-1}{2}, \cdots, \frac{N-1}{2}})$ are constructed from the same $R$ matrix since $J_{N-1}(K)$ is the colored Jones invariant corresponding to the $N$ dimensional representation $V^{(N)}$  of $\mathcal{U}_q(sl_2)$ at $q = e^{\pi i/N}$.  
Let $T$ be a $(1,1)$ tangle whose closure is isotopic to  $K$, then $T$ determines a scalar operator $\alpha\, \mathrm{id} : V^{(N)} \to V^{(N)}$ by assigning the $R$ matrix to each crossing of $T$ and the factor for the minimal and maximal points.  
Then $J_{N-1}(K) = \alpha$.  
On the other hand, 
\[
\ADO_{(N)}(K) = 
\left[\begin{matrix}2N-1\\N\end{matrix}\right]^{-1}\!\!\!\!\!\alpha 
=
\frac{\{2N-1\}\{2N-2\}\cdots\{N+1\}}{\{N-1\}\{N-2\}\cdots\{1\}}= (-1)^{N-1} \alpha.
\]  
Hence we have 
$
J_{N-1}(K) =(-1)^{N-1}\ADO_{N}(K)$
.
\end{proof}
In this paper, $N$ is assumed to be odd and we have
\begin{equation}
J_{N-1}(K) =ADO_{N}(K).  
\label{eq:AJ}
\end{equation}
\begin{rmk}
The knots treated in this paper is all colored by $N-1$ and their colored Jones polynomial $J_{N-1}$ and their ADO invariant $\ADO_N$ are not depend on the framings of them.  
\end{rmk}
\subsection{Borromean rings and their variants}
Here we compute the ADO invariants of the Borromean rings $B$ and its variants $B_1$, $B_{1,1}$.  
\begin{prop}
The ADO invariants of the Borromean rings $B$ and its variants $B_1$, $B_{1,1}$ are given as follows. 
\begin{equation}
J_{N-1}(B) = 
N^2\sum_{k,l =0}^{N-1}
\sum_{s=\max(k,l)}^{\min(k+l, N-1)}
\dfrac{  \{s\}!^2}
{\{s-k\}!^2 \, \{s-l\}!^2 \, \{k+l-s\}!^2 },
\label{eq:BJ}
\end{equation}
\begin{equation}
J_{N-1}(B_1) = 
N^2q^{-\frac{(N-1)^2}{2}}
\sum_{k,l=0}^{N-1}
\sum_{s=\max(k,l)}^{\min(k+l, N-1)}
\dfrac{ q^{\left(k-\frac{N-1}{2}\right)^2} \{s\}!^2}
{\{s-k\}!^2 \, \{s-l\}!^2 \, \{k+l-s\}!^2 },
\label{eq:JonesB1}
\end{equation}
\begin{equation}
J_{N-1}(B_{1,1}) = 
N^2
q^{-\frac{(N-1)^2}{2}}
\sum_{k,l=0}^{N-1}
\sum_{s=\max(k,l)}^{\min(k+l, N-1)}
\dfrac{ q^{\left(k-\frac{N-1}{2}\right)^2} 
q^{\left(l-\frac{N-1}{2}\right)^2} \{s\}!^2}
{\{s-k\}!^2 \, \{s-l\}!^2 \, \{k+l-s\}!^2 }.
\label{eq:JonesB11}
\end{equation}
\end{prop}
\begin{proof}
We compute the ADO invariants instead of the colored Jones invariant.  
For the Borromean rings $B$, $\ADO^{(N)}(B)$ is computed as follows.  
\begin{align*}
&\ADO^{(N)}(B) 
=
\ADO_N\left(
\begin{matrix}
\includegraphics[scale=0.8]{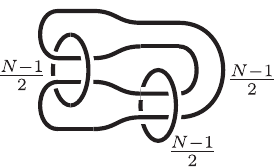}
\end{matrix}
\right)
\\
&= 
\sum_{k,l=0}^{N-1}
\left[\begin{matrix} 2k+N\\2k+1\end{matrix}\right]^{-1}
\left[\begin{matrix} 2l+N\\2l+1\end{matrix}\right]^{-1}
\ADO_N \left(
\begin{matrix}
\includegraphics[scale=0.8]{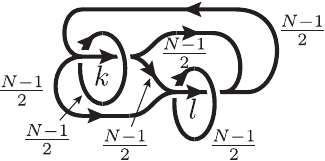}
\end{matrix}
\right)
\\
&= \ 
\sum_{k,l=0}^{N-1}
\left[\begin{matrix} 2k+N\\2k+1\end{matrix}\right]^{-1}
\left[\begin{matrix} 2l+N\\2l+1\end{matrix}\right]^{-1}
\, \times 
\\
&\qquad\qquad\qquad\qquad
\left<
\begin{matrix}
\includegraphics[scale=0.8]{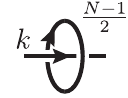}
\end{matrix}
\right>
\left<
\begin{matrix}
\includegraphics[scale=0.8]{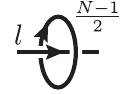}
\end{matrix}
\right>
\left<
\begin{matrix}
\includegraphics[scale=0.8]{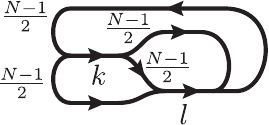}
\end{matrix}
\right>
\\
&= \sum_{k=0}^{N-1}
\left[\begin{matrix} 2k+N\\2k+1\end{matrix}\right]^{-1}
i^{N-1}  \{ 2k+N, N-1\}\, \times
\\
& \qquad\qquad\qquad\qquad
\sum_{l=0}^{N-1}
\left[\begin{matrix} 2l+N\\2l+1\end{matrix}\right]^{-1}
i^{N-1}  \{ 2l+N, N-1\} \,
%\\
\left\{\begin{matrix} 
\frac{N-1}{2}  &  \frac{N-1}{2}  & l 
\\[4pt]
\frac{N-1}{2}  &  \frac{N-1}{2}  & k 
\end{matrix}\right\}_q
\\
&=
\sum_{k,l=0}^{N-1}
\{N-1\}!^2
\sum_{s=m}^M
\dfrac{\{s\}!^2}
{\{s-k\}!^2 \, \{s-l\}!^2 \, \{k+l-s\}!^2 }
\\
&\qquad\qquad\qquad\qquad\qquad
\big(m = \max(k,l), \ \ M = \min(k+l, N-1)\big)
\\
&=
(-1)^{N-1} N^2
\sum_{k,l=0}^{N-1}
\sum_{s=m}^M
\dfrac{\{s\}!^2}
{\{s-k\}!^2 \, \{s-l\}!^2 \, \{k+l-s\}!^2 }.  
\end{align*}
%
%%
%%
%Therefore, we have
%\begin{equation}
%J_{N-1}(B) = 
%N^2\sum_{k,l =0}^{N-1}
%\sum_{s=\max(k,l)}^{\min(k+l, N-1)}
%\dfrac{  \{s\}!^2}
%{\{s-k\}!^2 \, \{s-l\}!^2 \, \{k+l-s\}!^2 }.
%\label{eq:BJ}
%\end{equation}
%
\par
For $B_1$, $\ADO(B_{1})$ is computed as follows.  
%Note that $N$ is odd.  
%
\begin{multline*}
\ADO^{(N)}(B_1)
=
\\
\sum_{k,l=0}^{N-1}
q^{(k-\frac{N-1}{2})^2 - \frac{(N-1)^2}{4}}%\, \times
%\\[-12pt]
\left[\begin{matrix} 2k+N\\2k+1\end{matrix}\right]^{-1}
\left[\begin{matrix} 2l+N\\2l+1\end{matrix}\right]^{-1}
\times \hfill
\\[-23pt]
\hfill
\ADO_N\left(
\begin{matrix}
\includegraphics[scale=0.7]{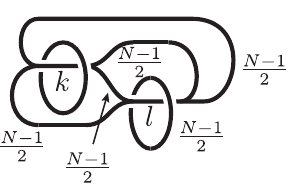}
\end{matrix}
\right)
\\=
(-1)^{N-1} N^2q^{-\frac{(N-1)^2}{4}}
\sum_{k,l=0}^{N-1}
\sum_{s=\max(k,l)}^{\min(k+l, N-1)}
\dfrac{ q^{\left(k-\frac{N-1}{2}\right)^2} \{s\}!^2}
{\{s-k\}!^2 \, \{s-l\}!^2 \, \{k+l-s\}!^2 }.  
\end{multline*}
%
%Hence
%%
%\begin{equation*}
%J_{N-1}(B_1) = 
%N^2q^{-\frac{(N-1)^2}{4}}
%\sum_{k,l=0}^{N-1}
%\sum_{s=\max(k,l)}^{\min(k+l, N-1)}
%\dfrac{ q^{\left(k-\frac{N-1}{2}\right)^2} \{s\}!^2}
%{\{s-k\}!^2 \, \{s-l\}!^2 \, \{k+l-s\}!^2 }.  
%%\label{eq:JonesB1}
%\end{equation*}
%
%Let $B_1^{(0)}$ be the $0$ framed version of $B_1$.  
%The framing of $B_1$ is $-1$, so the colored Jones invariant $J_{N-1}(B_1^{(0)})$ is
%%
%\begin{equation}
%J_{N-1}(B_1^{(0)}) = 
%N^2q^{-\frac{(N-1)^2}{2}}
%\sum_{k,l=0}^{N-1}
%\sum_{s=\max(k,l)}^{\min(k+l, N-1)}
%\dfrac{ q^{\left(k-\frac{N-1}{2}\right)^2} \{s\}!^2}
%{\{s-k\}!^2 \, \{s-l\}!^2 \, \{k+l-s\}!^2 }.  
%\label{eq:JonesB1}
%\end{equation}
%
%
\par
For $B_{1, 1}$, similar computation leads to \eqref{eq:JonesB11}.
%
%Since the framing of $B_{1,1}$ is $0$, this is the invariant for $B_{1,1}$ whose framing is $0$.  
\end{proof}
\subsection{Twisted Whitehead link}
%For the twisted Whitehead link $W_p$ with $p$ half twist, $V_{N-1}(W)$ is given by
%%
%%
%%
%\[
%V_{N-1}(W) = (-1)^{N-1}\frac{1}{[N]}\left<W^{N-1, N-1}\right> .  
%\]
%This formula has $[N]$ in the denominator, which is $0$ since $q = \exp(\pi i/N)$.  
%Of course, $\left<W^{N-1, N-1}\right>$ is divisible by $[N]$ and $V_{N-1}(W)$ is well-defined.  
%To handle this $0/0$ problem, we use the ADO invariant since
%the colored Jones polynomial $V_{N-1}(K)$ coincides with the ADO invariant $\mathrm{ADO}_{(N-1)/2}(K)$ for any knots and links as shown in \cite{MM}.  
%The merit to use the ADO invariant is that the colors assigned to the strings are continuous parameters and we can apply the l'Hopital's rule.  
%Since 
%$
%J_{N-1}(W) = \mathrm{ADO}_{(N-1)/2}(W)
%$,
For the twisted Whitehead link $W_p$, the ADO invariant $\mathrm{ADO}^{(N)}(W_p)$ is computed as follows.  
\begin{multline}
\mathrm{ADO}^{(N)}(W_p) = 
\mathrm{ADO}_{N}\left(
\begin{matrix}
\includegraphics[scale=0.6]{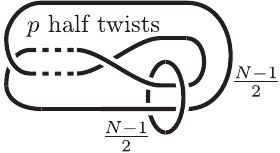}
\end{matrix}
\right) 
=
\\
\sum_{k,l=0}^{N-1}
q^{p(k-\frac{N-1}{2})^2-p\frac{(N-1)^2}{4}}
\left[\begin{matrix} 2k+N\\2k+1\end{matrix}\right]^{-1}
\left[\begin{matrix} 2l+N\\2l+1\end{matrix}\right]^{-1}
\!\!
\mathrm{ADO}_{N}\left(
\begin{matrix}
\includegraphics[scale=0.6]{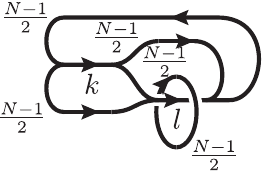}
\end{matrix}
\right)
\\
=
q^{-p\frac{(N-1)^2}{4}}
\sum_{k,l=0}^{N-1}
q^{p(k-\frac{N-1}{2})^2}
\left[\begin{matrix} 2k+N\\2k+1\end{matrix}\right]^{-1}
\left[\begin{matrix} 2l+N\\2l+1\end{matrix}\right]^{-1}
\hfill
\\
\hfill
i^{N-1}  \{ 2l+N, N-1\}\,
\left\{\begin{matrix} 
\frac{N-1}{2}  &  \frac{N-1}{2}  & l 
\\[4pt]
\frac{N-1}{2}  &  \frac{N-1}{2}  & k 
\end{matrix}\right\}_q
\\
=
q^{-p\frac{(N-1)^2}{4}}
\sum_{k,l=0}^{N-1}
q^{p(k-\frac{N-1}{2})^2}
\frac{\{N-1\}!^2\{2k+1\}}{\{2k+N, N\}}i^{N-1}
\hfill
\\
\hfill
\sum_{s=m}^M
\dfrac{\{s\}!^2}
{\{s-k\}!^2 \, \{s-l\}!^2 \, \{k+l-s\}!^2 }
\, \times\hfill
\\
\hfill\big(m = \max(k,l), \ \ M = \min(k+l, N-1)\big)
\\
=
-q^{-p\frac{(N-1)^2}{4}}
\sum_{k,l=0}^{N-1}
q^{p(k-\frac{N-1}{2})^2}
\, N^2\,\frac{\{2k+1\}}{ \{2Nk\}}\,
\times
\hfill
\\
\hfill
\sum_{s=m}^M
\dfrac{\{s\}!^2}
{\{s-k\}!^2 \, \{s-l\}!^2 \,  \{k+l-s\}!^2 }
.
\label{eq:wlimit}
\end{multline}
%
%The colored Jones invariant $J_{N-1}(W_p)$ is given in \eqref{eq:wlimit}.  
The denominator $\{2Nk\}$ of this formula is zero for integer $k$, but the numerator is also equal to zero and it must be well-defined since $J_{N-1}(W_p)$ is well-defined.  
Here we reformulate \eqref{eq:wlimit} as a limit of certain colored knotted graph.  
%We first explain for the case that the number of twist $p$ is even.  
We prepare a lemma to treat such perturbation of colors of a knotted graph.  
\begin{lem}
For  $\varepsilon\in \mathbb{C}$ near $0$, 
the following holds. 
\begin{equation}
\lim_{\varepsilon\to 0}
\ADO_N\left(
\begin{matrix}
\includegraphics[scale=0.8]{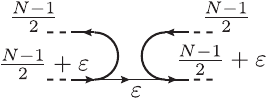}
\end{matrix}
\right)
=
\ADO_N\left(
\begin{matrix}
\includegraphics[scale=0.8]{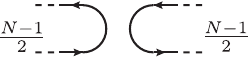}
\end{matrix}
\right).
\label{eq:small}
\end{equation}
\label{lem:small}
\end{lem}
\begin{proof}
Recall that the ADO invariant is defined by using the quantum $R$ matrix associated with the non-integral highest weight representation of $\mathcal{U}_q(sl_2)$ where $q$ is a root of unity.  
Let $V_a$ is the highest weight representation with the highest weight $a$.  
Then $\dim V_a = N$  if the weight $a$ is in $(\mathbb{C}\setminus\mathbb{Z}/2) \cup (N \mathbb{Z}-1)/2$.  
The left trivalent vertex in the lefthand side of \eqref{eq:small} represents the inclusion operator $ V_{\frac{N-1}{2}+\varepsilon} \to V_\varepsilon \otimes V_{\frac{N-1}{2}}$, and the right trivalent vertex represents the projection operator $V_\varepsilon \otimes V_{\frac{N-1}{2}} \to V_{\frac{N-1}{2}+\varepsilon}$.  
The limit 
\[
\lim_{\varepsilon\to 0} V_\varepsilon
=
V_0 = 
V^{(0)} \oplus V'
\]
where $V^{(0)}$ is the trivial $1$-dimensional representation and $V'$ is the $N-1$ dimensional representation with the highest weight $-1$.  
Then
\[
\lim_{\varepsilon\to 0} V_\varepsilon \otimes V_{\frac{N-1}{2}}
=
(V^{(0)} \oplus V') \otimes V_{\frac{N-1}{2}}
=
V_{\frac{N-1}{2}} \oplus (V' \otimes V_{\frac{N-1}{2}}),  
\]
and the above inclusion operator sends $V_{\frac{N-1}{2}}$ to $V_{\frac{N-1}{2}}$ part of $V_{\frac{N-1}{2}}\oplus (V' \otimes V_{\frac{N-1}{2}})$.  
Similarly, 
the projection operator corresponding to the right vertex picks up $V_{\frac{N-1}{2}}$ part of $V_{\frac{N-1}{2}}\oplus (V' \otimes V_{\frac{N-1}{2}})$, and discards $V' \otimes V_{\frac{N-1}{2}}$ part.  
These inclusion and projection restricted to $V_{\frac{N-1}{2}}$ are scalar operators.   
Hence, in the limiting case,  we can replace the representation $V_\varepsilon$ on the thin line by the trivial representation $V^{(0)}$, and the left diagram of \eqref{eq:small} is a scalar multiple of the right diagram.  
\par
Now we compute the scalar.  
By closing the diagrams of \eqref{eq:small} as in Figure \ref{fig:smallclose}, the lefthand side diagram is the righthand side diagram times $\left[\begin{matrix}
2N+2\varepsilon-1 \\ N\end{matrix}\right]$ by \eqref{eq:ADOtheta}, which converges to $(-1)^{N-1} = 1$ as $\varepsilon$ goes to $0$.  
Therefore, the scalar we wanted is $1$.  
\begin{figure}[htb]
\[
\begin{matrix}
\begin{matrix}
\includegraphics[scale=0.7]{small2}
\end{matrix}
& \to & 
\begin{matrix}
\includegraphics[scale=0.7]{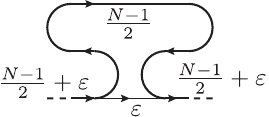}
\end{matrix}
& = & 
\text{\footnotesize$\left[\begin{matrix}
2N+2\varepsilon-1 \\ N\end{matrix}\right]
$}
\begin{matrix}
\includegraphics[scale=0.7]{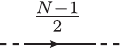}
\end{matrix}
\\
\begin{matrix}
\includegraphics[scale=0.7]{small3}
\end{matrix}
& \to & 
\begin{matrix}
\includegraphics[scale=0.7]{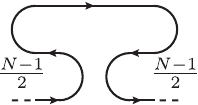}
\end{matrix}
& = & 
\begin{matrix}
\includegraphics[scale=0.8]{small30}
\end{matrix}
\end{matrix}
\]
\caption{Close the diagrams in \eqref{eq:small}.}
\label{fig:smallclose}
\end{figure}
%The image of $e_{N-1} \otimes e_0$ by the operator given by the lefthand side picture of \eqref{eq:small} is obtained as follows.
%The left vertex maps 
%\[
%e_0 \longrightarrow 
%C_{0, 0, 0}^{\frac{N-1}{2}, \varepsilon, \frac{N-1}{2}+\varepsilon}e_0 \otimes e_0
%=
%\left[\begin{matrix}
%N-1+2\varepsilon \\ 2\varepsilon
%\end{matrix}\right]e_0 \otimes e_0
%\underset{\varepsilon\to 0}{\longrightarrow}
%e_0 \otimes e_0,
%\]
%and the right vertex maps
%\begin{multline*}
%e_0 \otimes e_0\longrightarrow
%C_{N-1, N-1, N-1}^{N-1-\varepsilon, \frac{N-1}{2}, \frac{N-1}{2}-\varepsilon} e_0
%=
%\\
%q^{\frac{-N(N-1)}{2}} q^{(N-1)(\varepsilon-\frac{N-1}{2})}
%\left[\begin{matrix}
%N-1-2 \varepsilon \\ -2 \varepsilon
%\end{matrix}\right]^{-1}
%\left[\begin{matrix}
%N-1-2 \varepsilon \\ N-1-2 \varepsilon
%\end{matrix}\right]
%\\
%\sum_{z=0}^{N-1}
%(-1)^z q^{\frac{(2z-N+1)(1-2\varepsilon)}{2}}
%\left[\begin{matrix}
%N-1 \\ N-1-z
%\end{matrix}\right]
%\left[\begin{matrix}
%N-1+z-\varepsilon \\ N-1-\varepsilon
%\end{matrix}\right]
%\left[\begin{matrix}
%N-1-z \\ 0
%\end{matrix}\right]e_0
%\\
%\underset{\varepsilon\to 0}{\longrightarrow}
%(-1)^{N-1}\,e_0 = e_0,
%\end{multline*}
%since $N$ is odd.  
\end{proof}
Now we compute $J_{N-1}(W_p)$.  
\begin{prop}
For the twisted whitehead link $W_p$, $J_{N-1}(W_p)$ is given as follows.  
\begin{multline}
J_{N-1}(W_p)
=
N \frac{q^{p\frac{(N-1)^2}{4}}
}{4\pi i}
 \sum_{l=0}^{N-1}
 \frac{d}{dx}
 \left(\sum_{k=0}^{N-1}
%q^{p(k+\varepsilon-\frac{N-1}{2})^2-p\varepsilon^2}
q^{p(x-\frac{N-1}{2})^2}
\{2x+1\}
 \right.\, \times \hfill
\\
\hfill
\left.\left.
\sum_{s - \frac{x-k}{2} = \max(k,l)}^{\min(k+l, N-1)}
\!
\text{\small$\frac{ \{s, s- \frac{x-k}{2}\}^2}
{\{s\!-\!x,\! s\!-\! \frac{x+k}{2}\}^2
\{s\!-\!l, \!s\!-\!l\!-\! \frac{x-k}{2}\}^2
\{x\!+\!l\!-\!s, \!\frac{x+k}{2}\!+\!l\!-\!s_1\}^2
}$}
\right)\right|_{x=k}.
\label{eq:whiteheadADO}
\end{multline}
\end{prop}
\begin{proof}
We first compute $\ADO_N(W_p)$ for even $p$, which is the limit of the knotted graph in Figure \ref{fig:weven} at $\varepsilon\to 0$.    
\begin{figure}[htb]
\[
\begin{matrix}
\includegraphics[alt={twisted Whitehead link}]{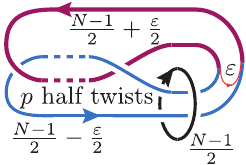}
\end{matrix}
\]
\caption{The colored knotted graph whose limit at $\varepsilon\to 0$ is $W_p$ colored by $\frac{N-1}{2}$.}
\label{fig:weven}
\end{figure}
\begin{multline*}
J_{N-1}(W_p) = \ADO^{(N)}(W_p)
\underset{\eqref{eq:small}}{=}
\lim_{\varepsilon\to 0}
\ADO_N\left(
\begin{matrix}
\includegraphics[scale=0.7]{whiteheadADOe}
\end{matrix}
\right)
 \\
 \underset{\eqref{eq:ADOparallel}, \eqref{eq:ADOinverse}}{=} 
 \lim_{\varepsilon\to 0}
\sum_{k,l=0}^{N-1}
q^{p(k+\varepsilon-\frac{N-1}{2})^2-p\varepsilon^2+p\frac{(N-1)^2}{4}}
\left[\begin{matrix} 2k+2\varepsilon+N\\2k+2\varepsilon+1\end{matrix}\right]^{-1}
\left[\begin{matrix} 2l+N\\2l+1\end{matrix}\right]^{-1}\, \times\hfill
\\ \hfill
\mathrm{ADO}_{N}\left(
\begin{matrix}
\includegraphics[scale=0.7]{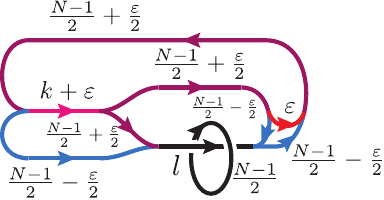}
\end{matrix}
\right) 
 \\
\underset{\eqref{eq:removetriangle2}}{= } 
 \lim_{\varepsilon\to 0}
\sum_{k,l=0}^{N-1}
q^{p(k+\varepsilon-\frac{N-1}{2})^2-p\varepsilon^2+p\frac{(N-1)^2}{4}}
\left[\begin{matrix} 2k+2\varepsilon+N\\2k+2\varepsilon+1\end{matrix}\right]^{-1}
\left[\begin{matrix} 2l+N\\2l+1\end{matrix}\right]^{-1}\, \times\hfill
\\ \hfill
\left\{\begin{matrix} 
l  &  \frac{N-1}{2}-\frac{\varepsilon}{2}  & \frac{N-1}{2} - \frac{\varepsilon}{2}
\\[4pt]
\varepsilon  &   \frac{N-1}{2}+\frac{\varepsilon}{2} &  \frac{N-1}{2}+\frac{\varepsilon}{2} 
\end{matrix}\right\}_q
\mathrm{ADO}_{N}\left(
\begin{matrix}
\includegraphics[scale=0.7]{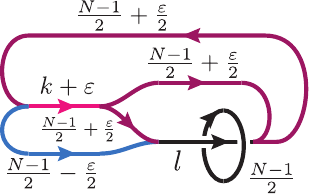}
\end{matrix}
\right) 
 \\
\underset{\eqref{eq:hopf}, \eqref{eq:ADO6j6}}{=}
q^{p\frac{(N-1)^2}{4}}
 \lim_{\varepsilon\to 0}
\sum_{k,l=0}^{N-1}
q^{p(k+\varepsilon-\frac{N-1}{2})^2-p\varepsilon^2}
\frac{\{N-1\}!^2\{2k+2\varepsilon+1\}}
{\{2k+2\varepsilon+N, N\}}i^{N-1}
\, \times\hfill
\\ \hfill
\frac{\{N-1-\varepsilon, N-1\}}{\{N-1\}!}
\left\{\begin{matrix} 
\frac{N-1}{2}-\frac{\varepsilon}{2}  &  \frac{N-1}{2}+\frac{\varepsilon}{2}  & l
\\[4pt]
\frac{N-1}{2}+\frac{\varepsilon}{2}  &   \frac{N-1}{2}+\frac{\varepsilon}{2} &  k+\varepsilon
\end{matrix}\right\}_q
 \\
\underset{\eqref{eq:ADO6j3}}{=}
N \, q^{p\frac{(N-1)^2}{4}}
 \lim_{\varepsilon\to 0}
 \sum_{l=0}^{N-1}
\frac{\{N-1-\varepsilon, N-1\}}{\{N-1\}!}
\, \times\hfill
\\ \hfill
\sum_{k=0}^{N-1}
q^{p(k+\varepsilon-\frac{N-1}{2})^2-p\varepsilon^2}
\frac{\{2k+2\varepsilon+1\}}
{\{2N(k+\varepsilon)\}}
\left\{\begin{matrix} 
\frac{N-1}{2}-\frac{\varepsilon}{2}  &  \frac{N-1}{2}+\frac{\varepsilon}{2}  & l
\\[4pt]
\frac{N-1}{2}+\frac{\varepsilon}{2}  &   \frac{N-1}{2}+\frac{\varepsilon}{2} &  k+\varepsilon
\end{matrix}\right\}_q
 \\
=
N \, q^{p\frac{(N-1)^2}{4}}
 \lim_{\varepsilon\to 0}
 \frac{1}{\{2N\varepsilon\}}
 \sum_{l=0}^{N-1}
\frac{\{N-1-\varepsilon, N-1\}}{\{N-1\}!}
\, \times\hfill
\\ \hfill
\sum_{k=0}^{N-1}
q^{p(k+\varepsilon-\frac{N-1}{2})^2-p\varepsilon^2}
\{2k+2\varepsilon+1\}
\left\{\begin{matrix} 
\frac{N-1}{2}-\frac{\varepsilon}{2}  &  \frac{N-1}{2}+\frac{\varepsilon}{2}  & l
\\[4pt]
\frac{N-1}{2}+\frac{\varepsilon}{2}  &   \frac{N-1}{2}+\frac{\varepsilon}{2} &  k+\varepsilon
\end{matrix}\right\}_q
 \\
=
N \,  \frac{q^{p\frac{(N-1)^2}{4}}}{4\pi i}
\, \times \hfill
\\
\hfill
 \frac{d}{d\varepsilon}
\left(
\frac{q^{-p\varepsilon^2}\{N-1-\varepsilon, N-1\}}{\{N-1\}!}
\right.
\, \times\hfill
\\ \hfill
\left.\left.
\sum_{k, l=0}^{N-1}
q^{p(k+\varepsilon-\frac{N-1}{2})^2}
\{2k+2\varepsilon+1\}
\left\{\begin{matrix} 
\frac{N-1}{2}-\frac{\varepsilon}{2}  &  \frac{N-1}{2}+\frac{\varepsilon}{2}  & l
\\[4pt]
\frac{N-1}{2}+\frac{\varepsilon}{2}  &   \frac{N-1}{2}+\frac{\varepsilon}{2} &  k+\varepsilon
\end{matrix}\right\}_q
\right)\right|_{\varepsilon=0}
 \\
=
N \,  \frac{q^{p\frac{(N-1)^2}{4}}
}{4\pi i}
\, \times \hfill
\\ \hfill
\left( \frac{d}{d\varepsilon}
\frac{q^{-p\varepsilon^2}\{N-1-\varepsilon, N-1\}}{\{N-1\}!}
\right)
\sum_{k,l=0}^{N-1}
q^{p(k-\frac{N-1}{2})^2}
\{2k+1\}
\left\{\begin{matrix} 
\frac{N-1}{2}  &  \frac{N-1}{2}  & l
\\[4pt]
\frac{N-1}{2}  &   \frac{N-1}{2} &  k
\end{matrix}\right\}_q \hfill
 \\
+
N  \frac{q^{p\frac{(N-1)^2}{4}}
}{4\pi i}
\, \times
\\
\hfill
\left.
 \frac{d}{d\varepsilon}
 \left(\sum_{k,l=0}^{N-1}
q^{p(k+\varepsilon-\frac{N-1}{2})^2}
\{2k+2\varepsilon+1\}
\left\{\begin{matrix} 
\frac{N-1}{2}-\frac{\varepsilon}{2}  &  \frac{N-1}{2}+\frac{\varepsilon}{2}  & l
\\[4pt]
\frac{N-1}{2}+\frac{\varepsilon}{2}  &   \frac{N-1}{2}+\frac{\varepsilon}{2} &  k+\varepsilon
\end{matrix}\right\}_q
\right)\right|_{\varepsilon=0}
\\
\underset{\eqref{eq:6jsymmetry}}{=}
N \ \frac{q^{p\frac{(N-1)^2}{4}}
}{4\pi i}\, \times \hfill
\\ 
 \sum_{l=0}^{N-1}
\left.\!\!
 \frac{d}{d\varepsilon}\!\!
 \left(\!\sum_{k=0}^{N-1}
q^{p(k+\varepsilon-\frac{N-1}{2})^2}\!
\{2k+2\varepsilon+1\}\!
\left\{\begin{matrix} 
\frac{N-1}{2}-\frac{\varepsilon}{2}  &  \frac{N-1}{2}+\frac{\varepsilon}{2}  & l
\\[4pt]
\frac{N-1}{2}+\frac{\varepsilon}{2}  &   \frac{N-1}{2}+\frac{\varepsilon}{2} &  k+\varepsilon
\end{matrix}\right\}_q
\right)\!\right|_{\varepsilon=0}
\\
\underset{\eqref{eq:ADO6j3}}{=}
N \frac{q^{p\frac{(N-1)^2}{4}}
}{4\pi i}\, \times \hfill
\\\hfill
 \sum_{l=0}^{N-1}
 \frac{d}{d\varepsilon}
 \left(\sum_{k=0}^{N-1}
\{2k+2\varepsilon+1\}
\sum_{s = \max(k,l)}^{\min(k+l, N-1)}
\frac{ q^{p(k+\varepsilon-\frac{N-1}{2})^2}
\{s\}!}{\{s-k\}!\{s-l\}!\{k+l-s\}!}
 \right.\times \hfill
\\
\hfill
\left.\left.
\frac{ \{s +\varepsilon, s\}}
{\{s-k-\varepsilon, s-k\}\{s-l+\varepsilon, s-l\}\{k+\varepsilon+l-s, k+l-s\}
}
\right)\right|_{\varepsilon=0}.
\end{multline*}
Now we replace $s$ by $s_1 + \frac{k+l}{2}$ and then use 
$\frac{d}{d\varepsilon}f(x)f(x+2\varepsilon) = \frac{d}{d\varepsilon}f(x+ \varepsilon)^2$, 
we get
\begin{multline*}
N \frac{q^{p\frac{(N-1)^2}{4}}
}{4\pi i}\!
 \sum_{l=0}^{N-1}\!
 \frac{d}{d\varepsilon}\!\!
 \left(\sum_{k=0}^{N-1}
%q^{p(k+\varepsilon-\frac{N-1}{2})^2-p\varepsilon^2}
\{2k+2\varepsilon+1\} \right.\,\times
\\
\sum_{s_1 + \frac{k+l}{2} = \max(k,l)}^{\min(k+l, N-1)}
\!\!
\frac{ q^{p(k+\varepsilon-\frac{N-1}{2})^2}
\{s_1 + \frac{k+l}{2}\}!}{\{s_1 + \frac{-k+l}{2}\}!\{s_1 + \frac{k-l}{2}\}!\{ \frac{k+l}{2}-s_1 \}!}
\times
\\
\hfill
\left.\left.
\text{\small$\frac{ \{s_1 + \frac{k+2\varepsilon+l}{2}, s_1 + \frac{k+l}{2}\}}
{\{s_1 + \frac{-k-2\varepsilon+l}{2}, s_1 + \frac{-k+l}{2}\}
\{s_1 + \frac{k + 2\varepsilon-l}{2}, s_1 + \frac{k-l}{2}\}
\{\frac{k+2\varepsilon+l}{2}-s_1, \frac{k+l}{2}-s_1\}
}$}
\right)\right|_{\varepsilon=0}
\\
=
N \frac{q^{p\frac{(N-1)^2}{4}}
}{4\pi i}
 \sum_{l=0}^{N-1}
 \frac{d}{d\varepsilon}
 \left(\sum_{k=0}^{N-1}
%q^{p(k+\varepsilon-\frac{N-1}{2})^2-p\varepsilon^2}
q^{p(k+\varepsilon-\frac{N-1}{2})^2}
\{2k+2\varepsilon+1\}
 \right.\, \times \hfill
\\
%\hfill
\left.\left.
\sum_{s_1 + \frac{k+l}{2} = \max(k,l)}^{\min(k+l, N-1)}
\!\!\!\!\!\!\!\!
\text{\scriptsize$\frac{ \{s_1 + \frac{k+\varepsilon+l}{2}, s_1 + \frac{k+l}{2}\}^2}
{\{s_1 \!+\! \frac{-k-\varepsilon+l}{2},\!s_1\! +\! \frac{-k+l}{2}\}^2
\{s_1\! +\! \frac{k + \varepsilon-l}{2},\!s_1\! +\! \frac{k-l}{2}\}^2
\{\frac{k+\varepsilon+l}{2}\!-\!s_1, \!\frac{k+l}{2}\!-\!s_1\}^2
}\!\!$}\!
\right)\!\right|_{\varepsilon=0}\!\!\!.
\end{multline*}
Then, by replacing $k+\varepsilon$ by $x$ and  $s_1$ by $s-\frac{x+l}{2}$,   we get
\begin{multline*}
J_{N-1}(W_p)
=
N \frac{q^{p\frac{(N-1)^2}{4}}
}{4\pi i}
 \sum_{l=0}^{N-1}
 \frac{d}{dx}
 \left(\sum_{k=0}^{N-1}
%q^{p(k+\varepsilon-\frac{N-1}{2})^2-p\varepsilon^2}
q^{p(k+\varepsilon-\frac{N-1}{2})^2}
\{2x+1\}
 \right.\, \times \hfill
\\
\hfill
\left.\left.
\sum_{s - \frac{x-k}{2} = \max(k,l)}^{\min(k+l, N-1)}
\text{\scriptsize$\frac{ \{s, s- \frac{x-k}{2}\}^2}
{\{s-x, s- \frac{x+k}{2}\}^2
\{s-l, s-l- \frac{x-k}{2}\}^2
\{x+l-s, \frac{x+k}{2}+l-s_1\}^2
}$}
\right)\right|_{x=k}.
\end{multline*}
Hence we obtained 
 \eqref{eq:whiteheadADO}.
\par
Next, we prove for odd case.  
For odd $p$, $W_p$ is considered as the limiting case of the knotted graph in Figure \ref{fig:wodd} at $\varepsilon=0$ by \eqref{eq:small}, and $J_{N-1}(W_p)$ is computed as follows.    
\begin{figure}[htb]
\[
\includegraphics[scale=1,alt={twisted Whitehead link}]{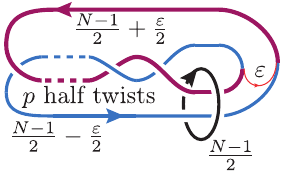}
\]
\caption{The colored knotted graph to compute $ANO_N(W_p)$ for odd $p$.}
\label{fig:wodd}
\end{figure}
\begin{multline*}
J_{N-1}(W_p) = \ADO_N(W_p)
\underset{\eqref{eq:small}}{=}
\lim_{\varepsilon\to 0}
\ADO_N\left(
\begin{matrix}
\includegraphics[scale=0.7]{whiteheadADOo}
\end{matrix}
\right)
 \underset{\eqref{eq:ADOparallel}}{=} 
 \\
 \text{\small$
 \lim_{\varepsilon\to 0}
\sum_{k,l=0}^{N-1}\!
q^{p(k+\varepsilon-\frac{N-1}{2})^2-p\varepsilon^2+p\frac{(N-1)^2}{4}}\!\!
\left[\begin{matrix} 2k+2\varepsilon+N\\2k+2\varepsilon+1\end{matrix}\right]^{-1}
\left[\begin{matrix} 2l+N\\2l+1\end{matrix}\right]^{-1}$}
\, \times\hfill
\\
\hfill
\mathrm{ADO}_{N}\!\left(
\begin{matrix}
\includegraphics[scale=0.7]{whiteheadADOo0}
\end{matrix}
\right) .
\end{multline*}
Then the rest of the computation is the same as the even $p$ case and we get \eqref{eq:whiteheadADO}.  
\end{proof}
\subsection{Twist knots and double twist knots}
Here we compute the colored Jones polynomial $J_N(D_{p, r})$ for the double twist knot $D_{p, r}$.  
Note that, if $p$ and $r$ are both odd, then $D_{p, r}$ is a two-component link.  
The following formula also holds for double twist links.  
\begin{prop}
For the double twist knot $D_{p, r}$, $J_{N-1}(D_{p, r})$ is given as follows.  
\begin{multline}
J_{N-1}(D_{p, r}) = \ADO_N(D_{p, r}) = 
\\
-\frac{N^2 \, q^{(p-r)\frac{(N-1)^2}{4}}}{16\pi^2}
\frac{\partial^2}{\partial x \partial y}
\sum_{k, l=0}^{N-1}
q^{p(x-\frac{N-1}{2})^2-r(y-\frac{N-1}{2})^2 }
\{2x+1\}\{2y+1\} \, \times
\\
\left.
 \sum_{s-\frac{x-k+y-l}{2} = \max(k,l)}^{\min(k+l, N-1)}
 \!\!\!\!\!\!\!\!\!\!\!
\text{\scriptsize$\frac{\{s, s - \frac{x-k+y-l}{2}\}^2}
{\{s -x, s - \frac{x+k+y-l}{2}\}^2
\{s -y, s - \frac{x-k+y+l}{2}\}^2
\{x\!+\!y-\!s,\! \frac{x+k+y+l}{2}\!-\!s\}^2}$}
\right|_{\text{$\begin{matrix}
x\! =\! k\\[-4pt]y\! =\! l\end{matrix}$}}\!\!\!\!.
\label{eq:double}
\end{multline}
\end{prop}
\begin{proof}
First we prove for the case that $p$ and $r$ are both even.  
For this case, we compute the ADO invariant of $D_{p, r}$ as a limit $\varepsilon, \delta \to 0$ of the knotted graph in Figure \ref{fig:doubleee}.  
\begin{figure}[htb]
\[
\includegraphics[scale=0.9, alt={double twist knot}]{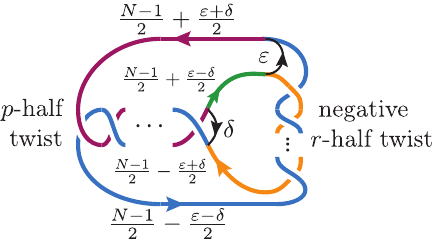}
\]
\caption{The knotted graph to compute $\ADO_{N}(D_{p, r})$ for even $p$, $r$.}
\label{fig:doubleee}
\end{figure}
\begin{multline*}
\ADO_N(D_{p,r})
=
\\
\lim_{\varepsilon, \delta\to 0}
\ADO_N\left(
\begin{matrix}
\includegraphics[scale=0.7]{doubletwistADOee}
\end{matrix}
\right)
=
q^{(p-r)\frac{(N-1)^2}{4}}
\, \times
\\
\lim_{\varepsilon, \delta\to 0}
\sum_{k, l=0}^{N-1}
q^{p(k+\varepsilon-\frac{N-1}{2})^2 - p(\frac{\varepsilon+\delta}{2})^2 - p(\frac{\varepsilon-\delta}{2})^2 }
q^{-r(l+\delta-\frac{N-1}{2})^2 + r(\frac{\varepsilon+\delta}{2})^2 + r(\frac{\varepsilon-\delta}{2})^2}
\, \times
\\
\left[\begin{matrix} 2k+2\varepsilon+N\\2k+2\varepsilon+1\end{matrix}\right]^{-1}
\left[\begin{matrix} 2l+2\delta+N\\2l+2\delta+1\end{matrix}\right]^{-1}
\left\{\begin{matrix} 
\frac{N-1}{2}-\frac{\varepsilon+\delta}{2}  &  \frac{N-1}{2} +\frac{\varepsilon-\delta}{2}  & -\delta 
\\[4pt]
\frac{N-1}{2}+ \frac{\varepsilon+\delta}{2}  &   \frac{N-1}{2}+ \frac{\varepsilon-\delta}{2} &  k+\varepsilon 
\end{matrix}\right\}_q
\, \times
\\ \hfill
\left\{\begin{matrix} 
l+\delta  &  \frac{N-1}{2} -\frac{\varepsilon+\delta}{2}  & \frac{N-1}{2} -\frac{\varepsilon-\delta}{2} 
\\[4pt]
\varepsilon  &   \frac{N-1}{2}+ \frac{\varepsilon+\delta}{2} &   \frac{N-1}{2}+ \frac{\varepsilon-\delta}{2}
\end{matrix}\right\}_q\ADO_N\left(
\begin{matrix}
\includegraphics[scale=0.7]{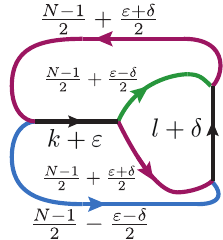}
\end{matrix}
\right)
\\
\underset{\eqref{eq:ADO6j7}, \eqref{eq:ADO6j8}}{=}
q^{(p-r)\frac{(N-1)^2}{4}}
\lim_{\varepsilon, \delta\to 0}
\sum_{k, l=0}^{N-1}
q^{p(k+\varepsilon-\frac{N-1}{2})^2-r(l+\delta-\frac{N-1}{2})^2 - (p-r)(\frac{\varepsilon^2+\delta^2}{2}) }
\, \times \hfill
\\
\text{\small$\frac{\{N-1\}!^2\{2k+2\varepsilon+1\}\{2l+2\delta+1\}}
{\{2k+2\varepsilon+N, N\}\{2l+2\delta+N, N\}}
\frac{\{k+\varepsilon-\delta, k\}\{N-1+\varepsilon+\delta, N-1\}}{\{k+\varepsilon+\delta, k\}\{N-1\}!}$}
\, \times
\\ \hfill
\frac{\{l+\varepsilon+\delta, l\}}{\{l-\varepsilon+\delta, l\}}
\ADO_N\left(
\begin{matrix}
\includegraphics[scale=0.7]{doubletwistADOee1}
\end{matrix}
\right)
\\
=
N^2\, q^{(p-r)\frac{(N-1)^2}{4}}
\lim_{\varepsilon, \delta\to 0}
\sum_{k, l=0}^{N-1}
q^{p(k+\varepsilon-\frac{N-1}{2})^2-r(l+\delta-\frac{N-1}{2})^2 - (p-r)( \frac{\varepsilon^2+\delta^2}{2})}
\, \times \hfill
\\
\frac{\{2k+2\varepsilon+1\}\{2l+2\delta+1\}}
{\{2N(k+\varepsilon)\}\{2N(l+\delta)\}}
\frac{\{k+\varepsilon-\delta, k\}\{N-1+\varepsilon+\delta, N-1\}}{\{k+\varepsilon+\delta, k\}\{N-1\}!}
\, \times
\\ \hfill
\frac{\{l+\varepsilon+\delta, l\}}{\{l-\varepsilon+\delta, l\}}
\ADO_N\left(
\begin{matrix}
\includegraphics[scale=0.7]{doubletwistADOee1}
\end{matrix}
\right)
\underset{\eqref{eq:ADO6j3}}{=}
\\
N^2 q^{(p-r)\frac{(N-1)^2}{4}}
\text{\small$
\frac{\partial^2}{\partial \varepsilon \partial \delta}
\frac{q^{ - (p-r)( \frac{\varepsilon^2+\delta^2}{2})}}
{\{2N\varepsilon\}\{2N\delta\}}
\frac{ \{N-1-2\delta, N-1\}} {\{N-1\}!} 
\frac{\{N-1+\varepsilon+\delta, N-1\}}
{\{N-1\}!}$}
\times 
\\ \hfill
\sum_{k, l=0}^{N-1}
q^{p(k+\varepsilon-\frac{N-1}{2})^2-r(l+\delta-\frac{N-1}{2})^2 }
\, \times \hfill
\\
\hfill
\frac{\{2k+2\varepsilon+1\}\{2l+2\delta+1\}}
{\{2N(k+\varepsilon)\}\{2N(l+\delta)\}}
\frac{\{k+\varepsilon-\delta, k\}}{\{k+\varepsilon+\delta, k\}}
\frac{\{l+\varepsilon+\delta, l\}}{\{l-\varepsilon+\delta, l\}}
%\, \times
%\\ \hfill
\, \times
\\ \qquad
 \sum_{s = \max(k,l)}^{\min(k+l, N-1)}
\frac{ \{s\}!}
{\{s-k\}!\{s-l\}!\{k+l-s\}!}
\, \times \hfill
\\ \hfill
\left.\frac{\{s+\varepsilon+\delta, s\}}
{\{s\!-\!k\!-\!\varepsilon\!+\!\delta, s\!-\!k\}
\{s\!-\!l\!+\!\varepsilon\!-\!\delta, s\!-\!l\}
\{k\!+\!l\!-\!s\!+\!\varepsilon\!+\!\delta, k\!+\!l\!-\!s\}}\right|_{\varepsilon=\delta=0}
\\
=
\frac{N^2 \, q^{(p-r)\frac{(N-1)^2}{4}}}{-16\pi^2}
\, \times \hfill
\\
\hfill
\frac{\partial^2}{\partial \varepsilon \partial \delta}
q^{ - (p-r)(\frac{\varepsilon^2+\delta^2}{2}) }
\frac{\{N-1+\varepsilon+\delta, N-1\}}{\{N-1\}!}
\frac{ \{N-1-2\delta, N-1\}} {\{N-1\}!}
\, \times \hfill
\\
\sum_{k, l=0}^{N-1}
q^{p(k+\varepsilon-\frac{N-1}{2})^2-r(l+\delta-\frac{N-1}{2})^2  }
\, \times
\\
\hfill
\frac{\{2k+2\varepsilon+1\}\{2l+2\delta+1\}\{k+\varepsilon-\delta, k\}\{l+\varepsilon+\delta, l\}}{\{k+\varepsilon+\delta, k\}\{l-\varepsilon+\delta, l\}}
\, \times
\\
 \sum_{s = \max(k,l)}^{\min(k+l, N-1)}
\frac{ \{s\}!}
{\{s-k\}!\{s-l\}!\{k+l-s\}!}
\, \times \hfill
\\ \hfill
\left.\text{\small$\frac{\{s+\varepsilon+\delta, s\}}{\{s-k-\varepsilon+\delta, s-k\}\{s-l+\varepsilon-\delta, s-l\}
\{k+l-s+\varepsilon+\delta, k+l-s\}}$}\right|_{\varepsilon=\delta=0}
.
\end{multline*}
By substituting $s_1 + \frac{k+l}{2}$ intto $s$, we have
\begin{multline*}
\ADO_N(D_{p, r}) = 
\\
\frac{N^2 \, q^{(p-r)\frac{(N-1)^2}{4}}}{-16\pi^2}
\, \times \hfill
\\
\hfill
\frac{\partial^2}{\partial \varepsilon \partial \delta}
q^{(r-p)(\frac{\varepsilon^2+\delta^2}{2})}
\frac{\{N-1+\varepsilon+\delta, N-1\}}{\{N-1\}!}
\frac{ \{N-1-2\delta, N-1\}} {\{N-1\}!}
\, \times
\\
\sum_{k, l=0}^{N-1}
q^{p(k+\varepsilon-\frac{N-1}{2})^2-r(l+\delta-\frac{N-1}{2})^2 }
\,\times\hfill
\\ \hfill
\frac{\{2k+2\varepsilon+1\}\{2l+2\delta+1\}\{k+\varepsilon-\delta, k\}\{l+\varepsilon+\delta, l\}}{\{k+\varepsilon+\delta, k\}\{l-\varepsilon+\delta, l\}}
\, \times \hfill
\\
 \sum_{s_1+\frac{k+l}{2} = \max(k,l)}^{\min(k+l, N-1)}
\frac{ \{s_1 + \frac{k+l}{2}\}!}
{\{s_1 + \frac{-k+l}{2}\}!\{s_1 + \frac{k-l}{2}\}!\{\frac{k+l}{2}-s_1\}!}
\, \times \hfill
\\ \hfill
\left.\text{\scriptsize$\frac{\{s_1 + \frac{k+2\varepsilon+l+2\delta}{2}, s_1 + \frac{k+l}{2}\}}{\{s_1\! +\! \frac{-k-2\varepsilon+l+2\delta}{2}, s_1\! +\! \frac{-k+l}{2}\}\{s_1\! +\! \frac{k+2\varepsilon-l-2\delta}{2}, s_1\! +\! \frac{k-l}{2}\}
\{\frac{k+2\varepsilon+l+2\delta}{2}\!-\!s_1, \frac{k+l}{2}\!-\!s_1\}}$}\right|_{\varepsilon=\delta=0}
.
\end{multline*}
Let
\begin{multline*}
f(k, l, \varepsilon, \delta)
=
q^{p(k+\varepsilon-\frac{N-1}{2})^2-r(l+\delta-\frac{N-1}{2})^2}
\{2k+2\varepsilon+1\}\{2l+2\delta+1\} \, \times
\\
 \sum_{s_1+\frac{k+l}{2} = \max(k,l)}^{\min(k+l, N-1)}
\frac{ \{s_1 + \frac{k+l}{2}\}!}
{\{s_1 + \frac{-k+l}{2}\}!\{s_1 + \frac{k-l}{2}\}!\{\frac{k+l}{2}-s_1\}!}
\, \times \hfill
\\ \hfill
\text{\footnotesize$\frac{\{s_1 + \frac{k+2\varepsilon+l+2\delta}{2}, s_1 + \frac{k+l}{2}\}}{\{s_1 + \frac{-k-2\varepsilon+l+2\delta}{2},\! s_1 + \frac{-k+l}{2}\}\{s_1 + \frac{k+2\varepsilon-l-2\delta}{2},\! s_1 + \frac{k-l}{2}\}
\{\frac{k+2\varepsilon+l+2\delta}{2}-s_1,\! \frac{k+l}{2}-s_1\}}$}
.
\end{multline*}
Note that  $f(k, l, 0, 0)$ is anti-symmetric with respect to $k$ and $l$. 
Moreover, $f(k, l, 0, \delta)$ and $f(k, l, \varepsilon, 0)$ are anti-symmetric with respect to $k$ and $l$ respectively.  
Therefore, we have
\[
\ADO_N(D_{p, r})
=
\frac{N^2 \, q^{(p-r)\frac{(N-1)^2}{4}}}{-16\pi^2}
\frac{\partial^2}{\partial \varepsilon \partial \delta}
\left.\sum_{k, l=0}^{N-1}
f(k, l, \varepsilon, \delta)\right|_{\varepsilon=\delta=0}.
\]
By using the definition of the derivation, we have
\begin{multline*}
\frac{\partial^2}{\partial \varepsilon \partial \delta}
\left.\sum_{k, l=0}^{N-1}
f(k, l, \varepsilon, \delta)\right|_{\varepsilon=\delta=0}
=
\\
\lim_{\varepsilon, \delta\to 0}
\frac{1}{\varepsilon\delta}
\sum_{k, l=0}^{N-1} f(k, l, \varepsilon, \delta)
-
f(k, l, 0, \delta) - f(k, l, \varepsilon, 0)
+
f(k, l, 0, 0)
=
\\
\lim_{\varepsilon, \delta\to 0}
\frac{1}{\varepsilon\delta}
\sum_{k, l=0}^{N-1} f(k-\varepsilon, l-\delta, \varepsilon, \delta)
-
f(k-\varepsilon, l-\delta, 0, \delta)
\\ \hfill
 - f(k-\varepsilon, l-\delta, \varepsilon, 0)
+
f(k-\varepsilon, l-\delta, 0, 0)
\\[-5pt]
=
\lim_{\varepsilon, \delta\to 0}
\frac{1}{\varepsilon\delta}
\sum_{k, l=0}^{N-1} 
f(k-\varepsilon, l-\delta, 0, 0).
\end{multline*}
Here we use that $\frac{\partial^2}{\partial \varepsilon \partial \delta}f(k, l, \varepsilon, \delta)$ is continuous with respect to $\varepsilon$ and $\delta$ for the second equality, and use that $ f(k-\varepsilon, l-\delta, \varepsilon, \delta)$, $f(k-\varepsilon, l-\delta, 0, \delta)$ and $f(k-\varepsilon, l-\delta, \varepsilon, 0)$ are anti-symmetric with respect to $k$ or $l$ for the last equality.  
\par
On the other hand, let
\begin{multline*}
g(k, l, \varepsilon, \delta)
=
q^{p(k+\varepsilon-\frac{N-1}{2})^2-r(l+\delta-\frac{N-1}{2})^2 }
\{2k+2\varepsilon+1\}\{2l+2\delta+1\} \, \times
\\
 \sum_{s_1+\frac{k+l}{2} = \max(k,l)}^{\min(k+l, N-1)}
 \!\!\!\!\!\!\!\!
\text{\scriptsize$\frac{\{s_1\! +\! \frac{k+\varepsilon+l+\delta}{2}, s_1\! +\! \frac{k+l}{2}\}^2}
{\{s_1\! +\! \frac{-k-\varepsilon+l+\delta}{2},\! s_1\! +\! \frac{-k+l}{2}\}^2
\{s_1\! +\! \frac{k+\varepsilon-l-\delta}{2},\! s_1\! + \!\frac{k-l}{2}\}^2
\{\frac{k+\varepsilon+l+\delta}{2}\!\!-\!\!s_1,\! \frac{k+l}{2}\!\!-\!\!s_1\}^2}$}
.
\end{multline*}
Then $g(k, l, 0, 0)$, $g(k, l, \varepsilon, 0)$ and $g(k, l, 0, \delta)$ are anti-symmetric with respect to $k$ or $l$, we have
\begin{multline*}
\frac{\partial^2}{\partial \varepsilon \partial \delta}
\left.\sum_{k, l=0}^{N-1}
g(k, l, \varepsilon, \delta)\right|_{\varepsilon=\delta=0}
\\
=
\lim_{\varepsilon, \delta\to 0}
\frac{1}{\varepsilon\delta}
\sum_{k, l=0}^{N-1} g(k, l, 0, 0)
-
g(k, l, -\varepsilon, 0) - g(k, l, 0, -\delta)
+
g(k, l, -\varepsilon, -\delta)
\\[-5pt]
=
\lim_{\varepsilon, \delta\to 0}
\frac{1}{\varepsilon\delta}
\sum_{k, l=0}^{N-1} 
g(k, l, -\varepsilon, -\delta).
\end{multline*}
Now look at $f(k-\varepsilon, l-\delta, 0, 0)$ and $g(k, l, -\varepsilon, -\delta)$.  
We have
\begin{multline*}
f(k-\varepsilon, l-\delta, 0, 0) = g(k, l, -\varepsilon, -\delta) = 
\\
q^{p(k-\varepsilon-\frac{N-1}{2})^2-r(l-\delta-\frac{N-1}{2})^2 - (p-r)(\frac{\varepsilon^2+\delta^2}{2}) }
\{2k-2\varepsilon+1\}\{2l-2\delta+1\} \, \times
\\
 \sum_{s_1+\frac{k+l}{2} = \max(k,l)}^{\min(k+l, N-1)}
 \!\!\!\!\!\!\!\!\!\!\!\!
\text{\scriptsize$\frac{\{s_1 + \frac{k-\varepsilon+l-\delta}{2}, s_1 + \frac{k+l}{2}\}^2}
{\{s_1\! +\! \frac{-k+\varepsilon+l-\delta}{2},\! s_1\! +\! \frac{-k+l}{2}\}^2
\{s_1\! +\! \frac{k-\varepsilon-l+\delta}{2}, \!s_1\! +\! \frac{k-l}{2}\}^2
\{\frac{k-\varepsilon+l-\delta}{2}\!-\!s_1, \!\frac{k+l}{2}\!-\!s_1\}^2}$}.  
\end{multline*}
Therefore, 
\[
\left.\frac{\partial^2}{\partial \varepsilon \partial \delta}
\sum_{k, l=0}^{N-1}
f(k, l, \varepsilon, \delta)\right|_{\varepsilon=\delta=0}
=
\left.\frac{\partial^2}{\partial \varepsilon \partial \delta}
\sum_{k, l=0}^{N-1}
g(k, l, \varepsilon, \delta)\right|_{\varepsilon=\delta=0}
\]
and we have
\[
\ADO_N(D_{p, r})
=
\left.\frac{N^2 \, q^{(p-r)\frac{(N-1)^2}{4}}}{-16\pi^2}
\frac{\partial^2}{\partial \varepsilon \partial \delta}
\sum_{k, l=0}^{N-1}
g(k, l, \varepsilon, \delta)\right|_{\varepsilon=\delta=0}.
\]
In $g(k, l, \varepsilon, \delta)$, $k$ and $\varepsilon$ appear as $k + \varepsilon$, and $l$ and $\delta$ appear as $l+\delta$, by putting $x = k+\varepsilon$, $y = l + \delta$, we get
\begin{multline*}
J_{N-1}(D_{p, r}) = 
\\
-\frac{N^2 \, q^{(p-r)\frac{(N-1)^2}{4}}}{16\pi^2}
\frac{\partial^2}{\partial x \partial y}
\sum_{k, l=0}^{N-1}
q^{p(x-\frac{N-1}{2})^2-r(y-\frac{N-1}{2})^2 }
\{2x+1\}\{2y+1\} \, \times
\\
\left.
 \sum_{s_1+\frac{k+l}{2} = \max(k,l)}^{\min(k+l, N-1)}
 \!\!\!\!\!\!
\text{\scriptsize$\displaystyle\frac{\{s_1 + \frac{x+y}{2}, s_1 + \frac{k+l}{2}\}^2}
{\{s_1 + \frac{-x+y}{2}, s_1 + \frac{-k+l}{2}\}^2
\{s_1 + \frac{x-y}{2}, s_1 + \frac{k-l}{2}\}^2
\{\frac{x+y}{2}-s_1, \frac{k+l}{2}-s_1\}^2}$}\right|_{\text{\scriptsize$\begin{matrix}
x\! =\! k\\[-3pt]
 y\! =\! l\end{matrix}$}}\!\!\!.
\end{multline*}
By replacing $s_1$ by $s-\frac{x+y}{2}$, we get \eqref{eq:double}.    
\par
The case for even $p$ and odd $r$ is computed as the limit $\varepsilon, \delta \to 0$ of the  colored knotted graph in Figure \ref{fig:doubleeo}.  
\begin{figure}[htb]
\[
\includegraphics[scale=0.9, alt={double twist knot}]{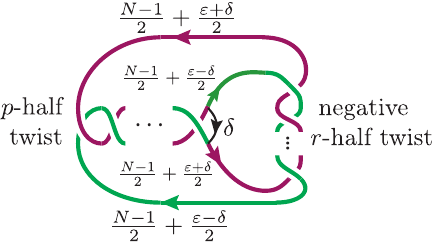}
\]
\caption{The knotted graph to compute $\ADO_{N}(D_{p, r})$ for even $p$ and odd $r$.}
\label{fig:doubleeo}
\end{figure}
\begin{multline*}
\ADO_N(D_{p,r})
=
\lim_{\varepsilon, \delta\to 0}
\ADO_N\left(
\begin{matrix}
\includegraphics[scale=0.7]{doubletwistADOeo}
\end{matrix}
\right)
\\
=
q^{(p-r)\frac{(N-1)^2}{4}}
\, \times
\\
\lim_{\varepsilon, \delta\to 0}
\sum_{k, l=0}^{N-1}
q^{p(k+\varepsilon-\frac{N-1}{2})^2 - p(\frac{\varepsilon+\delta}{2})^2 - p(\frac{\varepsilon-\delta}{2})^2 }
q^{-r(l+\delta-\frac{N-1}{2})^2 + r(\frac{\varepsilon+\delta}{2})^2 + r(\frac{\varepsilon-\delta}{2})^2}
\, \times
\\
\left[\begin{matrix} 2k+2\varepsilon+N\\2k+2\varepsilon+1\end{matrix}\right]^{-1}
\left[\begin{matrix} 2l+2\delta+N\\2l+2\delta+1\end{matrix}\right]^{-1}
\left\{\begin{matrix} 
\frac{N-1}{2}-\frac{\varepsilon+\delta}{2}  &  \frac{N-1}{2} +\frac{\varepsilon-\delta}{2}  & -\delta 
\\[4pt]
\frac{N-1}{2}+ \frac{\varepsilon+\delta}{2}  &   \frac{N-1}{2}+ \frac{\varepsilon-\delta}{2} &  k+\varepsilon 
\end{matrix}\right\}_q
\, \times
\\ \hfill
\ADO_N\left(
\begin{matrix}
\includegraphics[scale=0.7]{doubletwistADOee1}
\end{matrix}
\right)
\\
\underset{\eqref{eq:ADO6j7}, \eqref{eq:ADO6j8}}{=}
q^{(p-r)\frac{(N-1)^2}{4}}
\lim_{\varepsilon, \delta\to 0}
\sum_{k, l=0}^{N-1}
q^{p(k+\varepsilon-\frac{N-1}{2})^2-r(l+\delta-\frac{N-1}{2})^2 - (p-r)(\frac{\varepsilon^2+\delta^2}{2}) }
\, \times \hfill
\\
\frac{\{N-1\}!^2\{2k+2\varepsilon+1\}\{2l+2\delta+1\}}
{\{2k+2\varepsilon+N, N\}\{2l+2\delta+N, N\}}
\frac{\{k+\varepsilon-\delta, k\}\{N-1+\varepsilon+\delta\}}{\{k+\varepsilon+\delta, k\}\{N-1\}!}
\, \times
\\ \hfill
\ADO_N\!\!\left(
\begin{matrix}\!\!
\includegraphics[scale=0.65, alt={double twist link}]{doubletwistADOee1}
\end{matrix}\!
\right).
\end{multline*}
Then the rest of the computation to get \eqref{eq:double} is almost the same as even $p$, $r$ case.  
\par
The case for odd $p$ and $r$ is computed by the limit $\varepsilon, \delta \to 0$ of the ADO invariant of the colored knotted graph in Figure \ref{fig:doubleoo}.  
\begin{figure}[htb]
\[
\includegraphics[scale=0.9]{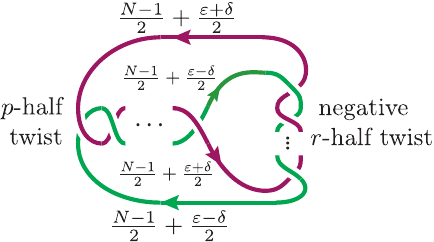}
\]
\caption{The knotted graph to compute $\ADO_{N}(D_{p, r})$ for odd $p$ and $r$.}
\label{fig:doubleoo}
\end{figure}
\begin{multline*}
\ADO_N(D_{p,r})
=
\lim_{\varepsilon, \delta\to 0}
\ADO_N\left(
\begin{matrix}
\includegraphics[scale=0.7]{doubletwistADOoo}
\end{matrix}
\right)
=
\\
q^{(p-r)\frac{(N-1)^2}{4}}\!\!
\lim_{\varepsilon, \delta\to 0}
\sum_{k, l=0}^{N-1}
q^{p(k+\varepsilon-\frac{N-1}{2})^2 - p(\frac{\varepsilon+\delta}{2})^2 - p(\frac{\varepsilon-\delta}{2})^2 }
\,\times \hfill
\\
\hfill
q^{-r(l+\delta-\frac{N-1}{2})^2 + r(\frac{\varepsilon+\delta}{2})^2 + r(\frac{\varepsilon-\delta}{2})^2}
\times\hfill
\\
\hfill
\left[\begin{matrix} 2k+2\varepsilon+N\\2k+2\varepsilon+1\end{matrix}\right]^{-1}
\left[\begin{matrix} 2l+2\delta+N\\2l+2\delta+1\end{matrix}\right]^{-1}
\ADO_N\left(
\begin{matrix}
\includegraphics[scale=0.7]{doubletwistADOee1}
\end{matrix}
\right)
\\
\underset{\eqref{eq:ADO6j7}, \eqref{eq:ADO6j8}}{=}
q^{(p-r)\frac{(N-1)^2}{4}}
\lim_{\varepsilon, \delta\to 0}
\sum_{k, l=0}^{N-1}
q^{p(k+\varepsilon-\frac{N-1}{2})^2-r(l+\delta-\frac{N-1}{2})^2 - (p-r)(\frac{\varepsilon^2+\delta^2}{2}) }
\, \times 
\\
\frac{\{N-1\}!^2\{2k+2\varepsilon+1\}\{2l+2\delta+1\}}
{\{2k+2\varepsilon+N, N\}\{2l+2\delta+N, N\}}
\ADO_N\left(
\begin{matrix}
\includegraphics[scale=0.7]{doubletwistADOee1}
\end{matrix}
\right).
\end{multline*}
Then the rest of the computation to get \eqref{eq:double} is almost same as even $p$, $r$ case.  
\end{proof}
%
%
%
%
%
%\par
%Let $D_{p, r}^{(0)}$ be the $(p, r)$ double twist knot with the writhe $0$. 
%The writhe of $D_{p, r}$ is given as follows.  
%\[
%\Wr(D_{p,r}) = \begin{cases}
%-p+r & \text{if $p$ and $r$ are both even,} \\
%p+r & \text{if $p$ is even and $r$ is odd, } \\
%-p-r & \text{if $p$ is odd and $r$ is even,} \\
%\pm(p+r) & \text{otherwise. (In this case, $D_{p, r}$ is a link)}
%\end{cases}
%\]
%
%
The twist knot $T_p$ is equal to $D_{p, 2}$, so \eqref{eq:double} also gives a formula for $J_{N-1}(T_p)$.   
%
% \begin{multline}
%J_{N-1}(T_{p}^{(0)}) 
%=
%-\frac{N^2q^{t\frac{(N-1)^4}{4}}}{16 \pi^2}
%\, \times
%\\
%\sum_{k,l=0}^{N-1}
%\frac{\partial}{\partial\alpha}
%q^{p(k - \frac{N-1}{2})^2-2(l-\frac{N-1}{2})^2}
%\{2N\alpha\}\{2N\beta\}
%%\, \times \hfill
%%\\
%%\hfill
%\left.
%C_N e^{\frac{N}{2\pi i} \left(\zeta(\alpha, \beta, \gamma_0(\alpha, \beta)) + O\left(
%\frac{1}{N}\right)\right)}
%\right|_{\text{$\scriptstyle\begin{matrix}\alpha=\frac{2k+1}{2N}\\[0pt]
%\beta=\frac{2l+1}{2N}\end{matrix}$}}, 
%\label{eq:twist}
%\end{multline}
%%
%where $t = -2p+4$ for odd $p$ and $t = -2p$ for even $p$.  
%
%\setcounter{section}{2}
\section{Asymptotics}
Here we investigate the asymptotic behavior of the colored Jones invariant for large $N$.  
We also reformulate the sum over the parameter $s$ inside the quantum $6j$ symbol.  
\subsection{Quantum dilogarithm function}
As a continuous version of the quantum factorial $\{n\}!$, we introduce Fateev's quantum dilogarithm function $\varphi_N(x)$, which is the analytic continuation of the following function defined for $0 < x < 1$.    
\[
\varphi_N(x) = 
\int_{-\infty}^\infty \frac{e^{(2x-1)t}\, dt}{4t \sinh t \sinh(t/N)}.
\]
It is shown in \cite{F} that
\[
\varphi_N({\alpha+\frac{1}{2N}}) - \varphi_N({\alpha-\frac{1}{2N}})
=
-\log (1-e^{2 \pi i\alpha}).  
\]
This implies that
\begin{multline}
\{x, n\} = \{x\}\{x-1\}\cdots\{x-n+1\}
=
\\
(-1)^n q^{-\frac{(2x-n+1)n}{2}} (1-q^{2x})(1-q^{2x-2}) \cdots (1-q^{2x-2n+2})
=
\\
(-1)^n q^{-\frac{(2x-n+1)n}{2}} e^{\varphi(\frac{2x-2n+1}{2N}) - \varphi(\frac{2x+1}{2N})}.
\label{eq:qdilog}
\end{multline}
For fixed any sufficient small $\delta>$ and any $M > 0$, 
\[
\varphi_N(t)=\frac{N}{2\pi i}\Li(e^{2\pi i t})+ O(\frac{1}{N})
\]  
in the domain
\[
\{t \in \mathbb{C} \mid 
\delta < \mathrm{Re}\, t < 1-\delta, \ 
\left|\mathrm{Im}\,t\right| <M\}
\]
by Proposition A.1 of \cite{O1}.  
%Moreover, it is shown that $ \frac{2 \pi i}{N}\varphi_N(t)$ uniformly converges to  $\Li(e^{2\pi i t})$ for $t \in [0, 1]\subset \mathbb{R}$. 
It is also shown by Lemma A of \cite{O1} that
\[
\varphi_N(\frac{1}{2N}) = \frac{N}{2\pi i}\frac{\pi^2}{6} + O(\log N), 
\quad
\varphi_N(1-\frac{1}{2N}) = \frac{N}{2\pi i}\frac{\pi^2}{6} + O(\log N).
\]
\subsection{Reformulation of the colored Jones polynomials}
Here we reformulate the colored Jones polynomials \eqref{eq:whiteheadADO} and \eqref{eq:double} by using the dilogarithm function.  
We first reformulate $J_{N-1}(W_p)$.  
Let
\begin{multline*}
\zeta_N(x, k, l, s) =
\\
 \frac{ \{s, s- \frac{x-k}{2}\}^2}
{\{s-x, s- \frac{x+k}{2}\}^2
\{s-l, s-l- \frac{x-k}{2}\}^2
\{x+l-s, \frac{x+k}{2}+l-s_1\}^2}.  
\end{multline*}
Then
\begin{multline*}
\left.
\frac{d}{dx}
\zeta_N(x, k, l, s) 
\right|_{x=k}
=
%\\
\frac{d}{dx}
q^{(4s^2 -8(l+x)s + 3x^2 + 2kx -k^2 + 4lx + 4l^2)/2}\, \times
\hfill\\
\exp\Big( - 2\varphi_N(\tfrac{2s+1}{2N})+2\varphi_N(\tfrac{2s-2x+1}{2N}) +
2\varphi_N(\tfrac{2s-2l+1}{2N})
\hfill\\
\left.
+2\varphi_N(\tfrac{2x+2l-2s+1}{2N}) 
- 2\varphi_N(\tfrac{x-k+1}{2N}) -2 \varphi_N(\tfrac{k-x+1}{2N})\Big)\right|_{x=k}
\\
=
\frac{d}{dx}
q^{2(s^2 -2(l+x)s +x^2  + lx + l^2)}
\exp\Big( - 2\varphi_N(\tfrac{2s+1}{2N})+2\varphi_N(\tfrac{2s-2x+1}{2N}) +
\hfill\\
\left.
2\varphi_N(\tfrac{2s-2l+1}{2N})+2\varphi_N(\tfrac{2x+2l-2s+1}{2N})
-4 \varphi_N(\tfrac{1}{2N})\Big)\right|_{x=k}
\end{multline*}
since
\[
\left.
\frac{d}{dx}
\Big(
\varphi_N(\tfrac{x-k+1}{2N}) + \varphi_N(\tfrac{k-x+1}{2N})
\Big)
\right|_{x=k} = 0  
, 
\]
\[\left.2kx - k^2\right|_{x=k} = 
\left.x^2\right|_{x=k} = k^2,\qquad
\left.\frac{d}{dx}(2kx - k^2)\right|_{x=k} = 
\left.\frac{d}{dx}x^2\right|_{x=k}
= 2k.
\]  
Let 
\begin{multline*}
\xi_N(x,  l, s)
=
q^{2(s^2 -2(l+x)s +x^2  + lx + l^2)}%\,\times
%\hfill\\
\exp\Big( - 2\varphi_N(\tfrac{2s+1}{2N})+
\\
\hfill2\varphi_N(\tfrac{2s-2x+1}{2N}) +
2\varphi_N(\tfrac{2s-2l+1}{2N})+2\varphi_N(\tfrac{2x+2l-2s+1}{2N})
-4 \varphi_N(\tfrac{1}{2N})\Big).
\end{multline*}
By using the relation between $\frac{2\pi i}{N}\varphi_N(t)$ and $\Li(e^{2\pi i t})$, we have
\begin{multline*}
\xi_N(x,  l, s)
=
E_N(x, l, s)
q^{2(s^2 -2(l+x)s + x^2 + lx + l^2)}%\,\times
%\hfill\\
\exp\Big(\frac{N}{2\pi i}
\big(
- 2\Li(q^{2s + 1})
\\ \hfill
+2\Li(q^{2s -2x+1})
+
2\Li(q^{2s-2l+1})
+
2\Li(q^{2x+2l - 2s +1})
-
\frac{\pi^2}{3} %+O(\frac{\log N}{N})
\big)\Big)
\end{multline*}
where $E_N(x, l, s)$ is a function which grows at most a polynomially  with respect to $N$.  
Therefore, 
\begin{multline}
J_{N-1}(W_p)
=
\\
N \frac{q^{p\frac{(N-1)^2}{4}}
}{4\pi i}
 \sum_{k,l=0}^{N-1}\!
 \frac{d}{dx}\!\!
 \left(\!
%q^{p(k+\varepsilon-\frac{N-1}{2})^2-p\varepsilon^2}
q^{p(x-\frac{N-1}{2})^2}
\{2x+1\}
\left.
\sum_{s - \frac{x-k}{2} = \max(k,l)}^{\min(k+l, N-1)}
\xi_N(x,  l, s)\!\right)\!\right|_{x=k}.
\label{eq:whitehead1}
\end{multline}
\par
Similarly, we have
\begin{multline}
J_{N-1}(D_{p,r})
=
N \frac{q^{p\frac{(N-1)^2}{4}-r\frac{(N-1)^2}{4}}
}{4\pi i}
\, \times
\\
 \sum_{k,l=0}^{N-1}\!
 \frac{\partial^2}{\partial x \partial y}\!\!
 \left(\!\!
%q^{p(k+\varepsilon-\frac{N-1}{2})^2-p\varepsilon^2}
q^{p(x-\frac{N-1}{2})^2-r(y-\frac{N-1}{2})^2}
\{2x+1\}\{2y+1\}
\!\!\!\!\!\!\!\!\!\!\!\!\!\!\!\!\!
\left.
\sum_{s - \frac{x-k+y-l}{2} = \max(k,l)}^{\min(k+l, N-1)}
\!\!\!\!\!\!\!\!\!\!\!\!\!\!\!\!
\xi_N(x, y, s)\!\!\right)\!\right|_{\text{\scriptsize
$\begin{matrix}x=k\\[-3pt]
y=l\end{matrix}$}}.
\label{eq:double1}
\end{multline}
%
%where
%%
%\begin{multline*}
%\zeta_N(x, y, k, l, s)
%=
%F_N(x, y, k, l, s)
%q^{(4 s^2 - 8(x+y)s+3(x+y)^2-(k-l)(k-l-2x+2y) )/2}
%\\ \hfill
%\exp\Big(\frac{N}{2\pi i}
%\big(
%- 2\Li(q^{2s + 1})
%+2\Li(q^{2s -2x+1})
%+
%2\Li(q^{2s-2l+1})
%+
%2\Li(q^{2x+2l - 2s +1})
%-
%\frac{\pi^2}{3} %+O(\frac{\log N}{N})
%\big)\Big)
%\end{multline*}
%%
%with the function $F_N(x, y, k, l, s)$ which grows at most polynomially with respect to $N$.  
%Note that
%\begin{equation}
%\xi_N(x, k, l, s) = 
%\zeta_N(x, l, k, l, s).  
%\label{eq:xizeta}
%\end{equation}
%
%\par
%
%
%
\subsection{Saddle points}
%
%Let  
%$N \alpha = x + \tfrac{1}{2}$, $N \eta=y+\frac{1}{2}$, $N\alpha = k+\tfrac{1}{2}$, $N\beta = l + \tfrac{1}{2}$, $N\gamma = s + \tfrac{1}{2}$ and
%
%\begin{multline*}
%\psi_N(x, y, s) =
%q^{2( s^2 - 2(x+y)s+x^2 + xy + y^2)}
% \, \times
%\\
%\exp\Big(\frac{N}{2\pi i}
%\big(
%- 2\Li(q^{2s + 1})
%+2\Li(q^{2s -2x+1})
%+
%2\Li(q^{2s-2y+1})
%+
%2\Li(q^{2x+2y- 2s +1})
%-
%\frac{\pi^2}{3} %+O(\frac{\log N}{N})
%\big)\Big)
%\end{multline*}
%%
We investigate the sum 
$
 \sum_s
\xi_N(x, y, s) 
$.
Since the function $\xi_N(x, y, s)$ is non-negative for each $s$ and there exists  $s_0$ such that $\xi_N(x, y, s_0)$ is the maximal among $\xi_N(x,y, s)$.  
In this case, there is some number $C_N$ satisfying $1 \leq C_N\leq N$ satisfying
\begin{equation}
 \sum_{s}
\xi_N(x, y, s)
=
C_N \, \xi_N(x, y, s_0),   
\label{eq:saddles}
\end{equation}
where $s_0$ satisfies
\[
\frac{\partial}{\partial s}\xi_N(x, y, s) = 0.  
\]
\par
Now we compute the maximal point $s_0$ of $\xi_N(x, y, s)$.   
Let $N \alpha = x + \frac{1}{2}$, $N \eta = y + \frac{1}{2}$, $N \gamma = s + \frac{1}{2}$
and
$u=e^{2 \pi i\alpha}$, $v=e^{2 \pi i \eta}$, $w=e^{2 \pi i \gamma}$.  
Then the equation for obtaining the maximal point is 
\[
\log \frac{(w-1)^2 (w- uv)^2}{(w-u)^2(w-v)^2} = 0.   
\]
To solve it,  we first solve
\begin{multline*}
(w-1)^2 (w- u v)^2 - (w-u)^2(w-v)^2 = 
\\
\big((w-1)(w-uv) - (w-u)(w-v)\big)\big((w-1)(w-u v) + (w-u)(w-v)\big) = 
\\
(-1-uv)w\big(2w^2 - (u +1)(v+1) z + 2 uv\big)=0.
\end{multline*}
The solution is  $w = \frac{(u +1)(v+1) \pm \sqrt{(u +1)^2(v+1)^2 - 16 u v}}{4}$.  
The solution corresponding to $q^{2s_0+1}$ is 
\begin{equation}
w_0 = \frac{(u +1)(v+1) - \sqrt{(u +1)^2(v+1)^2 - 16 u v}}{4}, 
\end{equation}
and
$s_0 = -\frac{1}{2}+\frac{N}{2 \pi i} \log w_0$.    
%
%By using this, the colored Jones invariant of $W_p$  is reformulated as follows.  
%%
%\begin{equation}
%J_{N-1}(W_p)
%=
%N \frac{q^{p\frac{(N-1)^2}{4}}
%}{4\pi i}
% \frac{d}{d x}
% \left.
%\sum_{k=0}^{N-1}
%\{2x+1\}\, q^{p(x-\frac{N-1}{2})^2}
%C_N \, D_N\, \zeta_N(x, l, k, l_{\max}^{(N)}, s_{\max}^{(N)})\right|_{x=k}. 
%\label{eq:whiteheadNEW}
%\end{equation}
%
%
%
%
%
\section{Neumann-Zagier function}
Here we recall some properties of the Neumann-Zagier function developed in \cite{NZ} and \cite{Yos}.  
The relation between this function and the potential function coming from the quantum invariant is observed in  \cite{Yok}. 
\subsection{Neumann-Zagier potential function}
To prove the volume conjecture for double twist knots, we extend the argument in \cite{Yok} to links.  
Let $L = L_1\cup L_2 \cup \cdots \cup L_k$ be a link with connected components.  
Let $\rho$ be an  $\mathrm{SL}(2, \mathbb{C})$ representation 
of $\pi_1(S^3\setminus L)$, $\mu_i$, 
$\lambda_i \in  \pi_1(S^3\setminus L)$ are
elements corresponding to the meridian and longitude of $L_i$, and 
$\xi_i$, $\eta_i$ are  the eigenvalues of $\rho(\mu_i)$ and $\rho(\lambda_i)$ respectively.  
Then there is an analytic function $f(\xi_1, \cdots, \xi_k)$
satisfying the following differential equation.  
\[
\frac{\partial}{\partial \xi_i} f(\xi_1, \cdots, \xi_k) 
=
-2 \log \eta_i. \quad (i = 1, 2, \cdots, l)
\]
Now we assume that $f(1, \cdots, 1)=0$.  
For an integer $l$ satisfying $0 \leq l \leq k$ and rational numbers $p_i/q_i$ for $i = 1, 2, \cdots, l$, 
let $M$ be a three manifold obtained by rational $p_i/q_i$ surgeries along  $L_1$, $L_2$, $\cdots$, $L_l$, and $\rho$ be the representation of $\pi_1(S^3\setminus L)$ corresponding to this surgery.   
Then 
\[
2 p_i \log\xi_i + 2 q_i \log\eta_i = 2 \pi\sqrt{-1}. \quad (i = 1, 2, \cdots, l)
\]  
This function corresponds to the deformation of the hyperbolic structure of the complement of $L$.  
\subsection{Complex volume}
Let $M$ be the manifold obtained by this surgery.  
Assume that $M$ is a hyperbolic manifold.  
Then the complex volume of $M$ is given by $f(\xi_1, \cdots, \xi_k)$ with a small modification.  
The complex volume of $M$ is
\[
\mathrm{Vol}(M) +\sqrt{-1}\,\mathrm{CS}(M)
\]
where $\mathrm{Vol}(M)$ be the hyperbolic volume and 
$\mathrm{CS}$ is the Chern-Simons invariant of $M$.  
Let $\gamma_i$ be the core geodesic of $L_i$ for this surgery, them
\[
\gamma_i = 
2(r_i \log\xi_i + s_i \log\eta_i)
\]
where $r_i$, $s_i$ are integers satisfying $p_i \,s_i - r_i \,q_i = 1$.  
\begin{thm} {\rm(\cite[Theorem 2]{Yos})}
The complex volume of $M$ is given by
\begin{equation}
\mathrm{Vol}(M) + \sqrt{-1}\, \mathrm{CS}(M)
=
\frac{1}{i}\left(
f(\xi_1, \cdots, \xi_k) + \sum_{i=1}^k \log \xi_i \, \log\eta_i - \frac{\pi i}{2}\sum_{i=1}^k \gamma_i
\right).  
\label{eq:NZ}
\end{equation}
\end{thm}
\section{Deformation of the integral region}
For $D_{6,2}$, $D_{5,3}$, $D_{4, 4}$, $D_{6,-3}$ and $D_{5, -4}$, we already sow that the integral region $[0,1]^2$ can be deformed to another region passing through the saddle point.  
Here we see that the integral region for other cases also can be deformed so that it passes through the saddle point.  
Let $f_{p,r}(\alpha, \beta) = -4\pi^2 \alpha - 4 \pi^2 \beta +\Phi_{D_{p, r}}(\alpha, \beta)$.  
Then $f_{p,r}(\alpha, \beta)$ is continuous with respect to $p$, $r$ for almost all $\alpha$ and $\beta$, the integral region for $D_{4,4}$ passing through the saddle point is deformed continuously with respect to $p$ and $r$.  
However, the analytic continuation of $f(p, r, \alpha, \beta)$ is a multi-variable function and it is not clear that the saddle point for $D_{4, 4}$ is moved to the saddle point of $D_{p,r}$ corresponds to the hyperbolic volume since $f_{p,r}(\alpha, \beta)$ has many singular points.  
Now we focus on the saddle point of $D_{3,3}$.  
Let 
\begin{multline*}
E = \{(\alpha, \beta) \mid 0.45 \leq \operatorname{Re} \alpha  \leq 0.88,\ 
 -0.12 \leq \operatorname{Im} \alpha \leq 0.01, 
\\
0.12 \leq \operatorname{Re} \beta  \leq 0.55,\ 
 -0.12 \leq \operatorname{Im} \beta \leq 0.01\}. 
\end{multline*}
Then $\alpha_0$, $\beta_0$ for $D_{3,3}$ is contained in this region.  
\begin{prop}
The function $f_{p,r}(\alpha, \beta)$ has only one singular point in $E$ for $p, r\geq 3$.  
\end{prop}
This proposition  implies that we are able to deform the integral region for $f_{3, 3}(\alpha, \beta)$ passing through the saddle point to that for $f_{p,r}(\alpha, \beta)$ passing through the saddle point for $D_{p, r}$ if $p, r \geq 3$.  
To show the proposition, we show the following.  
\begin{lem}
For fixed $p$, $r$ with $p, r \geq 3$, the gradient vector of the function $f_{p, r}$, which is $\left(\frac{\partial }{\partial\alpha}f_{p,r}(\alpha, \beta), \frac{\partial }{\partial\beta}f_{p,r}(\alpha, \beta)\right)$, is not vanish on the boundary $\partial E$.  
\end{lem}
\begin{proof}[Proof of Proposition E.1]
The previous lemma means that the index of the gradient vector  $\left(\frac{\partial}{\partial\alpha}f_{p,r}(\alpha, \beta), \frac{\partial }{\partial\beta}f_{p,r}(\alpha, \beta)\right)$ on $\partial E$ is stable for any $p, r \geq 3$.  
The actual computation shows that  $f_{3,3}(\alpha, \beta)$ has only one singular point in $E$, $f_{p,r}(\alpha, \beta)$ also has only one singular point in $E$ since the index of $\partial E$ is unchanged.  
\end{proof}
\begin{proof}[Proof of Lemma E.1.]
Let 
\begin{multline*}
\partial_1 E = \{(\alpha, \beta)\in E \mid \operatorname{Re}\alpha = 0.45\}
\cup \{(\alpha, \beta)\in E \mid \operatorname{Re}\alpha = 0.88\}
\cup
\\
\{(\alpha, \beta)\in E \mid \operatorname{Im}\alpha = -0.12\}
\cup\{(\alpha, \beta)\in E \mid \operatorname{Im}\alpha = 0.01\}
\end{multline*}
and
\begin{multline*}
\partial_2 E = \{(\alpha, \beta)\in E \mid \operatorname{Re}\beta = 0.12\}
\cup \{(\alpha, \beta)\in E \mid \operatorname{Re}\beta = 0.55\}
\cup
\\
\{(\alpha, \beta)\in E \mid \operatorname{Im}\beta = -0.12\}
\cup\{(\alpha, \beta)\in E \mid \operatorname{Im}\beta = 0.01\}.
\end{multline*}
Then $\partial_1 E$ and $\partial_2 E$ are both isomorphic to the solid torus and $\partial_1 E \cup \partial_2 E = \partial E$.  
We show that $\frac{\partial}{\partial\alpha} f_{p,r}( \alpha, \beta)$ does not vanish on $\partial_1 E$.  
The contour graph of $\operatorname{Im} f_{3,3}( \alpha, \beta)$ and $\operatorname{Im} \frac{1}{2} \big(2 \pi i (\alpha-\frac{1}{2})\big)^2$  on $E$ is given as in Figure \ref{fig:contourf} as a two dimensional movie picture.  
\begin{figure}[htb]
Contours of $\operatorname{Im}f_{p,r}(x + y i, \beta)$ and 
$\operatorname{Im} \frac{1}{2} \big(2 \pi i (x + y i -\frac{1}{2})\big)^2$.
\begin{scriptsize}
\[
\begin{matrix}
yi \uparrow &
\begin{matrix}
\includegraphics[scale=0.9]{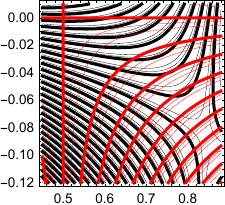}
\end{matrix}
&
\begin{matrix}
\includegraphics[scale=0.9]{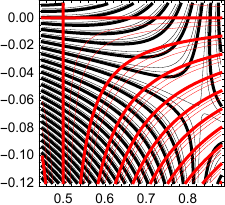}
\end{matrix}
&
\begin{matrix}
\includegraphics[scale=0.9]{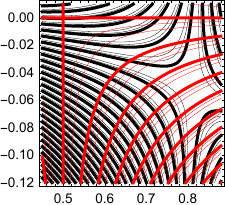}
\end{matrix}
\\
&\beta = 0.12-0.12i &
\beta=0.12-0.06i &
\beta = 0.12 + 0.01i
\\[14pt]
yi \uparrow &
\begin{matrix}
\includegraphics[scale=0.9]{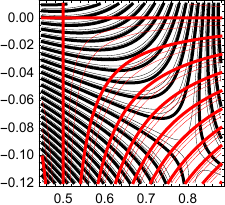}
\end{matrix}
&
\begin{matrix}
\includegraphics[scale=0.9]{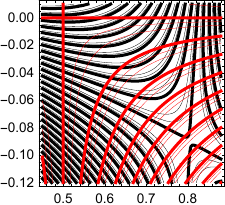}
\end{matrix}
&
\begin{matrix}
\includegraphics[scale=0.9]{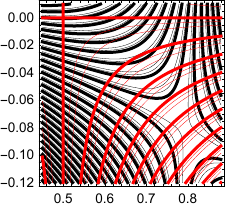}
\end{matrix}
\\
&\beta = 0.25-0.12i &
\beta=0.25-0.06i &
\beta = 0.25 + 0.01i
\\[14pt]
yi \uparrow &
\begin{matrix}
\includegraphics[scale=0.9]{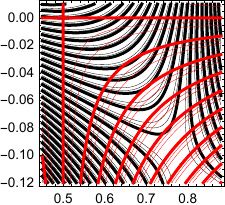}
\end{matrix}
&
\begin{matrix}
\includegraphics[scale=0.9]{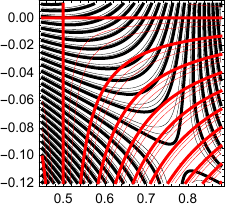}
\end{matrix}
&
\begin{matrix}
\includegraphics[scale=0.9]{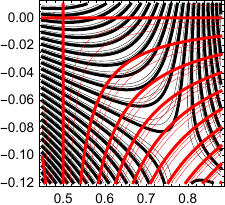}
\end{matrix}
\\
&\beta = 0.5-0.12i &
\beta=0.5-0.06i &
\beta = 0.5 + 0.01i
\\[14pt]
yi \uparrow &
\begin{matrix}
\includegraphics[scale=0.9]{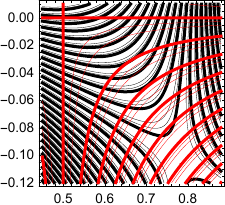}
\end{matrix}
&
\begin{matrix}
\includegraphics[scale=0.9]{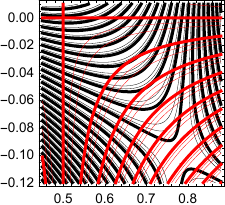}
\end{matrix}
&
\begin{matrix}
\includegraphics[scale=0.9]{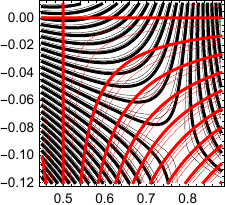}
\end{matrix}
\\
&\beta = 0.55-0.12i &
\beta=0.55-0.06i &
\beta = 0.55 + 0.01i
\end{matrix}
\]
\end{scriptsize}
\caption{The contour graph of the imaginary parts $\operatorname{Im} f_{3,3}(\alpha, \beta)$ and $\operatorname{Im}  \frac{1}{2} \big(2 \pi i (\alpha-\frac{1}{2})\big)^2$  on $E$.  Black lines are contours of $\operatorname{Im} f_{3,3}(\alpha, \beta)$ and red lines are contours of  $\operatorname{Im}  \frac{1}{2} \big(2 \pi i (\alpha-\frac{1}{2})\big)^2$.
The contour levels are $0.2 k$ for thick lines and $0.2 k+0.04$, $0.2 k + 0.08$ for thin lines where $k \in \mathbb{Z}$.}
\label{fig:contourf}
\end{figure}
For each graph, the gradients of the black lines at the boundary square are non-zero.  
Moreover, the gradient vectors of black lines and the red lines at any point of the boundary square are not oriented to the opposite direction, the differential $\frac{\partial }{\partial\alpha}f_{p,r}(\alpha, \beta)$ is not zero on $\partial_1E$ since 
$\frac{\partial}{\partial\alpha} f_{p,r}(\alpha, \beta) = \frac{\partial}{\partial\alpha} f_{3,3}(\alpha, \beta) + \frac{(p-3)}{2}\frac{\partial}{\partial\alpha}\big(2 \pi i (\alpha-\frac{1}{2})\big)^2$.  
\par
By using the similar argument, we see that $\frac{\partial}{\partial\beta}f_{p,r}(\alpha, \beta)$ is not zero on $\partial_2 E$.  
Therefore, $\left(\frac{\partial}{\partial\alpha}f_{p,r}(\alpha, \beta), \frac{\partial}{\partial\beta}f_{p,r}(\alpha, \beta)\right)$ is not zero on $E$.  
\end{proof}
By using similar argument, we can prove that there is only one singular point of $f_{p,-r}(\alpha, \beta)$ for $p,r \geq 3$ in the region
\begin{multline*}
E'
=
\{(\alpha, \beta) \mid 0.45 \leq \operatorname{Re} \alpha  \leq 0.88,\ 
 -0.12 \leq \operatorname{Im} \alpha \leq 0.01, 
\\
0.45 \leq \operatorname{Re} \beta  \leq 0.88,\ 
 -0.12 \leq \operatorname{Im} \beta \leq 0.01\},\end{multline*}
and there is only one singular point of $f(p, 2, \alpha, \beta)$ for $p \geq 6$ in the region
\begin{multline*}
E''
= 
\{(\alpha, \beta) \mid 0.45 \leq \operatorname{Re} \alpha  \leq 0.7,\ 
 -0.04 \leq \operatorname{Im} \alpha \leq 0.01, 
\\
0.12 \leq \operatorname{Re} \beta  \leq 0.55,\ 
 -0.18 \leq \operatorname{Im} \beta \leq 0.01\}. 
\end{multline*}
%
%
%\pagebreak
%


\begin{thebibliography}{00}

\bibitem{ADO}
Y. Akutsu, T.  Deguchi, T. Ohtsuki, 
 {\it Invariants of colored links,}
J. Knot Theory Ramifications {\bf 1} (1992), %no. 2,
161--184.
%doi:10.1142/S0218216592000094.

\bibitem{CZ}
Q. Chen, S. Zhu,
{\it On the asymptotic expansions of various quantum invariants II: the colored Jones polynomial of twist knots at the root of unity $e^{\frac{2\pi\sqrt{-1}}{N+\frac{1}{M}}}$ and $e^{\frac{2\pi\sqrt{-1}}{N}}$,} 
arXiv:2307.13670.

\bibitem{CK}
Y. Cho, H. Kim, 
{\it On the volume formula for hyperbolic tetrahedra,}
Discrete Comput. Geom. {\bf 22} (1999), 347--366.
%doi:10.1007/PL00009465.

%\bibitem{ChM}
%Q. Chen, J. Murakami, 
%Asymptotics of quantum $6j$ symbols.
% J. Differential Geom. {\bf 123}, no. 1 1 - 20. 
% https://doi.org/10.4310/jdg/1679503803
 
% \bibitem{C}
% F. Costantino, $6j$-symbols, hyperbolic structures and the volume conjecture, 
% Geom. Topol. {\bf 11} (2007), 1831--1854, MR2350469, Zbl 1132.57011.

\bibitem{CM}
F. Costantino, J. Murakami, 
{\it On the  $SL(2,\mathbb{C})$  quantum  $6j$-symbols and their relation to the hyperbolic volume,}
Quantum Topol. {\bf 4} (2013), no. 3, 303--351.
%doi:10.4171/QT/41.

\bibitem{F}
L. D. Faddeev, 
{\it Discrete Heisenberg-Weyl group and modular group, }
Lett. Math. Phys. {\bf 34} (1995), no. 3, 249--254. 
%MR 1345554, Zbl 0836.47012,
%doi.:10.1007/BF01872779.

%\bibitem{GL}
%S. Garoufalidis, T. T. Q. Le, 
%{\it On the volume conjecture for small angles,}
%arXiv:math/0502163.

%\bibitem{GR}
%N. Geer, N. Reshetikhin,
%On invariants of graphs related to quantum $sl(2)$ at roots of unity
%Lett. Math. Phys. {\bf 88} (2009), 321--331.

%\bibitem{Ka}
%R. M. Kashaev, 
%The hyperbolic volume of knots from the quantum dilogarithm
%Lett. Math. Phys. 39 (1997), no. 3, 269--275.

%\bibitem{KL}
%L. H. Kauffman, S. L. Lins,
%Temperley-Lieb recoupling theory and invariants of  3 -manifolds,
%Ann. of Math. Stud., {\bf 134}
%Princeton University Press, Princeton, NJ, 1994.

\bibitem{KR}
A.N. Kirillov, N.Yu. Reshetikhin, 
{\it Representations of the algebra $U_q(sl(2))$, $q$-orthogonal polynomials and invariants of links, }
Infinite-dimensional Lie algebras and groups (Luminy-Marseille, 1988), 285--339, Adv. Ser. Math. Phys., {\bf 7}, World Sci. Publ., Teaneck, NJ, 1989. 
%doi:10.1142/9789812798329\_0012.

\bibitem{KM}
A. Kolpakov,  J. Murakami, 
{\it Combinatorial decompositions, Kirillov-Reshetikhin invariants, and the volume conjecture for hyperbolic polyhedra,}
Exp. Math. {\bf 27} (2018), 193--207.
%doi:10.1080/10586458.2016.1242441.

%\bibitem{MeN}
%R. Meyerhoff and W. D. Neumann, 
%An asymptotic formula for the eta invariants of hyperbolic 3-manifolds. 
%Comment. Math. Helv. {\bf 67} (1992), 28--46.

\bibitem{MM}
H. Murakami, J. Murakami, 
{\it The colored Jones polynomials and the simplicial volume of a knot, }
Acta Math. {\bf 186} (2001), 85--104.
%doi:10.1007/BF02392716.

\bibitem{MMOTY}
H. Murakami, J. Murakami, M. Okamoto, T. Takata, Y. Yokota, 
{\it Kashaev's conjecture and the Chern-Simons invariants of knots and links,}
Experiment. Math. {\bf 11} (2002), 427--435.
%doi:10.1080/10586458.2002.10504485.

\bibitem{MYo}
H. Murakami, Y. Yokota,
{\it The colored Jones polynomials of the figure-eight knot and its Dehn surgery spaces,}
J. reine angew. Math. {\bf 607} (2007), 47--68  
%doi:10.1515/CRELLE.2007.045.

\bibitem{MYo1}
H, Murakami, Y. Yokota, 
{\it Volume conjecture for knots,}
SpringerBriefs Math. Phys., {\bf 30},
Springer, Singapore, 2018.

%\bibitem{M}
%J. Murakami,
%Colored Alexander invariants and cone-manifolds.
%Osaka J. Math. {\bf 45} (2008), no. 2, 541--564.

%\bibitem{MN}
%J. Murakami, K. Nagatomo, 
%Logarithmic knot invariants arising from restricted quantum groups.
%Internat. J. Math. {\bf 19} (2008), no. 10, 1203--1213.

\bibitem{MU}
J. Murakmai, A. Ushijima,
{\it A volume formula for hyperbolic tetrahedra in terms of edge lengths,}
J. Geom. {\bf 83} (2005), 153--163.
%doi:10.1007/s00022-005-0010-4.

\bibitem{MY}
J. Murakami, M. Yano,
{\it On the volume of a hyperbolic and spherical tetrahedron,}
Comm. Anal. Geom. {\bf 13} (2005), 379--400.
%doi:10.4310/CAG.2005.v13.n2.a5.

\bibitem{NZ}
W. D. Neumann, D. Zagier, 
{\it Volumes of hyperbolic 3-manifolds,}
Topology {\bf 24} (1985), 307--332.
%doi:10.1016/0040-9383(85)90004-7.

\bibitem{O1}
T. Ohtsuki, 
{\it On the asymptotic expansion of the Kashaev invariant of the $5_2$  knot,}
{Quantum Topol.}
  {\bf 7}   (2016),
%   number={4},
669--735, 
%MR 3593566,  doi:10.4171/QT/83.
   
\bibitem{O2}
T. Ohtsuki, 
{\it On the asymptotic expansions of the Kashaev invariant of hyperbolic knots with seven crossings,}
Internat. J. Math. {\bf 28} (2017), no. 13, 1750096, 143 pp.
%doi:10.1142/S0129167X17500963.

\bibitem{OT}
T. Ohtsuki, T. Takata, 
{\it On the Kashaev invariant and
the twisted Reidemeister torsion of two-bridge knots,}
Geometry and Topology {\bf 19} (2015), 853-952.
%doi:10.2140/gt.2015.19.853.

\bibitem{OY}
T. Ohtsuki, Y. Yokota,  
{\it On the asymptotic expansions of the Kashaev invariant of the knots with 6 crossings,}
Mathematical Proceedings of the Cambridge Philosophical Society {\bf 165} (2018), 287--339. 
%doi:10.1017/S0305004117000494.

\bibitem{U}
A. Ushijima,
{\it A volume formula for generalised hyperbolic tetrahedra,}
Math. Appl. (N.Y.), {\bf 581}
Springer, New York, 2006, 249--265.
%doi:10.1007/0-387-29555-0\_13.

\bibitem{Yok}
Y. Yokota, 
{\it On the potential functions for the hyperbolic structures of a knot complement, }
Geom. Topol. Monogr. {\bf 4} (2002), 303--311.
%doi:10.2140/gtm.2002.4.303.

\bibitem{Yos}
T. Yoshida, 
{\it The $\eta$-invariant of hyperbolic 3-manifolds, }
Invent. Math. {\bf 81} (1985), 473--514. 
%doi:10.1007/BF01388583.

\bibitem{Z}
H. Zheng, 
{\it Proof of the volume conjecture for Whitehead doubles of a family of torus knots,}
 Chinese Annals of Mathematics, Series B {\bf 28}  (2007), 375--388. %arXiv:math/0508138,
%doi:10.1007/s11401-006-0373-3.

\end{thebibliography}
\end{document}